\documentclass[a4paper,10pt,leqno]{article}
\usepackage[T1]{fontenc}
\usepackage{lmodern}
\usepackage[margin=3cm]{geometry}
\usepackage[all,2cell]{xy}
\UseAllTwocells
\usepackage{amsmath}
\usepackage{amssymb}
\usepackage{amsxtra}
\usepackage{amsthm}
\usepackage{mathtools}
\usepackage[shortlabels]{enumitem}
\usepackage[pagebackref,breaklinks]{hyperref}
\renewcommand*\backref[1]{\ifx#1\relax \else Cited on page(s) #1.\fi}
\theoremstyle{definition}
\newtheorem{construction}[subsection]{Construction}
\newtheorem{definition}[subsection]{Definition}
\newtheorem{notation}[subsection]{Notation}
\newtheorem{example}[subsection]{Example}
\newtheorem{remark}[subsection]{Remark}

\theoremstyle{plain}
\newtheorem{lemma}[subsection]{Lemma}
\newtheorem{prop}[subsection]{Proposition}
\newtheorem{theorem}[subsection]{Theorem}
\newtheorem{cor}[subsection]{Corollary}

\DeclareMathOperator{\Spec}{Spec}
\DeclareMathOperator{\Aut}{Aut}

\DeclareMathOperator{\Ker}{Ker}

\DeclareMathOperator{\Img}{Im}
\DeclareMathOperator{\rk}{rk}

\newcommand{\ie}{i.e.\ }
\newcommand{\loccit}{\emph{loc.\ cit.}}

\newcommand{\et}{{\mathrm{et}}}
\newcommand{\cd}{\mathrm{cd}}
\newcommand{\tr}{\mathrm{tr}}

\newcommand{\from}{\leftarrow}
\newcommand{\xto}[2][]{\xrightarrow[#1]{#2}}
\newcommand{\simto}{\xto{\sim}}
\newcommand{\Ar}{\mathrm{Ar}}
\newcommand{\Arm}{\Ar^-}

\newcommand{\Ad}{\mathrm{Ad}}
\newcommand{\cHom}{\mathcal{H}om}
\newcommand{\SG}{\mathfrak{S}}
\newcommand{\GL}{\mathrm{GL}}
\newcommand{\DD}{\mathrm{DD}}
\newcommand{\PS}{\mathrm{PS}}

\newcommand{\Hom}{\mathrm{Hom}}

\newcommand{\Norm}{\mathrm{Norm}}
\newcommand{\Lie}{\mathrm{Lie}}
\newcommand{\Mod}{\mathrm{Mod}}
\newcommand{\End}{\mathrm{End}}

\newcommand{\BC}{\mathrm{BC}}
\newcommand{\PF}{\mathrm{PF}}
\newcommand{\Grpd}{\mathrm{Grpd}}
\newcommand{\Point}{\mathrm{Point}}
\newcommand{\Faith}[1]{\St^{\mathrm{faith}}_{/#1}}
\newcommand{\St}{\mathrm{Stack}}
\newcommand{\PreSt}{\mathrm{PreStack}}
\newcommand{\Rep}[1]{\St^{\mathrm{rep}}_{/#1}}

\newcommand{\SmSp}[1]{\Sp^{\mathrm{sm}}_{/#1}}
\newcommand{\Sp}{\mathrm{Sp}}
\newcommand{\SpSeq}{\mathrm{SpSeq}}
\newcommand{\AlgSp}{\mathrm{AlgSp}}
\newcommand{\Eq}{\mathrm{Eq}}

\newcommand{\Sh}{\mathrm{Sh}}

\newcommand{\Flag}{\mathcal{F}\mathit{lag}}
\newcommand{\Sect}{\mathcal{S}\mathit{ect}}
\newcommand{\Sym}{\mathrm{S}}
\newcommand{\Trans}{\mathrm{Trans}}
\newcommand{\Cent}{\mathrm{Cent}}

\newcommand{\SpOb}{\mathrm{SpOb}}

\newcommand{\cosk}{\mathrm{cosk}}
\newcommand{\Ner}{\mathrm{Ner}}
\newcommand{\group}{\mathrm{group}}
\newcommand{\op}{\mathrm{op}}
\newcommand{\id}{\mathrm{id}}
\newcommand{\pr}{\mathrm{pr}}
\newcommand{\gr}{\mathrm{gr}}
\newcommand{\sgn}{\mathrm{sgn}}
\newcommand{\Ob}{\mathrm{Ob}}
\newcommand{\red}{\mathrm{red}}
\newcommand{\sm}{\mathrm{sm}}
\newcommand{\cart}{\mathrm{cart}}

\newcommand\cA{{\mathcal A}}
\newcommand\cB{{\mathcal B}}
\newcommand\cC{{\mathcal C}}
\newcommand\cD{{\mathcal D}}
\newcommand\cE{{\mathcal E}}
\newcommand\cF{{\mathcal F}}
\newcommand\cG{{\mathcal G}}
\newcommand\cH{{\mathcal H}}
\newcommand\cL{{\mathcal L}}

\newcommand\cO{{\mathcal O}}
\newcommand\cP{{\mathcal P}}
\newcommand\cR{{\mathcal R}}
\newcommand\cS{{\mathcal S}}
\newcommand\cT{{\mathcal T}}
\newcommand\cU{{\mathcal U}}

\newcommand\cW{{\mathcal W}}
\newcommand\cX{{\mathcal X}}
\newcommand\cY{{\mathcal Y}}
\newcommand\cZ{{\mathcal Z}}

\newcommand\fp{{\mathfrak p}}
\newcommand\fq{{\mathfrak q}}
\newcommand\X{{\mathrm X}}
\newcommand\F{{\mathbb F}}
\newcommand\G{{\mathbb G}}
\newcommand\Z{{\mathbb Z}}

\newcommand\bP{{\mathbb P}}

\newcommand\R{{\mathbb R}}

\newcommand\N{{\mathbb N}}
\newcommand\Gm{{\mathbb G}_{\mathrm{m}}}

\newcommand\bone{{\mathbf 1}}
\newcommand\PP{\mathrm{PP}}
\newcommand\aff{\mathrm{aff}}
\newcommand\Supp{\mathrm{Supp}}
\newcommand\grvec{\mathrm{GrVec}}
\newcommand\grmod{\mathrm{GrMod}}
\newcommand\FTSch[1]{\mathrm{Sch}^{\mathrm{ft}}_{/#1}}

\newcommand{\res}{\mathbin{|}}

\numberwithin{equation}{subsection}
\setlist[enumerate,1]{(a)}
\setlist[enumerate,2]{(i)}
\setlist{nosep}

\begin{document}
\title{Quotient stacks and equivariant \'etale cohomology algebras:
Quillen's theory revisited}
\author{Luc Illusie\and Weizhe Zheng}
\date{}
\maketitle

\begin{flushright}
\emph{To the memory of Daniel Quillen}
\end{flushright}

\begin{abstract}
Let $k$ be an algebraically closed field. Let $\Lambda$ be a noetherian
commutative ring annihilated by an integer invertible in $k$ and let
$\ell$ be a prime number different from the characteristic of~$k$. We
prove that if $X$ is a separated algebraic space of finite type over~$k$
endowed with an action of a $k$-algebraic group $G$, the equivariant
\'etale cohomology algebra $H^*([X/G],\Lambda)$, where $[X/G]$ is the
quotient stack of $X$ by $G$, is finitely generated over $\Lambda$.
Moreover, for coefficients $K \in D^+_c([X/G],\F_{\ell})$ endowed with a
commutative multiplicative structure, we establish a structure theorem for
$H^*([X/G],K)$, involving fixed points of elementary abelian
$\ell$-subgroups of $G$, which is similar to Quillen's theorem
\cite[Theorem 6.2]{Quillen1} in the case $K = \F_{\ell}$. One key
ingredient in our proof of the structure theorem is an analysis of
specialization of points of the quotient stack. We also discuss variants
and generalizations for certain Artin stacks.
\end{abstract}

\section*{Introduction}

In \cite{Quillen1}, Quillen developed a theory for mod $\ell$ equivariant
cohomology algebras $H^*_G(X,\F_\ell)$, where $\ell$ is a prime number, $G$
is a compact Lie group, and $X$ is a topological space endowed with an
action of $G$. Recall that, for $r \in {\N}$, an \emph{elementary abelian
$\ell$-group of rank $r$} is defined to be a group isomorphic to the direct
product of $r$ cyclic groups of order~$\ell$ \cite[Section~4]{Quillen1}.
Quillen showed that $H^*_G(X,\Lambda)$ is a finitely generated
$\Lambda$-algebra for any noetherian commutative ring $\Lambda$
\cite[Corollary 2.2]{Quillen1} and established structure theorems
(\cite[Theorem 6.2]{Quillen1}, \cite[Theorem 8.10]{Quillen2}) relating the
ring structure of $H^*_G(X,\F_\ell)$ to the elementary abelian
$\ell$-subgroups $A$ of $G$ and the components of the fixed points set
$X^A$. We refer the reader to \cite[Section~1]{IllQuillen} for a summary of
Quillen's theory.

In this article, we establish an algebraic analogue. Let $k$ be an
algebraically closed field of characteristic $\neq \ell$ and let $\Lambda$
be noetherian commutative ring annihilated by an integer invertible in $k$.
Let $G$ be an algebraic group over $k$ (\emph{not} necessarily affine) and
let $X$ be a separated algebraic space of finite type over $k$ endowed with
an action of $G$. We consider the \'etale cohomology ring
$H^*([X/G],\Lambda)$ of the quotient stack $[X/G]$. One of our main results
is that this ring is a finitely generated $\Lambda$-algebra (Theorem
\ref{t.finite}) and the ring homomorphism
\[H^*([X/G],\F_\ell)\to \varprojlim_{\cA} H^*(BA, \F_\ell)\]
given by restriction maps is a uniform $F$-isomorphism (Theorem
\ref{c.main}), \ie has kernel and cokernel killed by a power of $F\colon
a\mapsto a^\ell$ (see Definition \ref{s.grvec} for a review of this notion
introduced by Quillen \cite[Section~3]{Quillen1}). Here $\cA$ is the
category of pairs $(A,C)$, where $A$ is an elementary abelian
$\ell$-subgroup of $G$ and $C$ is a connected component of $X^A$. The
morphisms $(A,C)\to (A',C')$ of $\cA$ are given by elements $g\in G$ such
that $Cg\supset C'$ and $g^{-1}Ag\subset A'$. We also establish a
generalization (Theorem \ref{t.main2}) for $H^*([X/G],K)$, where $K\in
D^+_c([X/G],\F_\ell)$ is a constructible complex of sheaves on $[X/G]$
endowed with a commutative ring structure.

A key ingredient in Quillen's original proofs is the continuity property
\cite[Proposition 5.6]{Quillen1}. In the algebraic setting, this property is
replaced by an analysis of the specialization of points of the quotient
stack $[X/G]$. In order to make sense of this, we introduce the notions of
geometric points and of $\ell$-elementary points of Artin stacks. Our
structure theorems for equivariant cohomology algebras are consequences of
the following general structure theorem (Theorem \ref{p.str}): if
$\cX=[X/G]$ or $\cX$ is a Deligne-Mumford stack of finite presentation and
finite inertia over $k$, and if $K\in D^+_c(\cX,\F_\ell)$ is endowed with a
commutative ring structure, then the ring homomorphism
\[H^*(\cX,K)\to \varprojlim_{x\colon \cS\to \cX} H^*(\cS,x^*K)\]
given by restriction maps is a uniform $F$-isomorphism. Here the limit is
taken over the category of $\ell$-elementary points of $\cX$.

In \cite{IllZheng} we established an algebraic analogue \cite[Theorem
8.1]{IllZheng} of a localization theorem of Quillen \cite[Theorem
4.2]{Quillen1}, which he had deduced from his structure theorems for
equivariant cohomology algebras. This was one of the motivations for us to
investigate algebraic analogues of these theorems. We refer the reader to
\cite{report} for a report on the present article and on some results of
\cite{IllZheng}.

In Part \ref{p.1} we review background material on quotient and classifying
stacks (Section \ref{s.1}), and collect results on the cohomology of Artin
stacks (Section \ref{s.1bis}) that are used at different places in this
article. The ring structures of the cohomology algebras we are considering
reflect ring structures on objects of derived categories. We discuss this in
Section~\ref{s.8}.

The reader familiar with the general nonsense recalled in Part~\ref{p.1}
could skip it and move directly to Part~\ref{p.2}, which contains the main
results of the paper. In Section~\ref{s.2}, we prove the above-mentioned
finiteness theorem (Theorem \ref{t.finite}) for equivariant cohomology
algebras. One key step of the proof amounts to replacing an abelian variety
by its $\ell$-divisible group, which was communicated to us by Deligne. In
Section~\ref{s.3}, we present a crucial result on the finiteness of orbit
types, which is an analogue of \cite[Lemma 6.3]{Quillen1} and was
communicated to us by Serre.

In Section~\ref{s.4}, we state the above-mentioned structure theorems
(Theorems \ref{c.main}, \ref{t.main2}) for equivariant cohomology algebras.
In Section~\ref{s.5}, we introduce and discuss the notions of geometric
points and of $\ell$-elementary points of Artin stacks. Using them we state
in Section~\ref{s.6} the main result of this paper, the structure theorem
(Theorem \ref{p.str}) for cohomology algebras of certain Artin stacks, and
show that it implies the structure theorems of the equivariant case. In
Section~\ref{s.Kunneth}, we establish some K\"unneth formulas needed in the
proof of Theorem \ref{p.str}, which is given in Section~\ref{s.7}. Finally,
in Section~\ref{s.strat} we prove an analogue of Quillen's stratification
theorem \cite[Theorems 10.2, 12.1]{Quillen2} for the reduced spectrum of mod
$\ell$ \'etale equivariant cohomology algebras.

The results of this paper have applications to the structure of varieties of
supports. We hope to return to this in a future article.

\subsection*{Acknowledgments}
We thank Pierre Deligne for the proof of the finiteness theorem (Theorem
\ref{t.finite}) in the general case and Jean-Pierre Serre for communicating
to us the results of Section~\ref{s.3}. We are grateful to Michel Brion for
discussions on the cohomology of classifying spaces and Michel Raynaud for
discussions on separation issues. The second author thanks Ching-Li Chai,
Johan de Jong, Yifeng Liu, Martin Olsson, David Rydh, and Yichao Tian for
useful conversations. We thank the referees for their careful reading of the
manuscript and many helpful comments.

Part of this paper was written during a visit of both authors to the Korea
Institute for Advanced Study in Seoul in January 2013 and a visit of the
first author to the Morningside Center of Mathematics, Chinese Academy of
Sciences in Beijing in February and March 2013. Warm thanks are addressed to
these institutes for their hospitality and support.

The second author was partially supported by China's Recruitment Program of
Global Experts; National Natural Science Foundation of China Grant 11321101;
Hua Loo-Keng Key Laboratory of Mathematics, Chinese Academy of Sciences;
National Center for Mathematics and Interdisciplinary Sciences, Chinese
Academy of Sciences.

\subsection*{Conventions}
We fix a universe $\cU$, which we will occasionally enlarge. We say
``small'' instead of ``$\cU$-small'' when there is no ambiguity. We say that
a category is \emph{essentially small} (resp.\ \emph{essentially finite}) if
it is equivalent to a small (resp.\ finite) category. Schemes are assumed to
be small. Presheaves take values in the category of $\cU$-sets. For any
category $\cC$, we denote by $\widehat{\cC}$ the category of presheaves
on~$\cC$, which is a $\cU$-topos if $\cC$ is essentially small. If $f\colon
\cC\to \cD$ is a fibered category, we denote by $\cC(U)$ (or sometimes
$\cC_U$) the fiber category of $f$ over an object $U$ of $\cD$.

By a \emph{stack} over a $\cU$-site $C$ we mean a stack in groupoids over
$C$ \cite[02ZI]{stacks} whose fiber categories are essentially
small.\footnote{The fiber categories of prestacks over $C$ are also assumed
to be essentially small.} By a stack, we mean a stack over the big fppf site
of $\Spec(\Z)$. Unlike \cite{LMB}, we do not assume algebraic spaces and
Artin stacks to be quasi-separated. We say that a morphism $f\colon \cX\to
\cY$ of stacks is \emph{representable} (this property is called
``representable by an algebraic space'' in \cite[02ZW]{stacks}) if for every
scheme $U$ and every morphism $y\colon U\to \cY$, the 2-fiber product
$U\times_{y,\cY,f}\cX$ is representable by an algebraic space. By an
\emph{Artin stack} (resp.\ \emph{Deligne-Mumford stack}), we mean an
``algebraic stack'' (resp.\ Deligne-Mumford stack) over $\Spec(\Z)$ in the
sense of \cite[026O]{stacks} (resp.\ \cite[03YO]{stacks}), namely a stack
$\cX$ such that the diagonal $\Delta_{\cX}\colon\cX\to\cX \times \cX$ is
representable and such that there exists an algebraic space $X$ and a smooth
(resp.\ \'etale) surjective morphism $X\to \cX$.

By an \emph{algebraic group} over a field $k$, we mean a group scheme over
$k$ of finite type. Unless otherwise stated, groups act on the right.

\tableofcontents

\part{Preliminaries}\label{p.1}
\section{Groupoids and quotient stacks}\label{s.1}

Classically, if $G$ is a compact Lie group, a classifying space $BG$ for $G$
is the base of a contractible (right) $G$-torsor $PG$. Such a classifying
space exists and is essentially unique (up to homotopy equivalence). If $X$
is a $G$-space (i.e.\ a topological space endowed with a continuous (right)
action of $G$), one can twist $X$ by $PG$ and get a space $PG \wedge^G X$,
defined as the quotient of $PG \times X$ by the diagonal action of $G$,
$((p,x),g) \mapsto(pg,xg)$. This space $PG \wedge^G X$ is a fiber bundle
over $BG$ of fiber~$X$, and $PG \times X$ is a $G$-torsor over $PG \wedge^G
X$. If $\Lambda$ is a ring, the equivariant cohomology of~$X$ with value in
$\Lambda$ is defined by
\[
H^*_G(X,\Lambda) \coloneqq H^*(PG \wedge^G X,\Lambda) \simeq H^*(BG,R\pi_*\Lambda)
\]
where $\pi\colon PG\wedge^G X \to BG$ is the projection. The functorial
properties of this cohomology, introduced by Borel, are discussed by Quillen
in \cite[Section~1]{Quillen1}.

A well-known similar formalism exists in algebraic geometry, with
classifying spaces replaced by classifying stacks. We review this formalism
in this section.

\begin{construction}\label{s.Eq}
Let $C$ be a category in which finite limits are representable. We define
the category
\[\Eq(C)\]
of equivariant objects in $C$ as follows. The objects of $\Eq(C)$ are pairs
$(X,G)$ consisting of a group object $G$ of~$C$ and an object $X$ of $C$
endowed with an action of~$G$, namely a morphism $X\times G\to X$ satisfying
the usual axioms for composition and identity. A morphism $(X,G)\to (Y,H)$
in $\Eq(C)$ is a pair $(f,u)$ consisting of a homomorphism $u\colon G\to H$
and a $u$-equivariant morphism $f\colon X\to Y$. Here the $u$-equivariance
of $f$ is the commutativity of the following diagram in $C$:
\[\xymatrix{X\times G\ar[r]\ar[d]_{f\times u}& X\ar[d]^{f}\\
Y\times H\ar[r] & Y.}\]

While $\Eq(C)$ is a category, groupoids in $C$ form a
(2,1)-category\footnote{A (2,1)-category is a 2-category whose 2-morphisms
are invertible.}
\[\Grpd(C).\]
We regard groupoids $X_\bullet$ in $C$ as internal categories, consisting of
two objects $X_0$ and $X_1$ of $C$, called respectively the object of
objects and the object of morphisms, together with four morphisms in $C$,
\[e\colon X_0\to X_1,\quad s,t\colon X_1\to X_0,\quad m\colon
X_1\times_{s_X,X_0,t_X}X_1\to X_1,
\]
called respectively identity, source, target, and composition. A 1-morphism
of groupoids $f_{\bullet} \colon X_{\bullet} \to Y_{\bullet}$ is an internal
functor between the underlying internal categories, namely a pair of
morphisms $f_0\colon X_0\to Y_0$, $f_1\colon X_1\to Y_1$, compatible with
$e$, $s$, $t$, $m$. For 1-morphisms of groupoids $f_{\bullet},g_\bullet
\colon X_{\bullet} \to Y_{\bullet}$, a 2-morphism $f_{\bullet} \to
g_{\bullet}$ is an internal natural isomorphism, namely, a morphism $r\colon
X_0 \to Y_1$ of $C$ such that $s_Y r=f_0$, $t_Y r=g_0$, and $m_Y(g_1,rs_X) =
m_Y(rt_X,f_1)$. The last identity can be stated informally as for any $(u
\colon a \to b) \in X_1$, $g_1(u)r(a) = r(b)f_1(u)$.

We define a functor
\begin{equation}\label{e.EqGrpd}
\Eq(C) \to \Grpd(C).
\end{equation}
as follows.  To an object $(X,G)$ of $\Eq(C)$, we assign a groupoid in $C$
\[
(X,G)_{\bullet}
\]
with $(X,G)_0 = X$, $(X,G)_1 = X \times G$, $e(x)=(x,1)$, $s(x,g)=xg$,
$t(x,g)=x$, and composition given by $(x,g)(xg,h) = (x,gh)$.  The
inverse-assigning morphism is $(x,g) \mapsto (xg,g^{-1})$. Here we follow
the conventions of \cite[3.4.3]{LMB} (see also \cite[0444]{stacks} where
groups act on the left). A morphism $(f,u)\colon (X,G)\to (Y,H)$ in $\Eq(C)$
gives a morphism of groupoids $(f,u)_{\bullet} \colon (X,G)_{\bullet} \to
(Y,H)_{\bullet}$, $(f,u)_0=f$, $(f,u)_1=f\times u\colon (x,g) \mapsto
(f(x),u(g))$.
\end{construction}

The functor \eqref{e.EqGrpd} is faithful, but not fully faithful. The
maximal 2-subcategory $\Grpd^{\Eq}(C)$ of $\Grpd(C)$ spanned by the objects
in the image of \eqref{e.EqGrpd} can be described as follows.

\begin{prop}\label{p.eq}
Let $(X,G)$, $(Y,H)$, and $(Z,I)$ be objects of $\Eq(C)$.
\begin{enumerate}
\item For any morphism of groupoids $\varphi  =
    (\varphi_0,\varphi_1)\colon (X,G)_{\bullet} \to (Y,H)_{\bullet}$,
    there exist a unique pair of morphisms $f\colon X \to Y$, $u\colon X
    \times G \to H$ such that $\varphi_1(x,g)  = (f(x),u(x,g))$, and the
    pair $(f,u)$ satisfies the following relations:
\begin{enumerate}
\item $f$ is \emph{$u$-equivariant}, \ie $f(xg) = f(x)u(x,g)$,

\item $u(x,g)u(xg,g') = u(x,gg')$.
\end{enumerate}
Conversely, any pair $(f,u)$ satisfying (i), (ii) defines a morphism of
groupoids $\varphi_\bullet$. Moreover, if $(a,u)\colon (X,G)_\bullet \to
(Y,H)_\bullet$ and $(b,v)\colon (Y,H)_\bullet\to (Z,I)_\bullet$ are
morphisms of groupoids, the composition is given by $(ba,w)$, where
$w\colon X\times G\to I$ is given by $w(x,g)=v(a(x),u(x,g))$.

\item Let $\varphi_i = (f_i,u_i)\colon (X,G)_{\bullet} \to
    (Y,H)_{\bullet}$ ($i = 1, 2$) be 1-morphisms of groupoids. Then a
    2-morphism from $\varphi_1$ to $\varphi_2$ is a morphism $r\colon X
    \to H$ satisfying the relations
\begin{enumerate}
\item $f_1(x) = f_2(x)r(x)$,

\item $r(x)u_1(x,g)=u_2(x,g)r(xg)$.
\end{enumerate}
Composition of 2-morphisms is given by multiplication in $H$.
\end{enumerate}
\end{prop}

We will sometimes call a morphism $u\colon X \times G \to H$ satisfying (a)
(ii) a \textit{crossed homomorphism}.

\begin{proof}
In (a), the uniqueness of $(f,u)$ are clear, while the existence (resp.\
(i), resp.\ (ii)) expresses the compatibility of $\varphi$ with the target
(resp.\ source, resp.\ composition) morphism. The other statements are
straightforward.
\end{proof}

\begin{definition}\label{s.faithfun}
We say that a pseudofunctor $F\colon \cC\to \cD$ between 2-categories is
\emph{faithful} (resp.\ \emph{fully faithful}) if for every pair of objects
$X$ and $Y$ in $\cC$, the functor $\Hom_\cC(X,Y)\to \Hom_\cD(FX,FY)$ induced
by $F$ is fully faithful (resp.\ an equivalence of categories). We say that
$F$ is \emph{essentially surjective} if for every object $Y$ of $\cD$, there
exists an object $X$ of $\cC$ and an equivalence $FX\simeq Y$ in~$\cD$.
\end{definition}

\begin{construction}\label{s.GrpdSt}
Let $E$ be a $\cU$-topos (we will be mostly interested in the case where $E$
is the topos of fppf sheaves on some algebraic space), endowed with its
canonical topology. A groupoid $X_\bullet$ in $E$ defines a category
$[X_\bullet]'$ fibered in groupoids over~$E$ whose fiber at $U$ is
$X_\bullet(U)$. This is an $E$-prestack, and, as in \cite[3.4.3]{LMB}, we
denote the associated $E$-stack \cite[02ZP]{stacks} by $[X_\bullet]$. If
$\pi$ denotes the canonical composite morphism
\[\pi\colon X_0\to [X_\bullet]'\to [X_\bullet],\]
the groupoid can be recovered from $\pi$: there is a natural isomorphism
\begin{equation}\label{e.1.4.1}
X_1 \simto X_0 \times_{[X_\bullet]} X_0
\end{equation}
identifying the projections $p_1,p_2\colon X_0 \times_{[X_\bullet]} X_0\to
X_0$ with $s$ and $t$, and identifying the second projection $\id\times
\pi\times \id\colon X_0 \times_{[X_\bullet]} X_0 \times_{[X_\bullet]} X_0\to
X_0 \times_{[X_\bullet]} X_0$ with $m$. Here $X_0\times_{[X_\bullet]} X_0$
denotes the sheaf carrying $U$ to the set of isomorphism classes of triples
$(x,y,\alpha)$, $x,y\in X_0(U)$, $\alpha\colon \pi(x)\simeq \pi(y)$. More
generally, there is a natural isomorphism of simplicial objects
\begin{equation}\label{e.1.4.2}
\Ner(X_\bullet) \simto \cosk_0(\pi)
\end{equation}
between the nerve of the groupoid $X_\bullet$ and the $0$-th coskeleton of
$\pi$.

We denote by $\St(E)$ (resp.\ $\PreSt(E)$) the (2,1)-category of $E$-stacks
($E$-prestacks). The pseudofunctor $\Grpd(E)\to\PreSt(E)$ sending
$X_\bullet$ to $[X_\bullet]'$ is fully faithful and the pseudofunctor
$\PreSt(E)\to\St(E)$ sending an $E$-prestack to its associated $E$-stack is
faithful. Therefore, the composite pseudofunctor
\begin{equation}\label{e.GrpdSt}
  \Grpd(E)\to \St(E)
\end{equation}
sending $X_\bullet$ to its associated $E$-stack $[X_\bullet]$ is faithful.
In other words, if $X_\bullet$, $Y_\bullet$ are groupoids in $E$, and
$\varphi_i\colon X_\bullet \to Y_\bullet$ ($i=1,2$) is a morphism of
groupoids, then the natural map
\[\Hom(\varphi_1,\varphi_2)\to \Hom([\varphi_1],[\varphi_2])\]
is bijective. However, in general, not every morphism $f \colon
[X_{\bullet}] \to [Y_{\bullet}]$ is of the form $[\varphi]$ for a morphism
of groupoids $\varphi \colon X_{\bullet} \to Y_{\bullet}$ (see Remark
\ref{r.notff} below). On the other hand, \eqref{e.GrpdSt} is essentially
surjective.
\end{construction}

\begin{notation}\label{s.wedge}
In the case of the groupoid $(X,G)_{\bullet}$ associated with a $G$-object
$X$ of $E$, the stack $[(X,G)_{\bullet}]$ is denoted by
\begin{equation}
[X/G]
\end{equation}
and called the \textit{quotient stack} of $X$ by $G$. For $X=e$ the final
object of $E$ (with the trivial action of $G$), it is called the
\textit{classifying stack} of $G$ and denoted by
\begin{equation}\label{e.1.5.2}
BG \coloneqq [e/G].
\end{equation}

Recall (\cite[2.4.2]{LMB}, \cite[04WM]{stacks}) that the projection $X \to
[X/G]$ makes $X$ into a universal $G$-torsor over $[X/G]$, i.e. for $U$ in
$E$, the groupoid $[X/G](U)$ is canonically equivalent to the category of
pairs $(P,a)$, where $P$ is a right $G_U$-torsor and $a$ is a
$G$-equivariant morphism from $P$ to $X$; morphisms from $(P,a)$ to $(Q,b)$
are $G$-equivariant morphisms $c\colon P\to Q$ such that $a=bc$.

The action of $G$ on $X$ is recovered from $\pi$: the isomorphism
\eqref{e.1.4.1} takes the form
\begin{equation}\label{e.XGX}
X \times G \simto X \times_{[X/G]} X,
\end{equation}
identifying the projections $p_1$, $p_2$ with $(x,g) \mapsto xg$, $(x,g)
\mapsto x$.

For $X = e$,  $BG(U)$ is the groupoid of $G$-torsors on $U$ for $U$ in $E$,
which justifies the terminology ``classifying stack''. For general $X$, the
projection $[X/G] \to BG$ induces $X \to e$ by the base change $B\{1\} \to
BG$, so that one can think of $[X/G] \to BG$ as a ``fibration'' with fiber
$X$. In other words, $[X/G]$ plays the role of the object $PG \wedge^G X$
recalled at the beginning of Section~\ref{s.1}.

In order to describe morphisms from $[X/G]$ to $[Y/H]$ associated to
morphisms of groupoids from $(X,G)_{\bullet}$ to $(Y,H)_{\bullet}$, we need
to introduce the following notation. Let $(X,G)$ be an object of $\Eq(E)$,
and let $u \colon X \times G \to H$ be a crossed homomorphism (Proposition
\ref{p.eq}). We denote by
\begin{equation}
X \wedge^{G,u} H
\end{equation}
the quotient of $X \times H$ by $G$ acting by $(x,h)g = (xg,u(x,g)^{-1}h)$.
This is an $H$-object of~$E$, the action of $H$ on it being deduced from its
action by right translations on $X \times H$. For any $H$-object $Y$ of $E$,
the map
\begin{equation}\label{e.eq0}
\Hom_u(X,Y) \to \Hom_H(X \wedge^{G,u} H,Y)
\end{equation}
sending a $u$-equivariant morphism $f$ (Proposition \ref{p.eq} (a) (i)) to
the morphism $f^u\colon (x,h) \mapsto f(x)h$ is bijective.

When $u \colon X \times G \to H$ is defined by $u(x,g) = u_0(g)$ for a group
homomorphism $u_0 \colon G \to H$, $X \wedge^{G,u} H$ coincides with the
usual contracted product \cite[D\'efinition III.1.3.1]{Giraud}, i.e.\ the
quotient of $X \times H$ by the diagonal action of $G$, $(x,h)g \coloneqq
(xg,u_0(g)^{-1}h)$.
\end{notation}

The following proposition, whose verification is straightforward, describes
the restriction of \eqref{e.GrpdSt} to $\Grpd^{\Eq}(E)$.

\begin{prop}\label{p.EqE}
Let $(X,G)$ and $(Y,H)$ be objects of $\Eq(E)$.
\begin{enumerate}
\item Let $(f,u) \colon (X,G)_{\bullet} \to (Y,H)_{\bullet}$ be a morphism
    of groupoids (Proposition \ref{p.eq}), and let
    \[[f/u] \colon [X/G] \to [Y/H]\]
    be the associated morphism of stacks. For $(P,a) \in [X/G](U)$,
    $[f/u](P,a)$ is the pair consisting of the $H$-torsor $P \wedge^{G,v}
    H$ (where $v$ is the composition of $a \times \id_G \colon P \times G
    \to X \times G$ and $u$) and the $H$-equivariant morphism $a^v\colon P
    \wedge^{G,v} H \to Y$ defined by $a$ via \eqref{e.eq0}.

\item Let $\varphi_1, \varphi_2, r$ be as in Proposition \ref{p.eq} (b).
    Then the 2-morphism $[r]\colon [f_1/u_1] \to [f_2/u_2]$ induced by $r$
    is given by the $Y$-morphism $P \wedge^{G,v_1} H \to P \wedge^{G,v_2}
    H$ sending $(p,h)$ to $(p,r(a(p))^{-1}h)$.
\end{enumerate}
\end{prop}

For a crossed homomorphism $u \colon X \times G \to H$, the unit section of
$H$ defines a $u$-equivariant morphism
\begin{equation}\label{e.eq}
X \to X \wedge^{G,u} H.
\end{equation}
The morphism of $E$-stacks
\begin{equation}\label{e.eq2}
  [X/G]\to [(X\wedge^{G,u} H)/H]
\end{equation}
induced by \eqref{e.eq} sends $T\to X$ to $T\wedge^{G,u} H \to X\wedge^{G,u}
H$.

\begin{remark}\label{r.notff}
The restriction of \eqref{e.GrpdSt} to $\Grpd^{\Eq}(E)$ is not fully
faithful in general. In other words, for objects $(X,G)$, $(Y,H)$ of
$\Eq(E)$, a morphism of stacks $[X/G]\to [Y/H]$ does not necessarily come
from a morphism of groupoids $(X,G)_\bullet\to (Y,H)_\bullet$. In fact, in
the case $G=\{1\}$ and $Y$ is a nontrivial $H$-torsor over $X$, any
quasi-inverse of the equivalence $[Y/H]\to X$ does not come from a morphism
of groupoids. See Proposition \ref{p.eqstr} for a useful criterion. See also
\cite[Proposition 5.1]{Zheng} for a calculus of fractions for the composite
functor $\Eq(E)\to \St(E)$ of \eqref{e.EqGrpd} and \eqref{e.GrpdSt}.
\end{remark}

\begin{definition}\label{s.faith}
We say that a morphism $X\to Y$ in a 2-category $\cC$ is \emph{faithful}
(resp.\ a \emph{monomorphism}) if for every object $U$ of $\cC$, the functor
$\Hom(U,X)\to \Hom(U,Y)$ is faithful (resp.\ fully faithful).
\end{definition}

In a 2-category, we need to distinguish between 2-limits \cite[Definition
1.4.26]{GR} and strict 2-limits (called ``2-limits'' in \cite[Definition
7.4.1]{Borceux}). Strict 2-products are 2-products. If a diagram $X\to
Y\from X'$ in $\cC$ admits a 2-fiber product $X\times_Y X'$ and a strict
2-fiber product $Z$, the canonical morphism $Z\to X\times_Y X'$ is a
monomorphism.

In a (2,1)-category $\cC$ admitting 2-fiber products, a morphism $X\to Y$ is
faithful (resp.\ a monomorphism) if and only if its diagonal morphism $X\to
X\times_Y X$ is a monomorphism (resp.\ an equivalence).

A morphism of $E$-prestacks $\cX\to \cY$ is faithful (resp.\ a monomorphism,
resp.\ an equivalence) if and only if $\cX(U)\to \cY(U)$ is a faithful
functor (resp.\ a fully faithful functor, resp.\ an equivalence of
categories) for every $U$ in $E$. If $\cX'$ is an $E$-prestack and $\cX$ is
its associated $E$-stack, then the canonical morphism $\cX'\to \cX$ is a
monomorphism.

Let $(f,u)\colon (X,G)\to (Y,H)$ be a morphism of $\Eq(E)$. If $u$ is a
monomorphism, then $[f/u]\colon [X/G]\to [Y/H]$ is faithful.

\begin{remark}\label{r.EqGrpd}
The category $\Eq(C)$ admits finite limits, whose formation commutes with
the projection functors $(X,G) \mapsto X$ and $(X,G) \mapsto G$ from
$\Eq(C)$ to $C$ and to the category of group objects of $C$, respectively.
The 2-category $\Grpd(C)$ admits finite strict 2-limits, whose formation
commutes with the projection 2-functors $X_\bullet\mapsto X_0$ and
$X_\bullet\mapsto X_1$ from $\Grpd(C)$ to~$C$. The functor
$\Eq(C)\to\Grpd(C)$ \eqref{e.EqGrpd} sending $(X,G)$ to $(X,G)_\bullet$
carries finite limits to finite \emph{strict} 2-limits.

The 2-category $\Grpd(C)$ admits finite 2-limits as well. The 2-fiber
product of a diagram $X_\bullet\xrightarrow{f} Y_\bullet \xleftarrow{g}
Y'_\bullet$ in $\Grpd(C)$ is the groupoid $W_\bullet$ of triples
$(x,y,\alpha)$, where $x\in X_0$, $y\in Y'_0$, and $(\alpha\colon f(x)\simto
g(y))\in Y_1$. More formally, $W_0=X_0\times_{Y_0,s_Y} Y_1\times_{t_Y,Y_0}
Y'_0$ and $W_1$ is the limit of the diagram
\[X_1\to Y_1\xleftarrow{p_1} Y_1\times_{Y_0} Y_1 \xrightarrow{m} Y_1
\xleftarrow{m} Y_1\times_{Y_0} Y_1\xrightarrow{p_2} Y_1 \from Y'_1.
\]
A morphism $X_\bullet\to Y_\bullet$ in $\Grpd(C)$ is faithful (resp.\ a
monomorphism) if and only if the morphism $X_1\to (X_0\times
X_0)\times_{Y_0\times Y_0, (s_Y,t_Y)} Y_1$ is a monomorphism (resp.\
isomorphism).

The category $\St(E)$ admits small 2-limits. The pseudofunctor $\Grpd(E)\to
\St(E)$ \eqref{e.GrpdSt} preserves finite 2-limits and thus preserves
faithful morphisms and monomorphisms.
\end{remark}

\begin{remark}
A commutative square in $\Eq(E)$,
\begin{equation}\label{e.Eqsquare}
\xymatrix{(X',G') \ar[r]^{(f',\gamma')} \ar[d]_{(p,u)} & (Y',H') \ar[d]^{(q,v)}\\
(X,G) \ar[r]^{(f,\gamma)} & (Y,H)}
\end{equation}
induces a 2-commutative square of $E$-stacks
\begin{equation}\label{e.Stsquare}
\xymatrix{[X'/G'] \ar[r] \ar[d] & [Y'/H'] \ar[d] \\
[X/G] \ar[r] & [Y/H].}
\end{equation}
It is not true in general that if \eqref{e.Eqsquare} is cartesian,
\eqref{e.Stsquare} is 2-cartesian, as \eqref{e.XGX} already shows. However,
we have the following result, which is a partial generalization of
\cite[Proposition 5.4]{Zheng}.
\end{remark}

\begin{prop}\label{p.cart}
Consider a cartesian square \eqref{e.Eqsquare} in $\Eq(E)$. If the morphism
    \begin{equation}\label{e.groups}
    H'\times G \to H
    \end{equation}
in $E$ given by $(h,g)\mapsto v(h)\gamma(g)$ is an epimorphism, then
\eqref{e.Stsquare} is 2-cartesian.
\end{prop}

\begin{proof}
Let
\[\alpha\colon [X'/G']\to \cX \coloneqq [X/G]\times_{[Y/H]}[Y'/H']\]
be the induced morphism of $E$-stacks. By Remark \ref{r.EqGrpd} and the
remark following Definition \ref{s.faith}, $\alpha$ is a monomorphism. We
need to show that for any object $V$ of~$E$, the functor $\alpha_V \colon
[X'/G']_V \to \mathcal{X}_V$ is essentially surjective. By definition,
$\cX_V$ is the category of triples $((T,t),(T',t'),s)$, where $(T,t\colon
T\to X)$ is an object of $[X/G]_V$, $(T',t'\colon T'\to Y')$ is an object of
$[Y'/H']_V$, and $s\colon [f/u]_V(T,t)\to [q/v]_V(T',t')$ is an isomorphism.
In other words (Proposition \ref{p.EqE} (b)), $s\colon T\wedge^{G,\gamma}
H\to T'\wedge^{H',v} H$ is an isomorphism of $H$-torsors over $V$,
compatible with the morphisms to~$Y$ (induced by $qt'$ and $ft$). The
functor $\alpha_V$ sends an object $(P,w)$ of $[X'/G']_V$ to
$([p/u]_V(P,w),[f'/\gamma']_V(P,w),\sigma)$, where $\sigma\colon [fp/\gamma
u]_V (P,w)\to [qf'/v\gamma']_V (P,w)$ is the obvious isomorphism. Let
$a=((T,t), (T',t'), s)$ be an object of $\cX_V$. It remains to show that
there exist a cover $(V_i\to V)_{i\in I}$ and, for every $i\in I$, an object
$(P_i,w_i)$ of $[X'/G']_{V_i}$ such that $\alpha(P_i,w_i)\simeq a_{V_i}$.
Take a cover $(V_i\to V)_{i\in I}$ such that for every $i$, $T_{V_i}$ and
$T'_{V_i}$ are both trivial and choose trivializations of them. Then
$s_{V_i}$ is represented by the left multiplication by some $h_i\in H(V_i)$.
By the assumption on \eqref{e.groups}, we may assume
$h_i=v(h'_i)\gamma(g_i)$, $h'_i\in H'(V_i)$, $g_i\in G(V_i)$. In this case,
the square
\begin{equation}\label{e.cart}
  \xymatrix{H_{V_i}\ar[r]^{s_{V_i}}\ar[d]_{\lambda_{\gamma(g_i)}} & H_{V_i}\ar[d]^{\lambda_{v(h'_i)}^{-1}}\\
  H_{V_i}\ar[r]^{1} & H_{V_i}}
\end{equation}
commutes, where $\lambda_{h}$ is the left multiplication by $h$. Thus
\eqref{e.cart} gives an isomorphism $a_{V_i}\simeq b_i$, where
$b_i=((G_{V_i},t\lambda_{g_i}^{-1}),(H'_{V_i},t'\lambda_{h'_i}),1)$. Taking
the product of $(G_{V_i},t\lambda_{g_i}^{-1})$ and
$(H'_{V_i},t'\lambda_{h'_i})$ over $(H_{V_i}, (t\lambda_{g_i}^{-1})^{\gamma}
= (t'\lambda_{h'_i})^{v})$ gives us an element $(P_i,w_i)$ of $[X/G]_{V_i}$
whose image under $\alpha$ is $b_i$.
\end{proof}

\begin{cor}
Suppose $u \colon G \to Q$ is an epimorphism of groups of $E$, with
kernel~$K$. Then the natural morphism
\[
BK\simto e \times_{BQ} BG
\]
is an equivalence.
\end{cor}

In other words, we can view $Bu \colon BG \to BQ$ as a fibration of fiber
$BK$.

\begin{definition}
We say that a groupoid $X_\bullet$ in $E$ is an \emph{equivalence relation}
if $(s_X,t_X)\colon X_1\to X_0\times X_0$ is a monomorphism. In this case,
the associated $E$-stack $[X_\bullet]$ is represented by the quotient sheaf
in $E$. We say that the action of $G$ on $X$ is \emph{free} if the
associated groupoid $(X,G)_\bullet$ is an equivalence relation. In this
case, $[X/G]$ is represented by the sheaf $X/G$.
\end{definition}

\begin{prop}\label{c.quot}
Let $(X,G)$ be an object in $\Eq(E)$, and let $K$ be a normal subgroup of
$G$ acting freely on $X$. Then the morphism $f\colon [X/G]\to [(X/K)/(G/K)]$
is an equivalence.
\end{prop}

\begin{proof}
Indeed, for every $U$ in $E$, $[(X/K)/(G/K)]_U$ is the category of pairs
$(T,\alpha)$, where $T$ is a $G/K$-torsor and $\alpha\colon T\to X/K$ is a
$G$-equivariant map, and the functor $f_U$ admits a quasi-inverse carrying
$(T,\alpha)$ to its base change by the projection $X\to X/K$.
\end{proof}

The following \emph{induction formula} will be useful later in the
calculation of equivariant cohomology groups (cf.\ \cite[(1.7)]{Quillen1}).

\begin{cor}\label{p.wedge}
Let $(X,G)$ be an object of $\Eq(E)$ and let $u\colon X\times G\to H$ be a
crossed homomorphism. Assume that the action of $G$ on $X\times H$, as
defined in Notation \ref{s.wedge}, is free (Definition \ref{s.faith}). Then
$f\colon [X/G]\to [X\wedge^{G,u} H/H]$ \eqref{e.eq2} is an equivalence.
\end{cor}

\begin{proof}
The morphism $f$ can be decomposed as
\[[X/G]\xrightarrow{\alpha} [X\times H/G\times H]\xrightarrow{\beta} [X\wedge^{G,u} H/H],\]
where $\beta$ is an equivalence by Proposition \ref{c.quot}, and $\alpha$ is
induced by the morphism $X\to X\times H$ given by the unit section of $H$
and the crossed homomorphism $X\times G\to G\times H$ sending $(x,g)$ to
$(g,u(x,g))$. Since $\alpha$ is a 2-section of the morphism $[X\times
H/G\times H]\to [X/G]$, which is an equivalence by Proposition \ref{c.quot},
$\alpha$ is also an equivalence.
\end{proof}

\begin{cor}\label{p.quot}
Let $u\colon H\hookrightarrow G$ be a monomorphism of group objects in $E$.
Then
\begin{enumerate}
\item The morphism of stacks $BH \to [(H\backslash G)/G]$ is an
    equivalence.

\item The natural morphism $H \backslash G \to e \times_{BG} BH$ is an
    isomorphism.
\end{enumerate}
\end{cor}

In other words, (a) says that, for any homogeneous space $X$ of group $G$,
if $H$ is the stabilizer of a section $x$ of $X$, then the morphism $BH \to
[X/G]$ given by $x \colon e \to X$ is an equivalence, while (b) can be
thought as saying that $BH \to BG$ is a fibration of fiber $H\backslash G$.

\begin{proof}
  Assertion (a) follows from Corollary \ref{p.wedge}. Assertion (b)
  follows from Proposition \ref{p.cart} applied to the cartesian square
  \[\xymatrix{(H\backslash G, \{1\})\ar[r]\ar[d] & (e, \{1\})\ar[d]\\
  (H\backslash G,G)\ar[r] & (e,G)}\]
  (cf.\ the paragraph following \eqref{e.XGX}) and from (a).
\end{proof}

\begin{construction}\label{s.EqX}
We will apply the above formalism to a relative situation, which we now
describe. Let $\cX$ be an $E$-stack. We denote by $\St_{/\cX}$ the
(2,1)-category of $E$-stacks over $\cX$. An object of $\St_{/\cX}$ is a pair
$(\cY,y)$, where $\cY$ is an $E$-stack and $y\colon \cY\to \cX$ is a
morphism of $E$-stacks. A morphism in $\St_{/\cX}$ from $(\cY,y)$ to
$(\cZ,z)$ is a pair $(f,\alpha)$, where $f\colon \cY\to \cZ$ is a morphism
of $E$-stacks and $\alpha\colon y\to zf$ is a 2-morphism:
\begin{equation}\label{e.tri}
  \xymatrix{\cY\ar[r]^{f}\ar[dr]_{y} & \cZ\ar[d]^z\\& \cX.\ultwocell\omit{<2>}}
\end{equation}
A 2-morphism $(f,\alpha)\to (g,\beta)$ in $\St_{/\cX}$ is a 2-morphism
$\eta\colon f\to g$ in the (2,1)-category $\St(E)$ such that $\beta=(z*\eta)
\circ \alpha$.

A morphism $y\colon \cY\to \cX$ of $E$-stacks is faithful (Definition
\ref{s.faith}) if and only if for any object $U$ of $E$ and any morphism
$x\colon U\to \cX$, the 2-fiber product $U\times_{x,\cX,y} \cY$ is
isomorphic to a sheaf. Consider the 2-subcategory $\cS$ of $\St_{/\cX}$
spanned by objects $(\cY,y)$ with $y$ faithful. For any morphism
$(f,\alpha)\colon (\cY,y)\to (\cZ,z)$ in $\cS$, $f$ is necessarily faithful.
A 2-morphism $\eta\colon (f,\alpha)\to (g,\beta)$ in $\cS$, if it exists, is
uniquely determined by $(f,\alpha)$ and $(g,\beta)$. In other words, if we
denote by $\Faith{\cX}$ the category obtained from $\cS$ by identifying
isomorphic morphisms, then the 2-functor $\cS\to \Faith{\cX}$ is a
2-equivalence.

For any morphism $\phi\colon \cX\to \cY$ of $E$-stacks, base change by
$\phi$ induces a functor $\Faith{\cY}\to \Faith{\cX}$. If $S$ is an object
of $E$, $\Faith{S}$ is equivalent to $E_{/S}$. More generally, if
$U_\bullet$ is a groupoid in $E$, $\Faith{[U_\bullet]}$ is equivalent to the
category of descent data relative to $U_\bullet$. In particular, if $(X,G)$
is an object of $\Eq(E)$, $\Faith{[X/G]}$ is equivalent to the category of
$G$-objects of $E$, equivariant over $X$. For example, $\Faith{BG}$ is
equivalent to the topos $B_G$ of Grothendieck.
\end{construction}

\begin{prop}\label{p.stacktopos}\leavevmode
\begin{enumerate}
\item The category $\Faith{\cX}$ is a $\cU$-topos.

\item Let $\mathcal{X}$ be a stack. For any stack $\mathcal{Y}$ over
    $\mathcal{X}$, associating to any stack $\mathcal{Z}$ faithful over
    $\mathcal{X}$ the groupoid
    $\Hom_{\mathcal{X}}(\mathcal{Z},\mathcal{Y})$ defines a stack
    $\underline{\mathcal{Y}}$ over $\Faith{\cX}$. The 2-functor
 \begin{equation}\label{e.stacktopos}
\St_{/\cX}\to \St(\Faith{\cX}), \quad \mathcal{Y} \mapsto \underline{\mathcal{Y}}
\end{equation}
is a 2-equivalence.
\end{enumerate}
\end{prop}

\begin{proof}
(a) We apply Giraud's criterion \cite[IV Th\'eor\`eme 1.2]{SGA4}. If $\cT$
is a small generating family of~$E$, then $\coprod_{U\in \cT}\Ob(\cX(U))$ is
an essentially small generating family of $\Faith{\cX}$. Let us now show
that every sheaf $\cF$ on $\Faith{\cX}$ for the canonical topology is
representable. Consider, for every object $U$ of $E$, the category of pairs
$(x,s)$ consisting of $x\in \cX(U)$ and $s\in \Gamma(x,\cF)$, where the last
occurrence of $x$ is to be understood as the object $x\colon U\to \cX$ in
$\Faith{\cX}$. A morphism $(x,s)\mapsto (y,t)$ is a morphism $\alpha\colon
x\to y$ in $\cX(U)$ such that $\alpha^* t=s$. This defines an $E$-stack
$\cX'$. The faithful morphism $\cX'\to \cX$ of $E$-stacks defined by the
first projection $(x,s)\mapsto x$ represents $\cF$. The other conditions in
Giraud's criterion are trivially satisfied. Thus $\Faith{\cX}$ is a
$\mathcal{U}$-topos.

(b) We construct a 2-quasi-inverse to \eqref{e.stacktopos} as follows. Let
$\cC$ be a stack over $\Faith{\cX}$. For every object $U$ of $E$, consider
the category of pairs $(x,s)$ consisting of $x\in \cX(U)$ and $s\in \cC(x)$.
A morphism $(x,s)\to (y,t)$ is a pair $(\alpha,\beta)$ consisting of a
morphism $\alpha\colon x\to y$ in $\cX(U)$ and a morphism $\beta\colon
\alpha^*t\to s$ in $\cC(x)$. This defines an $E$-stack $\cY$. The first
projection $(x,s)\mapsto x$ defines a morphism $\cY\to \cX$ of $E$-stacks.
The construction $\cC\mapsto (\cY\to \cX)$ defines a pseudofunctor
\begin{equation}\label{e.StFaith}
\St(\Faith{\cX})\to \St_{/\cX},
\end{equation}
which is a 2-quasi-inverse to \eqref{e.stacktopos}.
\end{proof}

The composition of \eqref{e.GrpdSt} and \eqref{e.StFaith} is a faithful and
essentially surjective (Definition \ref{s.faithfun}) pseudofunctor
\begin{equation}\label{e.GrpdSt2}
  \Grpd(\Faith{\cX})\to \St_{/\cX}.
\end{equation}
We denote the image of a groupoid $X_\bullet$ in $\Faith{\cX}$ under
\eqref{e.GrpdSt2} by $[X_\bullet/\cX]$, and the image of a morphism
$f_\bullet$ of groupoids under \eqref{e.GrpdSt2} by $[f_\bullet/\cX]$. For
$(X,G)$ in $\Eq(\Faith{\cX})$, we denote the image of $(X,G)_\bullet$ under
\eqref{e.GrpdSt2} by $[X/G/\cX]$. For $(f,u)\colon (X,G)_\bullet \to
(Y,H)_\bullet$, we denote the image under \eqref{e.GrpdSt2} by $[f/u/\cX]$.

We now apply the above formalism to the big fppf topoi of algebraic spaces.
Recall that a stack is a stack over the big fppf site of $\Spec \Z$. The
following result will be useful in Sections \ref{s.5} and \ref{s.6}.

\begin{prop}\label{p.eqstr}
Let $\cX$ be a stack, and let $X_\bullet$, $Y_\bullet$ be objects in
$\Grpd(\Faith{\cX})$. Assume that $X_0$ is a strictly local scheme and the
morphisms $Y_1\rightrightarrows Y_0$ are representable and smooth. Then the
functor induced by \eqref{e.GrpdSt2}:
  \[F\colon \Hom_{\Grpd(\Faith{\cX})}(X_\bullet,Y_\bullet)\to \Hom_{\St_{/\cX}}([X_\bullet/\cX],[Y_\bullet/\cX])\]
  is an equivalence of categories.
\end{prop}

\begin{proof}
It remains to show that $F$ is essentially surjective. Let $\phi\colon
[X_\bullet/\cX]\to [Y_\bullet/\cX]$ be a morphism in $\St_{/\cX}$. For the
2-cartesian square
  \[\xymatrix{X'_0\ar[rr]\ar[d]&& Y_0\ar[d]\\
  X_0\ar[r]&[X_\bullet/\cX]\ar[r]^\phi &[Y_\bullet/\cX].}\]
Since $X'_0$ is representable and smooth over $X_0$, it admits a section by
\cite[Corollaire 17.16.3 (ii), Proposition 18.8.1]{EGAIV}, which induces a
2-commutative square
  \[\xymatrix{X_0\ar[d]\ar[r]^{f_0} & Y_0\ar[d]\\
  [X_\bullet/\cX]\ar[r]^\phi & [Y_\bullet/\cX].}\]
Let $f_1=f_0\times_\phi f_0\colon X_1\to Y_1$. Then $f_\bullet\colon
X_\bullet \to Y_\bullet$ is a morphism of groupoids in $\Faith{\cX}$ and
$\phi\simeq [f_\bullet/\cX]$.
\end{proof}

\begin{remark}\label{s.rep}
Let $\cX$ be a stack. We denote by $\Rep{\cX}$ the full subcategory of
$\Faith{\cX}$ consisting of representable morphisms $X\to \cX$.  A morphism
in this category from $X\to \cX$ to $Y\to \cX$ is an isomorphism class of
pairs $(f,\alpha)$ \eqref{e.tri}. The morphisms $f\colon X\to Y$ are
necessarily representable. Assume that $\cX$ is an Artin stack. For any
object $X\to \cX$ of $\Rep{\cX}$, $X$ is necessarily an Artin stack. For any
object $X_\bullet$ in $\Grpd(\Rep{\cX})$, if $s_X$ and $t_X$ are flat and
locally of finite presentation, then $[X_\bullet/\cX]$ is an Artin stack. In
particular, for any object $(X,G)$ in $\Eq(\Rep{\cX})$ with $G$ flat of and
locally of finite presentation over $\cX$, $[X/G/\cX]$ is an Artin stack.
\end{remark}

\section{Miscellany on the \'etale cohomology of Artin stacks}\label{s.1bis}

\begin{notation}\label{s.smtop}
Let $\cX$ be an Artin stack. We denote by $\AlgSp_{/\cX}$ the full
subcategory of $\Rep{\cX}$ (Remark \ref{s.rep}) consisting of morphisms
$U\to \cX$ with $U$ an algebraic space. We let $\SmSp{\cX}$ denote the full
subcategory of $\AlgSp_{/\cX}$ spanned by smooth morphisms $U\to \cX$. The
covering families of the smooth pretopology on $\SmSp{\cX}$ are those
$(U_i\to U)_{i\in I}$ such that $\coprod_{i\in I}U_i\to U$ is smooth and
surjective. The covering families for the \'etale pretopology on
$\SmSp{\cX}$ are those $(U_i\to U)_{i\in I}$ such that $\coprod_{i\in
I}U_i\to U$ is \'etale and surjective. Since every smooth cover in
$\SmSp{\cX}$ has an \'etale refinement by \cite[Corollaire 17.16.3
(ii)]{EGAIV}, the smooth pretopology and the \'etale pretopology generate
the same topology on $\SmSp{\cX}$ (cf.\ \cite[D\'efinition 12.1]{LMB}). We
let $\cX_{\sm}$ denote the associated topos, and call it the \emph{smooth
topos} of $\cX$.
\end{notation}

\begin{notation}\label{s.cart}
The category of sheaves in $\cX_\sm$ is equivalent to the category of
systems $(\cF_u,\theta_\phi)$, where $u\colon U\to \cX$ runs through objects
of $\SmSp{\cX}$, $\phi\colon u\to v$ runs through morphisms of $\SmSp{\cX}$,
$\cF_u$ is an \'etale sheaf on $U$, and $\theta_\phi\colon \phi^*\cF_v\to
\cF_u$, satisfying a cocycle condition \cite[12.2]{LMB} and such that
$\theta_\phi$ is an isomorphism for $\phi$ \'etale. Following
\cite[D\'efinition 12.3]{LMB}, we say that a sheaf $\cF$ on $\cX$ is
\emph{cartesian} if $\theta_\phi$ is an isomorphism for all $\phi$, or,
equivalently, for all $\phi$ smooth (cf. \cite[Lemma 3.8]{OlssonSh}). We
denote by $\Sh_\cart(\cX)$ the full subcategory of $\Sh(\cX_{\sm})$
consisting of cartesian sheaves.

Let $\Lambda$ be a commutative ring. Following \cite[D\'efinition
18.1.4]{LMB}, we say, if $\Lambda$ is noetherian, that a sheaf $\cF$ of
$\Lambda$-modules on $\cX$ is \emph{constructible} if $\cF$ is cartesian and
if $\cF_u$ is constructible for some smooth atlas $u\colon U\to\cX$, or
equivalently, for every smooth atlas $u\colon U\to \cX$. We denote by
$\Mod_\cart(\cX,\Lambda)$ (resp.\ $\Mod_c(\cX,\Lambda)$) the full
subcategory of $\Mod(\cX_\sm,\Lambda)$ consisting of cartesian (resp.\
constructible) sheaves.

We denote by $D_\cart(\cX,\Lambda)$ (resp.\ $D_c(\cX,\Lambda)$) the full
subcategory of $D(\cX_\sm,\Lambda)$ consisting of complexes with cartesian
(resp.\ constructible) cohomology sheaves. We have $D_c(\cX,\Lambda)\subset
D_\cart(\cX,\Lambda)$. We will work exclusively with $D_\cart(\cX,\Lambda)$
rather than $D(\cX_\sm,\Lambda)$. We have functors
\[\otimes^L_\Lambda\colon D_\cart(\cX,\Lambda)\times D_\cart(\cX,\Lambda)\to D_\cart(\cX,\Lambda),\quad R\cHom\colon D_\cart(\cX,\Lambda)^{\op}\times D_\cart(\cX,\Lambda)\to D_\cart(\cX,\Lambda)\]
defined on unbounded derived categories.

If $\cX$ is a Deligne-Mumford stack, we denote by $\cX_\et$ or simply $\cX$
its \'etale topos. The inclusion of the \'etale site in the smooth site
induces a morphism of topoi $(\epsilon_*,\epsilon^*)\colon \cX_{\sm}\to
\cX_\et$. Note that $\epsilon_*$ is exact and $\epsilon^*$ induces an
equivalence from $\cX_\et$ to $\Sh_\cart(\cX_\sm)$. For any commutative ring
$\Lambda$, $\epsilon^*$ induces $D(\cX,\Lambda)\simto D_\cart(\cX,\Lambda)$.
\end{notation}

\begin{notation}\label{s.fCart}
Let $f\colon \cX\to \cY$ be a morphism of Artin stacks and let $\Lambda$ be
a commutative ring. Although the smooth topos is not functorial, we have a
pair of adjoint functors
\[f^*\colon \Sh_\cart(\cY)\to \Sh_\cart(\cX),\quad f_*\colon \Sh_\cart(\cX)\to \Sh_\cart(\cY).\]
and a pair of adjoint functors \cite{LiuZheng}
\[f^*\colon D_\cart(\cY,\Lambda)\to D_\cart(\cX,\Lambda),\quad Rf_*\colon D_\cart(\cX,\Lambda)\to D_\cart(\cY,\Lambda),\]
where $f^*$ is $t$-exact and $Rf_*$ is left $t$-exact for the canonical
$t$-structures. Note that $Rf_*$ is defined on the whole category $D_\cart$,
not just on $D_\cart^+$. For $M,N\in D_\cart(\cY,\Lambda)$, we have a
natural isomorphism
\[f^*(M\otimes_\Lambda^L N)\simto f^*M\otimes_\Lambda^L f^*N.\]
If $f$ is a surjective morphism, then the functors $f^*$ are conservative
and the functor $f^*\colon \Sh_\cart(\cY)\to \Sh_\cart(\cX)$ is faithful.

A 2-morphism $\alpha\colon f\to g$ of morphisms of Artin stacks $\cX\to \cY$
induces natural isomorphisms $\alpha^*\colon g^*\to f^*$ and
$R\alpha_*\colon Rf_*\to Rg_*$. The following squares commute
\[\xymatrix{\bone_{D_\cart(\cY,\Lambda)}\ar[r]\ar[d] & Rf_*f^*\ar[d]^{R\alpha_*} & g^*Rf_*\ar[d]_{\alpha^*}\ar[r]^{R\alpha_*} & g^*Rg_*\ar[d]\\
Rg_*g^*\ar[r]^{\alpha^*} & Rg_*f^* & f^*Rf_*\ar[r] & \bone_{D_\cart(\cX,\Lambda)}.}\]
\end{notation}

Recall that a morphism of Artin stacks $f\colon \cX\to \cY$ is
\emph{universally submersive} \cite[06U6]{stacks} if for every morphism of
Artin stacks $\cY'\to \cY$, the base change $\cY'\times_\cY \cX\to \cY'$ is
submersive (on the underlying topological spaces).

\begin{prop}\label{p.descent}
Let $f\colon \cX\to \cY$ be a morphism of Artin stacks. Assume that $f$ is
universally submersive (resp.\ faithfully flat and locally of finite
presentation). Then $f$ is of descent (resp.\ effective descent) for
cartesian sheaves.
\end{prop}

Here effective descent means $f^*$ induces an equivalence
$\Sh_\cart(\cY)\simto \DD(f)$ to the category of descent data, whose objects
are cartesian sheaves $\cF$ on $\cX$ endowed with an isomorphism $p_1^*\cF
\to p_2^* \cF$ satisfying the cocycle condition, where $p_1,p_2\colon
\cX\times_\cY \cX \to \cX$ are the two projections.

\begin{proof}
By general properties of descent \cite[Proposition 6.25, Th\'eor\`eme
10.4]{GiraudDesc} and the case of schemes \cite[VIII Proposition 9.1]{SGA4}
(resp.\ \cite[VIII Th\'eor\`eme 9.4]{SGA4}), it suffices to show that smooth
atlases are of effective descent for cartesian sheaves. In other words we
may assume $f$ is smooth and $\cX$ is an algebraic space. In this case, we
construct a quasi-inverse $F$ of $\Sh_\cart(\cY)\to\DD(f)$ as follows. Let
$A$ be a descent datum for $f$. For every object $u\colon U\to \cY$ of
$\Sp^\sm_{/\cY}$, $A$ induces a descent datum $A_u$ for \'etale sheaves for
the base change $f_u\colon \cX\times_{\cY} U\to U$ of $f$ by $u$, and we
take $(FA)_u$ to be the corresponding \'etale sheaf on $U$. For a morphism
$\phi\colon u\to v$ in $\Sp^\sm_{/\cY}$, we take $\phi^*(FA)_v\to(FA)_u$ to
be the isomorphism induced by the isomorphism of descent data $\phi^*A_v\to
A_u$ for \'etale sheaves for $f_u$.
\end{proof}

\begin{cor}\label{c.XG}
Let $S$ be an algebraic space, let $G$ be a flat group algebraic $S$-space
locally of finite presentation, and let $X$ be an algebraic space over $S$,
endowed with an action of~$G$. Denote by $\alpha\colon G\times_S X\to X$ the
action and by $p\colon G\times_S X\to X$ the projection, and let $f\colon
X\to [X/G]$ be the canonical morphism. Then $f^*$ induces an equivalence of
categories from $\Mod_\cart([X/G])$ to the category of pairs $(\cF,a)$,
where $\cF\in \Sh(X)$ and $a\colon \alpha^* \cF\to p^*\cF$ is a map
satisfying the usual cocycle condition.
\end{cor}

Such pairs are called $G$-equivariant sheaves on $X$. The cocycle condition
implies that $i^*a\colon \cF\to \cF$ is the identity, where $i\colon X\to
G\times_S X$ is the morphism induced by the unit section of $G$.

\begin{proof}
This follows from Proposition \ref{p.descent} and the fact that $f$ is
faithfully flat of finite presentation.
\end{proof}

\begin{cor}\label{l.BG}
Let $S$ and $G$ be as in Corollary \ref{c.XG}. Assume that $G$ has
\emph{connected} geometric fibers. Let $f\colon S\to BG$ be a morphism
corresponding to a $G$-torsor $T$ on $S$. Then the functor
\[f^*\colon\Sh_\cart(BG)\to \Sh(S),\]
is an equivalence.
\end{cor}

\begin{proof}
By Proposition \ref{p.descent}, since $f$ is faithfully flat locally of
finite presentation, $f^*$ induces an equivalence of categories from
$\Sh_\cart(BG)$ to the category of pairs $(\mathcal{F},a)$, where $\cF$ is a
sheaf on $S$ and $a\colon p^*\cF\to p^*\cF$ is a descent datum with respect
to $f$. As $S\times_{f,BG,f} S$ is the sheaf $H$ on $S$ of $G$-automorphisms
of $T$, and $p_1 = p_2$ is the projection $p \colon H \to S$, $a$
corresponds to an action of $H$ on $\cF$. This action is trivial. Indeed,
this can be checked over geometric points $s \to S$, so we may assume that
$S$ is the spectrum of an algebraically closed field. In this case, $H\simeq
G$. As $p^*F$ is constant and $G$ is connected, and as the restriction of
$a$ to the unit section is the identity, $a$ is the identity.
\end{proof}

\begin{remark}\label{r.BG}
Corollary \ref{l.BG} implies that $f^*$ and $f_*$ are quasi-inverse to each
other and the natural transformations $\id_{\Sh_\cart(BG)} \to f_*f^*$,
$f^*f_*\to \id_{\Sh(S)}$ are natural isomorphisms. Since $f$ is a 2-section
of the projection $\pi \colon BG \to S$, we get natural isomorphisms
\[
\pi_*\simeq \pi_*f_*f^* \simeq f^*, \quad \pi^*\simeq f_*f^*\pi^* \simeq f_*.
\]
In particular, we have natural isomorphisms $f_*\pi_*\simeq \id$ and
$\pi^*f^* \simeq \id$.
\end{remark}

\begin{lemma}\label{l.tor}
Let $\cX$ be an Artin stack, let $\Lambda$ be a commutative ring, and let
$I\subset \Z$ be an interval. For $M\in D_\cart(\cX,\Lambda)$, the following
conditions are equivalent:
\begin{enumerate}
  \item For every $N\in \Mod_\cart(\cX,\Lambda)$,
      $\cH^q(M\otimes^L_\Lambda N)=0$ for all $q\in \Z-I$.

  \item For every finitely presented $\Lambda$-module $N$,
      $\cH^q(M\otimes^L_\Lambda N)=0$ for all $q\in \Z-I$.

  \item For every geometric point $i\colon x\to X$, $i^* M$ as an element
      of $D(x,\Lambda)$ is of tor-amplitude contained in $I$.
\end{enumerate}
\end{lemma}

If the conditions of the lemma are satisfied, we say $M$ is of
\emph{cartesian tor-amplitude} contained in $I$. If $M\in
D_{\cart}(\cX,\Lambda)$ has cartesian tor-amplitude contained in
$[a,+\infty)$ and $N\in D_\cart^{\ge b}(\cX,\Lambda)$, then
$M\otimes^L_\Lambda N$ is in $D^{\ge a+b}_\cart(\cX,\Lambda)$.
\begin{proof}
Obviously (a) implies (b) and (b) implies (c). Since the family of functors
$i^*\colon D_\cart(\cX,\Lambda)\to D(x,\Lambda)$ is conservative, where $i$
runs through all geometric points of $\cX$, (c) implies (a).
\end{proof}

\begin{prop}[Projection formula]\label{l.pf}
Let $f\colon \cX\to \cY$ be a morphism of Artin stacks and let $\Lambda$ be
a commutative ring. Let $L\in D_\cart(\cX,\Lambda)$, and let $K\in
D_\cart(\cY,\Lambda)$ such that $\cH^q K$ is constant for all $q$. Assume
one of the following:
\begin{enumerate}
\item $\Lambda$ is noetherian regular and $K\in D^+_c$, $L\in D^+$.

\item $\Lambda$ is noetherian and $K\in D^b_c(\Lambda)$ has finite
    cartesian tor-amplitude.

\item $Rf_*\colon D_{\cart}(\cX,\Lambda)\to D_{\cart}(\cY,\Lambda)$ has
    finite cohomological amplitude, $\Lambda$ is noetherian, $K\in D_c$,
    and either $K,L\in D^-$ or $L$ has finite cartesian tor-amplitude.

\item $f$ is quasi-compact quasi-separated, $\Lambda$ is annihilated by an
    integer invertible on $\cY$, $K\in D^+$, $L\in D^+$, and either
    $\Lambda$ is noetherian regular or $K$ has finite cartesian
    tor-amplitude.

\item $f$ is quasi-compact quasi-separated, $\Lambda$ is annihilated by an
    integer invertible on $\cY$, and $Rf_*\colon D_{\cart}(\cX,\Lambda)\to
    D_{\cart}(\cY,\Lambda)$ has finite cohomological amplitude.
\end{enumerate}
Then the map
\[K\otimes^L_\Lambda Rf_*L \to Rf_* (f^*K\otimes^L_\Lambda L)\]
induced by the composite map
\[f^*(K\otimes^L_\Lambda Rf_* L)\simto f^*K\otimes^L_\Lambda f^*Rf_*L\to f^*K\otimes^L_\Lambda L\]
is an isomorphism.
\end{prop}

\begin{proof}
In case (a), we may assume that $K$ is a (constant) $\Lambda$-module and we
are then in case (b). In case (b), we may assume that $\Lambda$ is local and
it then suffices to take a finite resolution of $K$ by finite projective
$\Lambda$-modules. In the first case of (c), we may assume $K$ is a constant
$\Lambda$-module. It then suffices to take a resolution of $K$ by finite
free $\Lambda$-modules. In the second case of (c), we reduce to the first
case of (c) using Corollary \ref{l.finitetor} below of the first case of
(c). In the first case of (d), we may assume $K\in D^b_c$ and we are in the
second case of (d). In the second case of (d), we may assume that $K$ is a
flat $\Lambda$-module, thus a filtered colimit of finite free
$\Lambda$-modules. Since $R^qf_*$ commutes with filtered colimits, we are
reduced to the trivial case where $K$ is a finite free $\Lambda$-module. In
case (e), since $Rf_*$ preserves small coproducts, we may assume that $L\in
D^-$ and $K$ is represented by a complex in $C^-(\Lambda)$ of flat
$\Lambda$-modules. We may further assume that $L\in D^b$ and $K$ is a flat
$\Lambda$-module. We are thus reduced to the second case of (d).
\end{proof}

\begin{cor}\label{l.finitetor}
Let $f\colon \cX\to \cY$ be a morphism of Artin stacks and let $\Lambda$ be
a noetherian commutative ring. Assume that the functor $Rf_*\colon
D_{\cart}(\cX ,\Lambda)\to D_\cart(\cY,\Lambda)$ has finite cohomological
amplitude. Then, for every $L\in D^-_\cart(\cX,\Lambda)$ of cartesian
tor-amplitude contained in $[a,+\infty)$, $Rf_*L$ has cartesian
tor-amplitude contained in $[a,+\infty)$.
\end{cor}

\begin{proof}
This follows immediately from the first case of Proposition \ref{l.pf} (c)
and Lemma \ref{l.tor}.
\end{proof}

The following statement on generic constructibility and generic base change
generalizes \cite[Theorem 9.10]{OlssonSh}.

\begin{prop}\label{l.gbc}
Let $\cZ$ be an Artin stack and let $f\colon \cX\to \cY$ be a morphism of
Artin stacks of finite type over $\cZ$. Let $\Lambda$ be a noetherian
commutative ring annihilated by an integer invertible on~$\cZ$, and let
$L\in D^+_c(\cX,\Lambda)$. Then for every integer $i$ there exists a dense
open substack $\cZ^\circ$ of $\cZ$ such that
\begin{enumerate}
  \item The restriction of $R^if_*L$ to $\cZ^\circ\times_\cZ \cY\subset
      \cY$ is constructible.

  \item $R^if_*L$ is compatible with arbitrary base change of Artin stacks
      $\cZ'\to\cZ^\circ\subset \cZ$.
\end{enumerate}
\end{prop}

\begin{proof}
  Recall first that for any 2-commutative diagram of Artin stacks of the form
  \[\xymatrix{\cX''\ar[r]^{h'}\ar[d]_{f''} & \cX'\ar[r]^{g'}\ar[d]^{f'} & \cX\ar[d]^f\\
  \cY''\ar[r]^{h} & \cY'\ar[r]^g & \cY}\]
  the following diagram commutes:
  \begin{equation}\label{e.bcdiag}
  \xymatrix{(gh)^* R f_* L\ar[rr]^{b_{gh}}\ar[d]_\simeq && Rf''_*(g'h')^* L\ar[d]^\simeq\\
  h^*g^*Rf_*L \ar[r]^{h^*b_{g}} & h^* Rf'_* {g'}^* L\ar[r]^{b_h} & Rf''_* {h'}^* {g'}^* L}
  \end{equation}
where $b_{gh}$, $b_g$, $b_h$ are base change maps.

If $\cZ$ is a scheme, then, as in \cite[Theorem 9.10]{OlssonSh},
cohomological descent and the case of schemes \cite[Th.\ finitude
1.9]{SGA4d} imply that there exists a dense open subscheme $\cZ^\circ$ of
$\cZ$ such that (a) holds and that $R^if_*L$ is compatible with arbitrary
base change of schemes $Z'\to \cZ^\circ\subset \cZ$. This implies (b). In
fact, for any base change of Artin stacks $g\colon \cZ'\to \cZ^\circ
\subset \cZ$, take a smooth atlas $p\colon Z'\to \cZ'$ where $Z'$ is a
scheme. Then $b_{p}$ is an isomorphism and $b_{gp}$ is an isomorphism by
assumption. It follows that $p^*b_g$ and hence $b_g$ are isomorphisms.

In the general case, let $p\colon Z\to \cZ$ be a smooth atlas. By the
preceding case, there exists a dense open subscheme $Z^\circ \subset Z$ such
that after forming the 2-commutative diagram with 2-cartesian squares
  \[\xymatrix{\cX_Z\ar[d]_{p_\cX}\ar[r]^{f_Z} & \cY_Z\ar[d]^{p_\cY}\ar[r] &Z\ar[d]^p\\
  \cX\ar[r]^f & \cY\ar[r] & \cZ}
  \]
the restriction of $R^i f_{Z*} p_\cX^* L$ to $Z^\circ\times_Z \cY_Z$ is
constructible and that $R^i f_{Z*} p_\cX^* L$ commutes with arbitrary base
change of Artin stacks $\cW\to Z^\circ\subset Z$. We claim that
$\cZ^\circ=p(Z^\circ)$ satisfies (a) and (b). To see this, let $p^\circ
\colon Z^\circ \to \cZ^\circ$ be the restriction of $p$. By definition
$p^\circ$ is surjective. Then (a) follows from the fact that
  \[p_\cY^{\circ*}(R^if_*L|\cZ^\circ\times_\cZ \cY)\simeq R^i f_{Z*} p_\cX^* L | Z^\circ\times_Z \cY_Z\]
  is constructible. For any base change of Artin stacks $\cZ'\to \cZ^\circ$, form the following 2-cartesian square:
  \[\xymatrix{Z'\ar[r]^{h}\ar[d]_{p'} & Z^\circ\ar[d]^{p^\circ} \\
  \cZ'\ar[r]^g & \cZ^\circ.}
  \]
By \eqref{e.bcdiag}, $b_{p'}({p'}^* b_g)$ can be identified with $b_{h} (h^*
b_{p^\circ})$. Since $p^\circ$ and $p'$ are smooth, $b_{p^\circ}$ and
$b_{p'}$ are isomorphisms. By the construction of $p^\circ$, $b_h$ is an
isomorphism. It follows that ${p'}^* b_g$ and hence $b_g$ are isomorphisms.
\end{proof}

\begin{remark}\label{r.gbc}
For $\cZ=BG$, where $G$ is an algebraic group over a field $k$, $f\colon
\cX\to \cY$ a quasi-compact and quasi-separated morphism of Artin stacks
over $\cZ$, and $\Lambda$ is a commutative ring annihilated by an integer
invertible in $k$, the above proof combined with the remark following
\cite[Th.\ finitude 1.9]{SGA4d} shows that $R f_* \colon
D^+_\cart(\cX,\Lambda)\to D^+_\cart(\cY,\Lambda)$ commutes with arbitrary
base change of Artin stacks $\cZ'\to \cZ$.
\end{remark}

\section{Multiplicative structures in derived categories}\label{s.8}

\begin{definition}
For us, a $\otimes$-\emph{category} is a symmetric monoidal category
\cite[Section VII.7]{MacLane}, that is, a category $\cT$ endowed with a
bifunctor $\otimes\colon \cT\times \cT \to \cT$, a unit object $\bone$ and
functorial isomorphisms
\begin{gather*}
  a_{LMN}\colon L\otimes(M\otimes N) \to (L\otimes M) \otimes N,\\
  c_{MN}\colon M\otimes N \to N\otimes M,\\
  u_M\colon M\otimes \bone \to M, \quad v_M \colon \bone \otimes M \to M,
\end{gather*}
satisfying the axioms of \loccit. We define a \emph{pseudo-ring} in~$\cT$ to
be an object $K$ of $\cT$ endowed with a morphism $m\colon K\otimes K \to K$
such that the following associativity diagram commutes:
\[
\xymatrix{& K \otimes (K\otimes K) \ar[d]_{a_{KKK}}\ar[r]^-{\id_K \otimes m} & K\otimes K \ar[rd]^{m}\\
 & (K\otimes K)\otimes K \ar[r]^-{m\otimes \id_K} & K\otimes K\ar[r]^{m} &K.}
\]
A pseudo-ring $(K,m)$ is called \emph{commutative} if the following diagram
commutes
\[
\xymatrix{K\otimes K\ar[rd]^{m}\ar[d]_{c_{KK}} \\
K\otimes K\ar[r]^{m} & K.}
\]
A \emph{homomorphism of pseudo-rings} $(K, m) \to (K',m')$ is a morphism
$f\colon K\to K'$ of~$\cT$ such that the following diagram commutes:
\[
\xymatrix{K\otimes K\ar[r]^-{m}\ar[d]_{f\otimes  f} &K\ar[d]^{f}\\
K'\otimes K' \ar[r]^-{m'}& K'.}
\]
We define a \emph{left $(K,m)$-pseudomodule} to be an object $M$ of $\cT$
endowed with a morphism $n\colon K\otimes M\to M$ such that the following
diagram commutes
\[\xymatrix{K\otimes (K\otimes M)\ar[r]^-{\id_K\otimes n}\ar[d]_{a_{KKM}} & K\otimes M \ar[rd]^{n}\\
(K\otimes K)\otimes M\ar[r]^-{m\otimes \id_M} & K\otimes M \ar[r]^{n} & M.}
\]
A \emph{homomorphism of left $(K,m)$-pseudomodules} $(M,n)\to (M',n')$ is a
morphism $h\colon M\to M'$ of~$\cT$ such that the following diagram commutes
\[\xymatrix{K\otimes M\ar[r]^n\ar[d]_{\id_K\otimes h} & M\ar[d]^h\\
K\otimes M'\ar[r]^{n'} & M'.}\]
\end{definition}

\begin{definition}\label{r.ringhom}
Let $f\colon (K,m)\to (K',m')$ be a homomorphism of pseudo-rings. We define
a \emph{splitting of $f$} to be a morphism $n\colon K'\otimes K\to K$,
making $K$ into a $(K',m')$-pseudomodule and such that the following diagram
commutes
\[\xymatrix{K\otimes K \ar[rd]_m\ar[r]^{f\otimes \id_K} &\ar[d]_{n} K'\otimes K\ar[r]^{\id_{K'}\otimes f}& K'\otimes K'\ar[d]^{m'}\\
&K\ar[r]^f & K'.}
\]
\end{definition}

\begin{definition}\label{s.ring}
We define a \emph{ring}  in $\cT$ to be a pseudo-ring $(K,m)$ in $\cT$
endowed with a morphism $e\colon \bone\to K$ such that the following
diagrams commute:
\[\xymatrix{K\otimes \bone \ar[r]^{\id_K\otimes e}\ar[rd]_{u_K}& K\otimes K\ar[d]^m& \bone\otimes K\ar[r]^{e\otimes \id_K}\ar[rd]_{v_K} & K\otimes K\ar[d]^{m}\\
&K&&K.}
\]
(Thus a ring in our sense is a ``monoid'' in the terminology of
\cite[Section VII.3]{MacLane}.) The unit $\bone$ endowed with
$u_{\bone}\colon \bone\otimes \bone \to \bone$ and $\id_{\bone}\colon
\bone\to \bone$ is a commutative ring in $\cT$. A \emph{ring homomorphism}
$(K,m,e)\to (K',m',e')$ is a homomorphism of pseudo-rings $f\colon (K,m)\to
(K',m')$ such that the following diagram commutes:
\[\xymatrix{\bone\ar[r]^{e}\ar[rd]_{e'} &K \ar[d]^{f}\\
&K'.}
\]
A \emph{left $(K,m,e)$-module} is a left $(K,m)$-pseudomodule $(M,n)$ such
that the following diagram commutes
\[\xymatrix{\bone \otimes M\ar[rd]_{v_M}\ar[r]^{e\otimes M} & K\otimes M\ar[d]^n\\
&M.}
\]
A \emph{homomorphism of left $(K,m,e)$-modules} $(M,n)\to (M',n')$ is a
homomorphism between the underlying left $(K,m)$-pseudomodules.
\end{definition}

\begin{construction}\label{s.tensor}
Let $\cT=(\cT,\otimes, a,c,u,v)$ and $\cT'=(\cT',\otimes,a',c',u',v')$ be
$\otimes$-categories, and let $\omega\colon \cT\to \cT'$ be a functor. A
\emph{left-lax $\otimes$-structure} on $\omega$ is a natural transformation
of functors $\cT\times \cT\to \cT'$ consisting of morphisms of $\cT'$
\[
o_{MN}\colon \omega(M\otimes N)\to\omega(M)\otimes \omega(N),
\]
such that the following diagrams commute:
\begin{gather*}
\xymatrix@C=4em{
\omega(L\otimes(M\otimes N)) \ar[r]^{o_{L,M\otimes N}}\ar[d]_{\omega(a_{LMN})} & \omega(L)\otimes \omega(M\otimes N) \ar[r]^-{\omega(L)\otimes o_{MN}} & \omega(L)\otimes (\omega(M) \otimes \omega(N))\ar[d]^{a'_{\omega(L)\omega(M)\omega(N)}}\\
\omega((L\otimes M)\otimes N)\ar[r]^{o_{L\otimes M,N}} & \omega(L\otimes M)\otimes \omega(N)\ar[r]^-{o_{LM}\otimes \omega(N)} & (\omega(L)\otimes \omega(M)) \otimes \omega(N)}
\\
\xymatrix{\omega(M\otimes N)\ar[r]^{o_{MN}}\ar[d]_{\omega(c_{MN})} &\omega(M)\otimes \omega(N)\ar[d]^{c'_{\omega(M)\omega(N)}}\\
\omega(N\otimes M)\ar[r]^{o_{NM}} & \omega(N)\otimes \omega(M).}
\end{gather*}
A \emph{right-lax $\otimes$-structure} on $\omega$ is a left-lax
$\otimes$-structure on $\omega^\op\colon \cT^\op\to {\cT'}^\op$. It is given
by functorial morphisms
\[
t_{MN}\colon \omega(M)\otimes \omega(N) \to \omega(M\otimes N),
\]
such that the above diagrams with arrows $o$ inverted and replaced by $t$
commute. A \emph{$\otimes$-structure} on $\omega$ is a left-lax
$\otimes$-structure $o$ such that $o_{MN}$ is an isomorphism for all $M$ and
$N$. In this case $t_{MN}=o_{MN}^{-1}$ defines a right-lax
$\otimes$-structure. If $t$ is a right-lax $\otimes$-structure on $\omega$
and $(K, m)$ is a pseudo-ring in $\cT$, we endow $\omega(K)$ with the
pseudo-ring structure
\[\omega(K)\otimes \omega(K)\xto{t_{KK}} \omega(K\otimes K) \xto{\omega(m)} \omega(K).\]
If, moreover, $(M,n)$ is a left $(K,m)$-pseudomodule, we endow $\omega(M)$
with the left $\omega(K)$-pseudomodule structure
\[\omega(K)\otimes \omega (M)\xto{t_{KM}}\omega(K\otimes M)\xto{\omega(n)}\omega(M).\]
If $(K,m)$ is commutative, then $\omega(K)$ is commutative. This
construction sends homomorphisms of pseudo-rings to homomorphisms of
pseudo-rings and homomorphisms of left pseudomodules to homomorphisms of
left pseudomodules.

If $(\omega,t)$, $(\omega',t')$ are functors endowed with right-lax
$\otimes$-structures, we say that a natural transformation $\alpha\colon
\omega\to \omega'$ preserves the right-lax $\otimes$-structures if the
following diagram commutes
\[\xymatrix{\omega(M)\otimes \omega(N)\ar[r]^{t_{MN}}\ar[d]_{\alpha_M\otimes \alpha_N} & \omega(M\otimes N)\ar[d]^{
\alpha_{M\otimes N}} \\
\omega'(M)\otimes \omega'(N)\ar[r]^{t'_{MN}} & \omega'(M\otimes N).}
\]
In
this case, for any pseudo-ring $K$ in $\cT$, $\alpha_K\colon \omega(K)\to
\omega'(K)$ is a homomorphism of pseudo-rings.
\end{construction}

\begin{construction}\label{s.adj}
Now suppose that $\omega\colon \cT\to \cT'$ admits a right adjoint
$\tau\colon \cT' \to \cT$. For any left-lax $\otimes$-structure $o$ on
$\omega$, endow $\tau$ with the right-lax $\otimes$-structure $t$ such that
$t_{MN}\colon \tau(M)\otimes \tau(N) \to \tau(M\otimes N)$ is adjoint to the
composition
\[\omega(\tau(M)\otimes \tau(N))\xto{o_{\tau(M)\tau(N)}} \omega(\tau(M))\otimes \omega(\tau(N))\xto{\alpha_M\otimes \alpha_N} M\otimes N,\]
where $\alpha_M\colon \omega(\tau(M)) \to M$, $\alpha_N\colon
\omega(\tau(N)) \to N$ are adjunction morphisms. It is straightforward to
check that this construction defines a bijection from the set of left-lax
$\otimes$-structures on $\omega$ to the set of right-lax
$\otimes$-structures on $\tau$.

In the above construction, if $o$ is a $\otimes$-structure on $\omega$, then
the adjunction morphisms $\alpha\colon \omega\tau\to \id_{\cT'}$ and $\beta
\colon \id_{\cT}\to \tau\omega$ preserve the resulting right-lax
$\otimes$-structures.
\end{construction}

\begin{construction}
This formalism has a unital variant. A \emph{left-lax unital
$\otimes$-structure} on a functor $\omega\colon \cT\to \cT'$ is a left-lax
$\otimes$-structure endowed with a morphism $p\colon \omega(\bone)\to
\bone'$ in $\cT'$ such that the following diagrams commute
\[\xymatrix{\omega(M\otimes \bone)\ar[r]^{o_{M\bone}}\ar[d]_{\omega(u_M)}& \omega(M)\otimes \omega(\bone)\ar[d]^{\id_{\omega(M)}\otimes p}
& \omega(\bone \otimes M)\ar[r]^{o_{\bone M}}\ar[d]_{\omega(v_M)} & \omega(\bone)\otimes \omega(M)\ar[d]^{p\otimes \id_{\omega(M)}}\\
\omega(M) & \omega(M)\otimes \bone'\ar[l]^{u'_{\omega(M)}} & \omega(M) &
\bone'\otimes \omega(M)\ar[l]^{v'_{\omega(M)}}}
\]
A \emph{right-lax unital $\otimes$-structure} is a left-lax unital
$\otimes$-structure on $\omega^\op\colon \cT^\op\to {\cT'}^\op$. It consists
of a right-lax $\otimes$-structure endowed with a morphism $s\colon
\bone'\to \omega(\bone)$ in $\cT'$ such that the above diagrams, with arrows
$o$ inverted and replaced by $t$, arrows $p$ inverted and replaced by $s$,
commute. A \emph{unital $\otimes$-structure} is a left-lax unital
$\otimes$-structure $(o,p)$ such that $o$ is a $\otimes$-structure and $p$
is invertible. Constructions \ref{s.tensor} and \ref{s.adj} can be carried
over to the unital case.

Let $\cT$ be a $\otimes$-category, and let $\cC$ be a category. Then the
category $\cT^\cC$ of functors $\cC\to \cT$ has a natural
$\otimes$-structure. The constant functor $\cT\to \cT^\cC$ defined by
$M\mapsto (M)_\cC$ has a natural unital $\otimes$-structure.
\end{construction}

\begin{construction}\label{s.rtopos}
Let $X=(X,\cO_X)$ be a commutatively ringed topos. Two $\otimes$-categories
will be of interest to us:
\begin{enumerate}
\item The (unbounded) derived category $D(X)=D(X,\cO_X)$, equipped with
    $\otimes^L_{\cO_X}\colon D(X)\times D(X)\to D(X)$ \cite[Theorem
    18.6.4]{KashiwaraSCat}.

\item The category $\grmod(X)=\grmod(X,\cO_X)$ of graded $\cO_X$-modules
    $H=\bigoplus_{n\in \Z}H^n$, with $\otimes$ given by $(H\otimes
    K)^n=\bigoplus_{i+j=n}H^i\otimes_{\cO_X} K^j$, the isomorphism
    $c\colon H\otimes K\to K\otimes H$ being given by the usual sign rule.
\end{enumerate}

The cohomology functor
  \[\cH^* \colon D(X) \to \grmod(X)\]
has a natural right-lax unital $\otimes$-structure given by the canonical
maps $\cH^*L\otimes \cH^*M\to \cH^*(L\otimes^L M)$. (This is a unital
$\otimes$-structure when $\cO_X$ is a constant field, which is the case we
are mostly interested in).

Let $f\colon X=(X,\cO_X)\to Y=(Y,\cO_Y)$ be a morphism of commutatively
ringed topoi. We endow $f^*\colon \grmod(Y)\to \grmod(X)$ with the unital
$\otimes$-structure defined by the functorial isomorphisms
  \[f^*(M\otimes N)\to f^*M\otimes f^* N, \quad f^*\cO_Y\to \cO_X. \]
We endow $Lf^*\colon D(Y)\to D(X)$ \cite[Theorem 18.6.9]{KashiwaraSCat} with
the unital $\otimes$-structure defined by the functorial isomorphisms
  \[Lf^*(M\otimes^L N)\to Lf^*M\otimes^L f^* N, \quad Lf^*\cO_Y\to \cO_X. \]
We endow the right adjoint functors $f_*\colon \grmod(X)\to \grmod(Y)$ and
$Rf_*\colon D(X)\to D(Y)$ with the induced right-lax unital
$\otimes$-structures.
\end{construction}

\begin{construction}\label{c.Dring}
Let $\cX$ be an Artin stack, and let $\Lambda$ be a commutative ring. We
consider the $\otimes$-categories $D_\cart(\cX,\Lambda)$ and
$\grmod_\cart(\cX,\Lambda)$, the category of graded cartesian sheaves of
$\Lambda$-modules.

Let $f\colon \cX\to \cY$ be a morphism of Artin stacks. As in Construction
\ref{s.rtopos}, we endow the functors $f^*\colon
\grmod_\cart(\cY,\Lambda)\to \grmod_\cart(\cX,\Lambda)$ and $f^*\colon
D_\cart(\cY,\Lambda)\to D_\cart(\cX,\Lambda)$ with the natural unital
$\otimes$-structures. We endow the right adjoint functors $f_*\colon
\grmod_\cart(\cY,\Lambda)\to \grmod_\cart(\cX,\Lambda)$ and $Rf_*\colon
D_\cart(\cX,\Lambda)\to D_\cart(\cY,\Lambda)$ with the induced right-lax
unital $\otimes$-structures.

Assume that $\Lambda$ is annihilated by an integer $n$ invertible on $\cY$
and $f$ is locally of finite presentation. Then we have $Rf^!\colon
D_\cart(\cY,\Lambda)\to D_\cart(\cX,\Lambda)$. As in \cite[Cycle
(1.2.2.3)]{SGA4d}, for $M$ and $N$ in $D_\cart(\cY,\Lambda)$, we have a
morphism
\[f^*M\otimes^L Rf^!N \to Rf^!(M\otimes^L N)\]
given by the morphism $Rf^!N\to Rf^!R\cHom(M,M\otimes^L N)\simeq R\cHom(f^*
M,Rf^!(M\otimes^L N))$. For a pseudo-ring $(L,m)$ in $D_\cart(\cY,\Lambda)$,
we endow $Rf^! L$ with the left $f^* L$-pseudomodule structure given by the
composition
\[f^* L\otimes^L Rf^! L \to Rf^!(L\otimes^L L) \xto{Rf^! m} Rf^!L\]

Assume moreover that $f=i$ is a closed immersion. Then the right-lax
$\otimes$-structure on $i_*=Ri_*$ is an isomorphism and its inverse is a
$\otimes$-structure consisting of a functorial isomorphism
\[i_*(M\otimes^L N)\to i_*M\otimes^L i_*N.\]
We endow the right adjoint functor $Ri^!$ of $i_*$ with the induced
right-lax $\otimes$-structure. Note that the right \emph{unital}
$\otimes$-structure on $i_*$ is not invertible in general. For a pseudo-ring
$(L,m)$ in $D_\cart(\cY,\Lambda)$, the above left $i^*L$-pseudomodule
structure on $Ri^!L$ is a splitting of the homomorphism of pseudo-rings
$Ri^! L \to i^*L$ (Definition \ref{r.ringhom}).
\end{construction}

In the rest of this section, we discuss multiplicative structures on
spectral objects. We will only consider spectral objects of type $\tilde
\Z$, where $\tilde\Z$ is the category associated to the ordered set $\Z\cup
\{\pm \infty\}$.

\begin{definition}
Let $\cT$ be category endowed with a bifunctor $\otimes\colon \cT\times
\cT\to \cT$. Let $J$ be a category endowed with a bifunctor $*\colon J\times
J\to J$. Let $X,X',X''$ be functors $J\to \cT$. A \emph{pairing} from
$X$,$X'$ to $X''$ is a natural transformation of functors $J\times J\to \cT$
consisting of morphisms of $\cT$
\[X(j)\otimes X'(j')\to X''(j*j').\]

Assume moreover that $(\cT,\otimes)$ and $(J,*)$ are endowed with structures
of $\otimes$-categories. A pairing from $X$, $X$ to $X$ is called
\emph{associative} if for $j,j',j''\in J$, the following diagram commutes
\[\xymatrix{X(j)\otimes (X(j')\otimes X(j''))\ar[r]\ar[d]_a &
X(j)\otimes X(j'*j'')\ar[r] & X(j*(j'*j''))\ar[d]^a\\
(X(j)\otimes X(j'))\otimes X(j'')\ar[r] & X(j*j')\otimes X(j'')\ar[r] & X((j*j')*j''),
}
\]
and is called \emph{commutative} if for $j,j'\in J$, the following diagram
commutes
\begin{equation}\label{e.pcomm}
\xymatrix{X(j)\otimes X(j')\ar[r]\ar[d]_c&X(j*j')\ar[d]^c\\
X(j')\otimes X(j)\ar[r] & X(j'*j).}
\end{equation}

Assume moreover that $\cT$ is additive and $\otimes$ is an additive
bifunctor. Let $\SG$ be the $\otimes$-category given by the discrete
category $\{\pm 1\}$ and the ordinary product. Let $\sigma\colon J\to \SG$
be a $\otimes$-functor. A pairing from $X$, $X$ to $X$ is called
\emph{$\sigma$-commutative} if for $j,j'\in J$, the diagram \eqref{e.pcomm}
is $\max\{\sigma(j),\sigma(j')\}$-commutative.
\end{definition}

\begin{construction}
Let $\Ar(\tilde \Z)$ be the category of morphisms of $\tilde \Z=\Z\cup\{\pm
\infty\}$. We represent objects of $\Ar(\tilde\Z)$ by pairs $(p,q)$, $p,q\in
\tilde\Z$, $p\le q$. We endow $\Ar(\tilde \Z)$ with a structure of
$\otimes$-category by the formula
\[
(p,q)*(p',q')=(\max\{p+q'-1,p'+q-1\},q+q'-1).
\]
Here we adopt the convention that
$(-\infty)+(+\infty)=-\infty=(+\infty)+(-\infty)$.
\end{construction}

\begin{definition}
Let $\cD$ be a triangulated category endowed with a triangulated bifunctor
$\otimes\colon \cD\times \cD\to \cD$ \cite[Definition
10.3.6]{KashiwaraSCat}. Let $(X,\delta)$, $(X',\delta)$, $(X'',\delta'')$ be
spectral objects with values in $\cD$ \cite[II 4.1.2]{Verdier}. A
\emph{pairing} from $(X,\delta)$, $(X',\delta')$ to $(X'',\delta'')$
consists of a pairing from $X$, $X'$ to $X''$, namely a natural
transformation of functors $\Ar(\tilde \Z) \times \Ar(\tilde \Z) \to \cD$
consisting of morphisms of $\cD$
\[X(p,q)\otimes X'(p',q')\to X''((p,q)*(p',q')),\]
such that for $p\le q\le r$, $p'\le q'\le r'$ in $\tilde\Z$ satisfying
$q+r'=q'+r$ and $p+r'=p'+r$, the diagram
\[\xymatrix{X(q,r)\otimes X'(q',r')\ar[r]\ar[d]_{(\delta\otimes \id, \id\otimes \delta')} &
X''(q'',r'')\ar[d]^{\delta''}\\
(X(p,q)[1]\otimes X'(q',r'))\oplus (X(q,r)\otimes X'(p',q')[1])\ar[r] & X''(p'',q'')[1]}
\]
commutes. Here $(q'',r'')=(q,r)*(q',r')$,
$(p'',q'')=(p,q)*(q',r')=(q,r)*(p',q')$.

Assume moreover that $(\cD,\otimes)$ is endowed with a structure of
$\otimes$-category\footnote{Here we do not assume that the constraints of
the $\otimes$-category are natural transformations of triangulated functors
\cite[Definition 10.1.9 (ii)]{KashiwaraSCat} in each variable.}. A pairing
from $(X,\delta)$, $(X,\delta)$ to $(X,\delta)$ is called \emph{associative}
(resp.\ \emph{commutative}) if the underlying pairing from $X$, $X$ to $X$
is.
\end{definition}

\begin{example}\label{e.pairso}
Let $X$ be a commutatively ringed topos, and let $K,K',K''\in D(X)$. We
consider the second spectral object $(K,\delta)$ associated to $K$ \cite[III
4.3.1, 4.3.4]{Verdier}, with $K(p,q)=\tau^{[p,q-1]}K$, where
$\tau^{[p,q-1]}$ is the canonical truncation functor. Similarly, we have
spectral objects $(K',\delta')$, $(K'',\delta'')$. A map $K\otimes^L K'\to
K''$ in $D(X)$ defines a pairing from $(K,\delta)$, $(K',\delta')$ to
$(K'',\delta'')$ given by
\begin{multline*}
\tau^{[p,q-1]}K\otimes^L \tau^{[p',q'-1]}K'
\to \tau^{\ge p''} (\tau^{[p,q-1]}K\otimes^L \tau^{[p',q'-1]}K')
\simeq \tau^{\ge p''} (\tau^{\le q-1} K\otimes^L \tau^{\le q'-1}K')\\
\xrightarrow{\alpha} \tau^{[p'',q''-1]}(K\otimes^L K')\to \tau^{[p'',q''-1]}K'',
\end{multline*}
where $(p'',q'')=(p,q)*(p',q')$, $\alpha$ is given by the map $\tau^{\le
q-1}K\otimes^L \tau^{\le q'-1}K'\to \tau^{\le q''-1}(K\otimes^L K'')$
induced by adjunction from the map $\tau^{\le q-1}K\otimes^L \tau^{\le
q'-1}K'\to K\otimes^L K'$. Moreover, if $K$ is a pseudo-ring (resp.\
commutative pseudo-ring), then the induced pairing from $(K,\delta)$,
$(K,\delta)$ to $(K,\delta)$ is associative (resp.\ commutative).

The above also holds with $D(X)$ replaced by $D_\cart(\cX,\Lambda)$, where
$\cX$ is an Artin stack and $\Lambda$ is a commutative ring.
\end{example}

\begin{definition}
Let $\cA$ be an abelian category endowed with an additive bifunctor
$\otimes\colon \cA\times \cA\to \cA$. Let $(H^n,\delta)$, $(H'^n,\delta')$,
$(H''^n,\delta'')$ be spectral objects with values in $\cA$ \cite[II
4.1.4]{Verdier}. A \emph{pairing} from $(H^n,\delta)$, $(H'^n,\delta')$ to
$(H''^n,\delta'')$ consists of a pairing from $H^*$, $H'^*$ to $H''^*$,
namely a natural transformation of functors $(\Z\times \Ar(\tilde \Z))\times
(\Z\times \Ar(\tilde \Z)) \to \cA$ consisting of morphisms of $\cA$
\[H^n(p,q)\otimes H'^{n'}(p',q')\to H''^{n+n'}((p,q)*(p',q')),\]
such that for $p\le q\le r$, $p'\le q'\le r'$ in $\tilde\Z$ satisfying
$q+r'=q'+r$ and $p+r'=p'+r$, the diagram
\[\xymatrix{H^n(q,r)\otimes H'^{n'}(q',r')\ar[r]\ar[d]_{(\delta\otimes \id, (-1)^n\id\otimes \delta')} &
H''^{n+n'}(q'',r'')\ar[d]^{\delta''}\\
(H^{n+1}(p,q)\otimes H'^{n'}(q',r'))\oplus (H^n(q,r)\otimes H'^{n'+1}(p',q'))\ar[r] & H''^{n+n'+1}(p'',q'')}
\]
commutes. Here $(q'',r'')=(q,r)*(q',r')$,
$(p'',q'')=(p,q)*(q',r')=(q,r)*(p',q')$. Note that if $(H^n,\delta)$,
$(H'^n,\delta')$, and $(H''^n,\delta'')$ are stationary \cite[II
4.4.2]{Verdier}, then the pairing from $H^*$, $H'^*$ to $H''^*$ is uniquely
determined by the pairing from $H^*\res \Arm$, $H'^*\res \Arm$ to $H''^*\res
\Arm$, where $\Arm=\Ar(\Z\cup \{-\infty\})$. In fact, in this case, for
every $n$, there exists an integer $u(n)$ such that for every $q\ge u(n)$,
the morphism $H^n(-\infty,q)\to H^n(-\infty,+\infty)$ is an isomorphism.

Consider the induced spectral sequences $(E_2^{pq}\Rightarrow H^n)$,
$(E'^{pq}_2\Rightarrow H'^n)$, $(E''^{pq}_2\Rightarrow H''^n)$ given by
\cite[II (4.3.3.2)]{Verdier}. A pairing from $(H^n,\delta)$,
$(H'^n,\delta')$ to $(H''^n,\delta'')$ induces compatible pairings of
differential bigraded objects of $\cA$
\[E_r^{pq}\otimes E'^{p'q'}_r\to E''^{p+p',q+q'}_r\]
for $2\le r\le \infty$ (satisfying $d''_r(xy)=d_r(x)y+(-1)^{p+q}xd'_r(y)$
for $x\in E_r^{pq}$, $y\in E'^{p'q'}_r$) and a pairing of
filtered\footnote{For the filtration, we use the convention
$F^pH^n=\Img(H^n(-\infty,n-p+1)\to H^n(-\infty,\infty))$. In particular, in
Example \ref{e.ssmult} below, $F^pH^n(X,K)=\Img (H^n(X,\tau^{\le n-p}K)\to
H^n(X,K))$.} graded objects of $\cA$
\[F^pH^n \otimes F^{p'}H'^{n'} \to F^{p+p'}H''^{n+n'},\]
compatible with the pairing on $E_\infty$.

Assume moreover that $(\cA,\otimes)$ is endowed with a structure of
$\otimes$-category. A pairing from $(H^n,\delta)$, $(H^n,\delta)$ to
$(H^n,\delta)$ is called \emph{associative} (resp.\ \emph{commutative}) if
the underlying pairing from $H^*$, $H^*$ to $H^*$ is associative (resp.\
$\sigma$-commutative, where $\sigma \colon \Z\times\Ar(\tilde \Z)\to \SG$ is
given by $(n,(p,q))\mapsto (-1)^n$). An associative (resp.\ commutative)
pairing from $(H^n,\delta)$, $(H^n,\delta)$ to $(H^n,\delta)$ induces
associative (resp.\ commutative) pairings on $E_r^{pq}$ and $F^pH^n$. Here
the commutativity for $E_r^{pq}$ and $F^pH^n$ are relative to the functors
$\Z\times \Z\to \SG$ given by $(p,q)\mapsto (-1)^{p+q}$ and $(p,n)\mapsto
(-1)^n$, respectively.
\end{definition}

\begin{remark}\label{r.pairfun}
Let $\cD$, $\cD'$ be triangulated categories endowed with triangulated
bifunctors $\otimes\colon \cD\times \cD\to \cD$, $\otimes\colon \cD'\times
\cD'\to \cD'$. Let $\tau\colon \cD\to \cD'$ be a triangulated functor
endowed with a natural transformation of functors $\cD\times \cD\to \cD'$
consisting of morphisms $\tau(M)\otimes \tau(N)\to \tau(M\otimes N)$ of
$\cD'$ that is a natural transformation of triangulated functors in each
variable. Let $(X,\delta)$, $(X',\delta')$, $(X'',\delta'')$ be spectral
objects with values in $\cD$. Then a pairing from $(X,\delta)$,
$(X',\delta')$ to $(X'',\delta'')$ induces a pairing from $\tau(X,\delta)$,
$\tau(X',\delta')$ to $\tau(X'',\delta'')$. If $(\cD,\otimes)$,
$(\cD',\otimes)$ are endowed with structures of $\otimes$-categories and
$\tau$ is a right-lax $\otimes$-functor (Construction \ref{s.tensor}), then
an associative (resp.\ commutative) pairing from $(X,\delta)$, $(X,\delta)$
to $(X,\delta)$ induces an associative (resp.\ commutative) pairing from
$\tau(X,\delta)$, $\tau(X,\delta)$ to $\tau(X,\delta)$.

Similarly, let $\cA$ be an abelian category endowed with an additive
bifunctor $\otimes \colon\cA\times \cA\to \cA$ and let $H\colon \cD\to \cA$
be a cohomological functor endowed with a natural transformation of functors
$\cD\times \cD\to \cA$ consisting of morphisms $H(M)\otimes H(N)\to
H(M\otimes N)$ of $\cA$. We adopt the convention that for $p\le q\le r$ in
$\tilde \Z$, the map $\delta^n\colon H^n(X(q,r))\to H^{n+1}(X(p,q))$ is
$(-1)^n$ times the map obtained by applying $H$ to $\delta[n]\colon
X(q,r)[n]\to X(p,q)[n+1]$. Then a pairing from $(X,\delta)$, $(X',\delta')$
to $(X'',\delta'')$ induces a pairing from $H^*(X,\delta)$,
$H^*(X',\delta')$ to $H^*(X'',\delta'')$ given by
\begin{multline*}
H(X(p,q)[n])\otimes H(X'(p',q')[n']) \to H(X(p,q)[n]\otimes X'(p',q')[n'])\\
\simeq H((X(p,q)\otimes X'(p',q'))[n+n'])
\to H(X''((p,q)*(p',q'))[n+n']).
\end{multline*}
Here we have used the composite of the isomorphisms
\[M[m]\otimes N[n]\simeq (M\otimes N[n])[m]\simeq (M\otimes N)[m+n]\]
given by the structure of bifunctor of additive categories with translation
\cite[Definition 10.1.1 (v)]{KashiwaraSCat} on $\otimes\colon \cD\times
\cD\to \cD$. If $(\cD,\otimes)$, $(\cA,\otimes)$ are endowed with structures
of $\otimes$-categories and $H$ is a right-lax $\otimes$-functor, and if the
associativity (resp.\ commutativity) constraint of $(\cD,\otimes)$ is a
natural transformation of triangulated functors in each variable, then an
associative (resp.\ commutative) pairing from $(X,\delta)$, $(X,\delta)$ to
$(X,\delta)$ induces an associative (resp.\ commutative) pairing from
$H^*(X,\delta)$, $H^*(X,\delta)$ to $H^*(X,\delta)$. Indeed, the assumption
on the commutativity constraint implies the $(-1)^{mn}$-commutativity of the
following diagram
\[\xymatrix{M[m]\otimes N[n]\ar[r]^\sim\ar[d]_\simeq &
(M\otimes N[n])[m]\ar[r]^\sim & (M\otimes N)[m+n]\ar[d]^\simeq\\
N[n]\otimes M[m]\ar[r]^\sim & (N\otimes M[m])[n] \ar[r]^\sim & (N\otimes M)[m+n].}
\]
\end{remark}

\begin{example}\label{e.ssmult}
Let $X$ be a commutatively ringed topos and let $K$ be an object of $D(X)$.
The second spectral sequence of hypercohomology
\[E_2^{pq}=H^p(X,\cH^q K)\Rightarrow H^{p+q}(X,K)\]
is induced from the spectral object $H^*(K,\delta)$, where $(K,\delta)$ is
the second spectral object associated to $K$. If $K$ is a pseudo-ring in
$D(X)$, then Remark \ref{r.pairfun} applied to Example \ref{e.pairso} endows
the spectral sequence with an associative multiplicative structure, which is
graded commutative when $K$ is commutative.
\end{example}

\part{Main results}\label{p.2}
\section{Finiteness theorems for equivariant cohomology rings}\label{s.2}

We will first discuss Chern classes of vector bundles on Artin stacks. Let
$\cX$ be an Artin stack, let $n\ge 2$ be an integer invertible on $\cX$, and
let $\cL$ be a line bundle on $\cX$. The isomorphism class of $\cL$ defines
an element in $H^1(\cX,\G_m)$. We denote by
\begin{equation}\label{e.c1}
c_1(\cL)\in H^2(\cX,\Z/n\Z (1))
\end{equation}
the image of this element by the homomorphism $H^1(\cX,\G_m)\to H^2(\cX,\Z/n\Z(1))$ induced by the short exact sequence
\[1\to \Z/n\Z(1)\to \G_m \xrightarrow{n} \G_m \to 1,\]
where the map marked by $n$ is raising to the $n$-th power. For any integer
$i$, we write $\Z/n\Z(i)=\Z/n\Z(1)^{\otimes i}$. We say a quasi-coherent
sheaf \cite[06WG]{stacks} $\cE$ on $\cX$ is a \emph{vector bundle} if there
exists a smooth atlas $p\colon X\to \cX$ such that $p^*\cE$ is a locally
free $\cO_X$-module of finite rank. The following theorem generalizes the
construction of Chern classes of vector bundles on schemes
(\cite[Th\'eor\`eme 1.3]{Riou} and \cite[VII 3.4, 3.5]{SGA5}). If $\cX$ is a
Deligne-Mumford stack, it yields the Chern classes over the \'etale topos of
$\cX$ locally ringed by $\cO_\cX$, defined by Grothendieck in
\cite[(1.4)]{GrothDix}. In particular, it also generalizes
\cite[(2.3)]{GrothDix}.

\begin{theorem}\label{t.Chern}
  There exists a unique way to define, for every Artin stack $\cX$ over $\Z[1/n]$ and every vector bundle $\cE$ on $\cX$, elements $c_i(\cE)\in H^{2i}(\cX,\Z/n\Z(i))$ for all $i\ge 0$ such that the formal power series $c_t(\cE)=\sum_{i\ge 0}c_i(\cE)t^i$ satisfies the following conditions:
\begin{enumerate}
  \item (Functoriality) If $f\colon \cY\to \cX$ is a morphism of stacks over $\Z[1/n]$, then $f^*(c_t(\cE))=c_t(f^*\cE)$;

  \item (Additivity) If $0\to \cE'\to \cE\to \cE''\to 0$ is an exact sequence of vector bundles, then $c_t(\cE)=c_t(\cE')c_t(\cE'')$;

  \item (Normalization) If $\cL$ is a line bundle on $\cX$, then
      $c_1(\cL)$ coincides with the class defined in \eqref{e.c1} and
      $c_t(\cL)=1_\cX+c_1(\cL)t$. Here $1_\cX$ denotes the image of $1$ by
      the adjunction homomorphism $\Z/n\Z\to H^0(\cX,\Z/n\Z)$.
\end{enumerate}
  Moreover, we have:
\begin{itemize}
\item[(d)]  $c_0(\cE)=1_\cX$ and $c_i(\cE)=0$ for $i>\rk(\cE)$.
\end{itemize}
\end{theorem}

The $c_i(\cE)$ are called the (\'etale) \emph{Chern classes} of $\cE$. It
follows from (b) and (d) that $c_t(\cE)$ only depends on the isomorphism
class of $\cE$.

To prove Theorem \ref{t.Chern}, we need the following result, which
generalizes \cite[VII Th\'eor\`eme 2.2.1]{SGA5} and \cite[Th\'eor\`eme
1.2]{Riou}.

\begin{prop}\label{p.proj}
Let $\cX$ be an Artin stack and let $\cE$ be a vector bundle of constant
rank $r$ on $\cX$. Let $n$ be an integer invertible on $\cX$ and let
$\Lambda$ be a commutative ring over $\Z/n\Z$.  We denote by $\pi\colon
\bP(\cE)\to \cX$ the projective bundle of $\cE$. Let
$\xi=c_1(\cO_{\bP(\cE)}(1))\in H^2(\bP(\cE),\Lambda(1))$ as in \eqref{e.c1}.
Then the powers $\xi^i\in H^{2i}(\bP(\cE),\Lambda(i))$ of $\xi$ define an
isomorphism in $D(\cX,\Lambda)$
\begin{equation}\label{e.proj0}
  (1,\xi,\dots,\xi^{r-1})\colon \bigoplus_{i=0}^{r-1} \Lambda(-i)[-2i]\simto R\pi_* \Lambda.
\end{equation}
\end{prop}

\begin{proof}
By base change \cite{LiuZheng}, we reduce to the case of schemes, which is
proven in \cite[VII Th\'eor\`eme 2.2.1]{SGA5}.
\end{proof}

The uniqueness of Chern classes is a consequence of the following lemma,
which generalizes \cite[Propositions 1.4, 1.5]{Riou}.

\begin{lemma}
Let $\cX$ be an Artin stack, let $n$ be an integer invertible on
  $\cX$, and let $\Lambda$ be a commutative ring over $\Z/n\Z$.
  \begin{enumerate}
    \item (Splitting principle) Let $\cE$ be a vector bundle on $\cX$ of
        rank $r$ and let $\pi\colon \Flag(\cE)\to \cX$ be the fibration of
        complete flags of $\cE$. Then $\pi^*\cE$ admits a canonical
        filtration by vector bundles such that the graded pieces are line
        bundles, and the morphism $\Lambda \to R\pi_* \Lambda$ is a split
        monomorphism.

    \item Let $E\colon 0\to \cE'\to \cE\xrightarrow{p} \cE''\to 0$ be a
        short exact sequence of vector bundles and let $\pi\colon
        \Sect(E)\to \cX$ be the fibration of sections of $p$. Then
        $\Sect(E)$ is a torsor under $\cHom(\cE'',\cE')$ and $\pi^*E$ is
        canonically split. Moreover, the morphism $\Lambda\to
        R\pi_*\Lambda$ is an isomorphism.
  \end{enumerate}
\end{lemma}

\begin{proof}
(a) follows from Proposition \ref{p.proj}, as $\pi$ is a composite of $r$
successive projective bundles. For (b), up to replacing $\cX$ by an atlas,
        we may assume that $\pi$ is the projection from an affine
        space. In this case the assertion follows from \cite[XV
       Corollaire 2.2]{SGA4}.
\end{proof}

To define $c_i(\cE)$, we may assume $\cE$ is of constant rank $r$. As usual,
we define
\[c_i(\cE)\in H^{2i}(\cX,\Z/n\Z(i)),\quad 1\le i\le r,\]
as the unique
elements satisfying
\[\xi^r+\sum_{1\le i\le r} (-1)^i c_i(\cE)\xi^{r-i}=0,\]
where $\xi=c_1(\cO_{\bP(\cE)}(1))\in H^2(\bP(\cE),\Z/n\Z(1))$. We put
$c_0(\cE)=1$ and $c_i(\cE)=0$ for $i>r$. The properties (a) to (d) follow
from the case of schemes. If $c_i(\cE)=0$ for all $i>0$, in particular if
$\cE$ is trivial, then \eqref{e.proj0} is an isomorphism of rings in
$D(\cX,\Lambda)$.

\begin{theorem}\label{l.finite}
Let $S$ be an algebraic space, let $n$ be an integer invertible on $S$, and
let $\Lambda$ be a commutative ring over $\Z/n\Z$. Let $N\ge 1$ be an
integer, let $G=GL_{N,S}$, and let $T=\prod_{i=1}^N T_i\subset G$ be the
subgroup of diagonal matrices, where $T_i=\G_{m,S}$. Let $\pi\colon BG\to
S$, $\tau\colon BT\to S$, $f'\colon G/T\to S$ be the projections, let
$k\colon S\to BG$ be the canonical section, let $f\colon BT\to BG$ be the
morphism induced by the inclusion $T\to G$ and let $h\colon G/T\to BT$ be
the morphism induced by the projection $G\to S$, as shown in the following
2-commutative diagram
\begin{equation}\label{e.dfinite}
\xymatrix{G/T\ar[d]_{f'}\ar[r]^h & BT\ar[d]_f\ar[rd]^\tau\\
S\ar[r]^k  & BG\ar[r]^\pi &S.}
\end{equation}
Let $\cE$ be the standard vector bundle of rank $N$ on $BG$, corresponding
to the natural representation of $G$ in $\cO_S^N$. The $i$-th Chern class
$c_i(\cE)$ of $\cE$ induces a morphism
    \[\alpha_i\colon K_i=\Lambda_S(-i)[-2i]\to R\pi_*\Lambda.\]
Let ${\cL}_i$ be the inverse image on $BT$ of the standard line bundle on
$BT_i$. Its first Chern class $c_1(\cL_i)$ induces a morphism
\[\beta_i\colon L_i=\Lambda_S(-1)[-2]\to R\tau_*\Lambda.\]
For a graded sheaf of $\Lambda$-modules $M=\bigoplus_{i\in \Z} M_i$ on $S$,
we let $M^\Delta= \bigoplus_i M_i(-i)[-2i]\in D(S,\Lambda)$. Let
$\Lambda_S[x_1,\dots, x_N]$ (resp.\ $\Lambda_S[t_1,\dots,t_N]$) be a
polynomial algebra on generators $x_i$ of degree $i$ (resp.\ $t_i$ of degree
$1$). The corresponding object $\Lambda_S[x_1,\dots,x_N]^{\Delta}$ (resp.\
$\Lambda_S[t_1,\dots,t_N]^{\Delta}$) is naturally identified with
$\Sym_\Lambda(\bigoplus_{1\le i \le N} K_i)$ (resp.\
$\Sym_\Lambda(\bigoplus_{1\le i \le N} L_i)$). Then the ring homomorphisms
\begin{gather}
\label{e.dfinitea}\alpha\colon \Lambda_S[x_1,\dots,x_N]^{\Delta} \to R\pi_* \Lambda,\\
\label{e.dfiniteb}\beta\colon \Lambda_S[t_1,\dots,t_N]^{\Delta}\to R\tau_* \Lambda,
\end{gather}
defined respectively by $\alpha_i$ and $\beta_i$, are isomorphisms of rings
in $D(S,\Lambda)$, and fit into a commutative diagram of rings in
$D(S,\Lambda)$
\begin{equation}\label{e.dfinite2}
\xymatrix{\Lambda_S[x_1,\dots, x_N]^\Delta\ar[r]^\sigma\ar[d]_\alpha^\simeq & \Lambda_S[t_1,\dots, t_N]^\Delta
\ar[d]_\beta^\simeq\ar[r]^-\rho & (\Lambda_S[t_1,\dots, t_N]/(\sigma_1,\dots, \sigma_N))^\Delta\ar[d]_{\gamma}^\simeq\\
 R\pi_*\Lambda\ar[r]^{a_f}
   & R\tau_*\Lambda\ar[r]^{a_h} & Rf'_*\Lambda,}
\end{equation}
which commutes with arbitrary base change of algebraic spaces $S'\to S$.
Here $\sigma$ sends $x_i$ to the $i$-th elementary symmetric polynomial
$\sigma_i$ in $t_1,\dots, t_N$, $\rho$ is the projection, $a_f$ is induced
by adjunction by $f$ and $a_h$ is induced by adjunction by $h$.  Moreover,
as graded module over $R^{2*}\pi_*\Lambda(*)$, $R^{2*}\tau_*\Lambda(*)$ is
free of rank $N!$.
\end{theorem}

In particular, we have canonical decompositions
\[R\pi_*\Lambda\simeq \bigoplus_q R^{2q}\pi_*
\Lambda[-2q],\quad R\tau_*\Lambda\simeq \bigoplus_q R^{2q}\tau_*\Lambda[-2q],\quad
Rf'_*\Lambda \simeq \bigoplus_q R^{2q}f'_*\Lambda[-2q],
\]
$a_h$ induces an epimorphism $R^*\tau_*\Lambda\to R^*f'_*\Lambda$ and $a_f$
induces an isomorphism $R^*\pi_*\Lambda\simto (R^*\tau_*\Lambda)^{\SG_N}$,
where $\SG_N$ is the symmetric group on $N$ letters. Moreover,
\eqref{e.dfinite2} induces a commutative diagram of sheaves of
$\Lambda$-algebras on $S$
\begin{equation}
\xymatrix{\Lambda_S[x_1,\dots, x_N]\ar[r]^\sigma\ar[d]_\alpha^\simeq & \Lambda_S[t_1,\dots, t_N]
\ar[d]_\beta^\simeq\ar[r]^-\rho & \Lambda_S[t_1,\dots, t_N]/(\sigma_1,\dots, \sigma_N)\ar[d]_{\gamma}^\simeq\\
 R^{2*}\pi_*\Lambda(*)\ar[r]^{a_f}
   & R^{2*}\tau_*\Lambda(*)\ar[r]^{a_h} & R^{2*}f'_*\Lambda(*),}
\end{equation}
where $\alpha$ carries $x_i$ to the image of $c_i(\cE)$ under the edge
homomorphism
\[H^{2i}(BG,\Lambda(i))\to H^0(S,R^{2i} \pi_* \Lambda(i)), \]
and $\beta$ carries $t_i$ to the image of $c_1(\cL_i)$ under the edge
homomorphism
\[H^{2i}(BT,\Lambda(i))\to H^0(S,R^{2i} \tau_* \Lambda(i)). \]

We will derive from Theorem \ref{l.finite} a formula for $Rf_*$ (see
Corollary \ref{c.Rf}).

\begin{proof}
As in \cite[Lemma 2.3.1]{BehrendInv}, we approximate $BG$ by a finite
Grassmannian $G(N,N')=M^*/G$,\footnote{This approximation argument was
explained by Deligne to the first author in the context of de Rham
cohomology in 1967.} where $N'\ge N$, $M$ is the algebraic $S$-space of
$N'\times N$ matrices $(a_{ij})_{\substack{1\le i\le N'\\1\le j\le N}}$,
$M^*$ is the open subspace of~$M$ consisting of matrices of rank $N$. Let
$B\subset G$ be the subgroup of upper triangular matrices. The square on the
right of the diagram with 2-cartesian squares
\[\xymatrix{M^*\ar[d]_y\ar[r] & M^*/T\ar[r]^p\ar[d]_{\psi}\ar[rrdd]^(.3)w
& M^*/B\ar[r]\ar[rdd]^(.3)v &M^*/G\ar[d]^{\phi}\ar@/^2pc/[dd]^u \\
  S\ar[r]^g & BT\ar@{-->}[rr]^(.2)f\ar[rrd]_{\tau} && BG\ar[d]^\pi\\
  &&&S}
\]
induces a commutative square
  \begin{equation}\label{e.BUsq}
  \xymatrix{R\pi_* \Lambda \ar[rr]\ar[d] && R\tau_* \Lambda \ar[d]\\
  Ru_* \Lambda\ar[r] & Rv_* \Lambda\ar[r]^{a_p}& Rw_* \Lambda.}
  \end{equation}
Here $a_p$ is induced by the adjunction $\Lambda\to Rp_*\Lambda$. The latter
is an isomorphism by \cite[XV Corollaire 2.2]{SGA4}, because $p$ is a
$(B/T)$-torsor and $B/T$ is isomorphic to the unipotent radical of $B$,
which is an affine space over $S$. The diagram
\[\xymatrix{M-M^*\ar[r]^-i\ar[rd]_z & M\ar[d]^x & M^*\ar[dl]^y\ar[l]\\
&S}
\]
induces an
exact triangle
  \[Rz_* Ri^! \Lambda \to Rx_* \Lambda \to Ry_* \Lambda \to.\]
Since $M$ is an affine space over $S$, the adjunction $\Lambda\to Rx_*
\Lambda$ is an isomorphism \cite[XV Corollaire 2.2]{SGA4}. Since $x$ is
smooth and the fibers of $z$ are of codimension $N'-N+1$, we have
$Ri^!\Lambda \in D^{\ge 2(N'-N+1)}$ by semi-purity \cite[Cycle
2.2.8]{SGA4d}. It follows that the adjunction $\Lambda\to \tau^{\le
2(N'-N)}Ry_* \Lambda$ is an isomorphism. By smooth base change by $g$
(resp.\ $fg$) \cite{LiuZheng}, this implies that the adjunction $\Lambda\to
\tau^{\le 2(N'-N)}R\psi_* \Lambda$ (resp.\ $\Lambda\to \tau^{\le
2(N'-N)}R\phi_* \Lambda$) is an isomorphism, so that the right (resp.\ left)
vertical arrow of $\tau^{\le 2(N'-N)}\eqref{e.BUsq}$ is an isomorphism.

The assertions then follow from an explicit computation of $Ru_* \Lambda$
and $Rv_*\Lambda$. Note that $M^*/B$ is a partial flag variety of the free
$\cO_S$-module $\cO_S^{N'}$ of type $(1,\dots, 1, N'-N)$. By \cite[VII
Propositions 5.2, 5.6 (a)]{SGA5} applied to $u$ and $v$, we have a
commutative square
\[\xymatrix{A^\Delta\ar[r]^-\sim\ar[d]_\sigma & Ru_*\Lambda\ar[d]\\
  C^\Delta\ar[r]^-\sim & Rv_*\Lambda.}
\]
Here
\begin{gather*}
A=\Lambda_S[x_1,\dots,x_N,y_1,\dots, y_{N'-N}]/(\sum_{i+j=m} x_i y_j)_{m\ge 1},\\
C=\Lambda_S[t_1,\dots,t_N,y_1,\dots,y_{N'-N}]/(\sum_{i+j=m} \sigma_i
y_j)_{m\ge 1},
\end{gather*}
the upper horizontal arrow sends $x_i$ to the $i$-th Chern class
$c_i(\cE_{N'})$ of the canonical bundle $\cE_{N'}$ of rank $N$ on the
Grassmannian $M^*/G$, the lower horizontal arrow sends $t_i$ to the first
Chern class $c_1(\cL_{i,N'})$ of the $i$-th standard line bundle
$\cL_{i,{N'}}$ of the partial flag variety $M^*/B$, and the upper (resp.\
lower) horizontal arrow sends $y_i$ to the $i$-th Chern class
$c_i(\cE'_{N'})$ of the canonical bundle $\cE'_{N'}$ of rank $N'-N$ on
$M^*/G$ (resp.\ on $M^*/B$). In the definition of the ideals, we put
$x_0=y_0=1$, $x_i=0$ for $i>N$ and $y_i=0$ for $i>N'-N$, and we used the
fact that $c_m$ of the trivial bundle of rank $N'$ is zero for $m\ge 1$. As
$\cE_{N'}$ (resp.\ $\cL_{i,{N'}}$) is induced from $\cE$ (resp.\ $\cL_i$),
by the functoriality of Chern classes (Theorem \ref{t.Chern}), these
isomorphisms are compatible with the morphisms $\alpha$ \eqref{e.dfinitea}
and $\beta$ \eqref{e.dfiniteb}. We can rewrite $A$ as
$\Lambda[x_1,\dots,x_N]/(P_m(x_1,\dots,x_m))_{m>N'-N}$ and rewrite $C$ as
$\Lambda[t_1,\dots,t_N]/(Q_m(t_1,\dots,t_m))_{m>N'-N}$, where $P_m$ is an
isobaric polynomial of weight $m$ in $x_1,\dots, x_m$, $x_i$ being of weight
$i$, and $Q_m$ is a homogeneous polynomial of degree $m$ in $t_1,\dots,
t_m$. As the vertical arrows of \eqref{e.BUsq} induce isomorphisms after
application of the truncation functor $\tau^{\le 2(N'-N)}$, it follows that
$\tau^{\le 2(N'-N)}$ of the square on the left of \eqref{e.dfinite2} is
commutative and the vertical arrows induce isomorphisms after application of
$\tau^{\le 2(N'-N)}$. To get the square on the right of \eqref{e.dfinite2},
it suffices to apply the preceding computation of $Rw_*\Lambda$ (via
$Rv_*\Lambda$) to the case $N'=N$, because, in this case, $f'=w$. The fact
that \eqref{e.dfinite2} commutes with base change follows from the
functoriality of Chern classes. The last assertion of the theorem then
follows from \cite[VII Lemme 5.4.1]{SGA5}.
\end{proof}

\begin{cor}\label{c.Rf}
With assumptions and notation as in Theorem \ref{l.finite}:
\begin{enumerate}
\item For every locally constant $\Lambda$-module $\cF$ on $S$, the
    projection formula maps
\[\cF\otimes^L_\Lambda R\pi_*\Lambda \to R\pi_*\pi^* \cF, \quad \cF\otimes^L_\Lambda R\tau_* \Lambda \to R\tau_*\tau^* \cF\]
are isomorphisms.

\item The classes $c_1(\cL_i)$ induce an isomorphism of rings
    $\Lambda_{BG}[t_1,\dots,t_N]^\Delta/J\to Rf_*\Lambda$ in
    $D(BG,\Lambda)$, where $J$ is the ideal generated by
    $\sigma_i-c_i(\cE)$. Moreover, the left square of \eqref{e.dfinite}
    induces an isomorphism of $R\pi_*\Lambda$-modules
    $R\tau_*\Lambda\simeq R\pi_*\Lambda\otimes^L_\Lambda Rf'_*\Lambda$.
\end{enumerate}
\end{cor}

\begin{proof}
(a) We may assume that $\cF$ is a constant $\Lambda$-module of value $F$.
Then the assertion follows from Theorem \ref{l.finite} applied to $\Lambda$
and to the ring of dual numbers $\Lambda\oplus F$ (with $m_1m_2=0$ for
$m_1,m_2\in F$).

(b) Since $f^*\cE\simeq \bigoplus_{i=1}^{N} \cL_i$,
$f^*c_i(\cE)=c_i(f^*\cE)$ is the $i$-th elementary symmetric polynomial in
$c_1(\cL_1),\dots,c_1(\cL_N)$. Thus the ring homomorphism
$\Lambda_{BG}[t_1,\dots,t_N]^\Delta\to Rf_*\Lambda$ induced by $c_1(\cL_i)$
factorizes through a ring homomorphism
$\Lambda_{BG}[t_1,\dots,t_N]^\Delta/J\to Rf_*\Lambda$. By Proposition
\ref{p.cart} applied to the square
  \[\xymatrix{(G,T)\ar[r]\ar[d]& (S,T)\ar[d]\\
  (G,G)\ar[r] & (S,G),}\]
the left square of \eqref{e.dfinite} is 2-cartesian. By smooth base change
by $k$, we have $k^* Rf_*\Lambda \simeq Rf'_*\Lambda$. The first assertion
then follows from Theorem \ref{l.finite}. By Remark \ref{r.BG}, it follows
that $Rf_*\Lambda\simeq \pi^*k^*Rf_*\Lambda\simeq \pi^* Rf'_*\Lambda$. Thus
$R\tau_*\Lambda\simeq R\pi_* Rf_*\Lambda \simeq R\pi_*\pi^*Rf'_*\Lambda$ and
the second assertion follows from (a).
\end{proof}

Let $k$ be a separably closed field, let $n$ be an integer invertible in
$k$, and let $\Lambda$ be a noetherian commutative ring over $\Z/n\Z$. The
next sequence of results are analogues of Quillen's finiteness theorem
\cite[Theorem 2.1, Corollaries 2.2, 2.3]{Quillen1}. Recall that an algebraic
space over $\Spec k$ is of finite presentation if and only if it is
quasi-separated and of finite type.

\begin{theorem}\label{t.finite}
Let $G$ be an algebraic group over $k$, let $X$ be an algebraic space of
finite presentation over $\Spec k$ equipped with an action of $G$, and let
$K$ be an
    object of $D^b_c([X/G],\Lambda)$ (see Notation \ref{s.cart}). Then $H^*(BG,\Lambda)$ is a finitely generated
    $\Lambda$-algebra and $H^*([X/G],K)$ is a finite
    $H^*(BG,\Lambda)$-module. In particular, if $K$ is a ring in the sense
    of Definition \ref{s.ring}, then the graded center $ZH^*([X/G],K)$ of $H^*([X/G],K)$ is a finitely generated
    $\Lambda$-algebra.
\end{theorem}

Initially the authors established Theorem \ref{t.finite} for $G$ either a
linear algebraic group or a semi-abelian variety. The finiteness of
$H^*(BG,\Lambda)$ in the general case was proved by Deligne in
\cite{Deligneletter}.

\begin{cor}\label{t.finite2}
Let $G$ be an algebraic group over $k$ and let $f\colon\cX\to BG$ be a
representable morphism of Artin stacks of finite presentation over $\Spec
k$, and let $K\in D^b_c(\cX,\Lambda)$. Consider $H^*(\cX,K)$ as an
$H^*(BG,\Lambda)$-module by restriction of scalars via the map $f^*\colon
H^*(BG,\Lambda) \rightarrow H^*(\cX,\Lambda)$. Then $H^*(\cX,K)$ is a finite
$H^*(BG,\Lambda)$-module.
\end{cor}

\begin{proof}
It suffices to apply Theorem \ref{t.finite} to $Rf_*K\in D^b_c(BG,\Lambda)$.
\end{proof}

\begin{cor}\label{c.finite}
Let $X$ (resp.\ $Y$) be an algebraic space of finite presentation over
$\Spec k$, equipped with an action of an algebraic group $G$ (resp.\ $H$)
over $k$. Let $(f,u) \colon (X,G) \rightarrow (Y,H)$ be an equivariant
morphism. Assume that $u$ is a monomorphism. Then the map $[f/u]^*$  makes
$H^*([X/G],\Lambda)$ a finite $H^*([Y/H],\Lambda)$-module.
\end{cor}

Indeed, since the map $[X/G]\to BH$ induced by $u$ is representable,
$H^*([X/G],\Lambda)$ is a finite $H^*(BH,\Lambda)$-module by Corollary
\ref{t.finite2}, hence a finite $H^*([Y/H],\Lambda)$-module.

\begin{proof}[Proof of Theorem \ref{t.finite}]
By the invariance of \'etale cohomology under schematic universal
homeomorphisms, we may assume $k$ algebraically closed and $G$ reduced
(hence smooth). Then $G$ is an extension $1\to G^0 \to G \to F\to 1$, where
$F$ is the finite group $\pi_0(G)$ and $G^0$ is the identity component of
$G$. By Chevalley's theorem (cf.\ \cite[Theorem 1.1.1]{Brion} or
\cite[Theorem 1.1]{Conrad}), $G^0$ is an extension $1\to L\to G^0 \to A \to
1$, where $A$ is an abelian variety and $L=G_\aff$ is the largest connected
affine normal subgroup of $G^0$. Then $L$ is also normal in $G$, and if
$E=G/L$, then $E$ is an extension $1\to A\to E\to F\to 1$. We will sum up
this d\'evissage by saying that $G$ is an iterated extension $G=L\cdot
A\cdot F$.

By \cite[VIII 7.1.5, 7.3.7]{Giraud}, for every algebraic group $H$ over $k$,
the extensions of $F$ by $H$ with given action of $F$ on $H$ by conjugation
are classified by $H^2(BF,H)$. In particular, the extension $E$ of $F$ by
$A$ defines an action of $F$ on $A$ and a class in $H^2(BF,A)$, which comes
from a class $\alpha$ in $H^2(BF,A[m])$, where $m$ is the order of $F$ and
$A[m]$ denotes the kernel of $m\colon A\to A$. Indeed the second arrow in
the exact sequence
\[H^2(BF,A[m])\to H^2(BF,A) \xrightarrow{\times m} H^2(BF,A)\]
is equal to zero. This allows us to define an inductive system of subgroups
$E_i=A[mn^i]\cdot F$ of~$E$, given by the image of $\alpha$ in
$H^2(F,A[mn^i])$. This induces an inductive system of subgroups $G_i=L\cdot
A[mn^i]\cdot F$ of $G$, fitting into short exact sequences
\[\xymatrix{1\ar[r]& L \ar[r]\ar@{=}[d] & G_i\ar[r]\ar[d]\ar@{}[rd]|\square & E_i\ar[r]\ar[d] &
1\\
1\ar[r] & L\ar[r] & G\ar[r] &E\ar[r] & 1.
}
\]
Form the diagram with cartesian squares
\[\xymatrix{[X/G_i]\ar[r]\ar[d]_{f_i}& BG_i \ar[d]& G/G_i\ar[l]\ar[d]\\
[X/G] \ar[r]&  BG & \Spec k. \ar[l]}
\]
Note that $G/G_i=A/A[mn^i]$ and the vertical arrows in the above diagram are
proper representable. By the classical projection formula \cite[XVII
(5.2.2.1)]{SGA4}, $R f_{i*}f_i^*K\simeq K\otimes^L_\Lambda Rf_{i*}\Lambda$.
Moreover, $f_{i*}\Lambda \simeq \Lambda$. Thus we have a distinguished
triangle
\begin{equation}\label{e.finitetri}
K\to Rf_{i*}f_i^*K \to K \otimes^L_\Lambda \tau^{\ge 1}Rf_{i*}\Lambda\to .
\end{equation}
The first term forms a constant system and the third term $N_i=K
\otimes^L_\Lambda\tau^{\ge 1}Rf_{i*}\Lambda$ forms an AR-null system of
level $2d$ in the sense that $N_{i+2d}\to N_i$ is zero for all $i$, where
$d=\dim A$. Indeed the stalks of $R^q f_{i*}\Lambda$ are
$H^q(A/A[mn^i],\Lambda)$, which is zero for $q>2d$. For $q=0$, the
transition maps of $(H^0(A/A[mn^i],\Lambda))$ are $\id_\Lambda$ and for
$q>0$, the transition maps of $(H^q(A/A[mn^i],\Lambda))$ are zero. Thus, in
the induced long exact sequence of \eqref{e.finitetri}
\[H^{*-1}([X/G],N_i) \to H^*([X/G],K) \xrightarrow{\alpha_i} H^*([X/G_i], f_i^* K)\to H^*([X/G],N_i),\]
the system $(H^*([X/G],N_i))$ is AR-null of level $2d$. Therefore,
$\alpha_i$ is injective for $i\ge 2d$ and $\Img
\alpha_i=\Img(H^*([X/G_{i+2d}],f_{i+2d}^* K)\to H^*([X/G_{i}],f_{i}^* K))$
for all $i$. Taking $i=2d$, we get $H^*([X/G],K)=\Img(
H^*([X/G_{4d}],f_{4d}^* K)\to H^*([X/G_{2d}],f_{2d}^* K))$. In particular,
$H^*(BG,\Lambda)$ is a quotient $\Lambda$-algebra of $H^*(BG_{4d},\Lambda)$,
and $H^*([X/G],K)$ is a quotient $H^*(BG,\Lambda)$-module of
$H^*([X/G_{4d}],f_{4d}^*K)$. Therefore, it suffices to show the theorem with
$G$ replaced by $G_{4d}$. In particular, we may assume that $G$ is a linear
algebraic group.

Let $G\to \GL_r$ be an embedding into a general linear group. By Corollary
\ref{p.wedge}, the morphism of Artin stacks over $B\GL_r$,
\[[X/G]\to [(X\wedge^G \GL_r)/\GL_r],\]
is an equivalence. Replacing $G$ by $\GL_r$ and $X$ by $X\wedge^G \GL_r$, we
may assume that $G=\GL_r$.  Let $f\colon [X/G]\to BG$. Then $Rf_*K\in
D^b_c(BG,\Lambda)$. Thus we may assume $X=\Spec k$. The full subcategory of
objects $K$ satisfying the theorem is a triangulated category. Thus we may
further assume $K\in \Mod_c(BG,\Lambda)$. In this case, since $G$ is
connected, $K$ is necessarily constant (Corollary \ref{l.BG}) so that
$K\simeq \pi^* M$ for some finite $\Lambda$-module $M$, where $\pi\colon
BG\to \Spec k$. In this case, by Theorem \ref{l.finite},
$H^*(BG,\Lambda)\simeq\Lambda[c_1,\dots, c_r]$ is a noetherian ring and
$H^*(BG,K)\simeq M\otimes_\Lambda \Lambda[c_1,\dots,c_r]$ is a finite
$H^*(BG,\Lambda)$-module.
\end{proof}

\begin{remark}\label{r.finite}
We have shown in the proof of Theorem \ref{t.finite} that
$H^*([X/G],K)\simeq \Img(H^*([X/G_{4d}],f_{4d}^*K)\to
H^*([X/G_{2d}],f_{2d}^*K))$. In particular, $H^*([X/G],K)$ is a quotient
$H^*(BG,\Lambda)$-module of $H^*([X/G_{4d}],f_{4d}^*K)$. Here
$G_{2d}<G_{4d}$ are affine subgroups of $G$, independent of $X$ and $K$, and
$f_{2d}\colon [X/G_{2d}]\to [X/G]$, $f_{4d}\colon [X/G_{4d}]\to [X/G]$.
\end{remark}

In the following examples, we write $H^*(-)$ for $H^*(-,\Lambda)$, with
$\Lambda$ as in Theorem \ref{t.finite}.

\begin{example}
Let $G/k$ be an extension of an abelian variety $A$ of dimension $g$ by a
torus $T$ of dimension $r$. Then
\begin{enumerate}
\item $H^1(A)$, $H^1(T)$, $H^1(G)$ are free over $\Lambda$ of ranks $2g$,
    $r$ and $2g +r$ respectively, and the sequence $0 \to H^1(A) \to
    H^1(G) \to H^1(T) \to 0$ is exact. The inclusion
    $H^1(G)\hookrightarrow H^*(G)$ induces an isomorphism of
    $\Lambda$-modules
\begin{equation}\label{e.4.8.1}
\wedge H^1(G) \simto H^*(G).
\end{equation}

\item The homomorphism
\[
d_2^{01} \colon H^1(G) \to H^2(BG)
\]
in the spectral sequence
\[
E^{pq}_2 = H^p(BG) \otimes H^q(G) \Rightarrow H^{p+q}(\Spec k)
\]
of the fibration $\Spec k \to BG$ is an isomorphism.

\item We have $H^{2i+1}(BG) = 0$ for all $i$, and the inclusion $H^2(BG)
    \hookrightarrow H^*(BG)$ extends to an isomorphism of
    $\Lambda$-algebras
\[
\Sym(H^2(BG)) \simto H^*(BG).
\]
\end{enumerate}

Let us briefly sketch a proof.

Assertion (a) is standard. By projection formula, we may assume $\Lambda =
\Z/n\Z$. As the multiplication by $n$ on $T$ is surjective, the sequence $0
\to T[n] \to G[n] \to A[n] \to 0$ is exact. The surjection $\pi_1(G)\to
G[n]$ induces an injection $\Hom(G[n],\Z/n\Z) \to H^1(G)$. The fact that
this injection and \eqref{e.4.8.1} are isomorphisms follows (after reducing
to $n = \ell$ prime) from the structure of Hopf algebra of $H^*(G)$, as
$H^{2g+r}(G) \simto H^{r}(T) \otimes H^{2g}(A)$ is of rank 1 (cf.\
\cite[Chapter VII, Proposition 16]{SerreGA}).

Assertion (b) follows immediately from (a).

To prove (c) we calculate $H^*(BG)$ using the nerve $B_\bullet G$ of $G$
(cf.\ \cite[6.1.5]{HodgeIII}):
\[
H^*(BG) = H^*(B_\bullet G),
\]	
which gives the Eilenberg-Moore spectral sequence:
\begin{equation}\label{e.4.8.star}
E^{ij}_1 = H^j(B_iG) \Rightarrow H^{i+j}(BG).
\end{equation}
One finds that
\[
E^{\bullet,j}_1 \simeq L\wedge^j(H^1(G)[-1]).
\]
By \cite[I 4.3.2.1 (i)]{Illusie239} we get
\[
E^{\bullet,j}_1 \simeq L\Sym^j(H^1(G))[-j].
\]
Thus
\begin{equation*}
E^{ij}_2 \simeq \begin{cases}\Sym^j(H^1(G)) & \text{if $i = j$,} \\
0  & \text{if $i \ne j$.}\end{cases}
\end{equation*}
The $E_2$ term is concentrated on the diagonal, hence \eqref{e.4.8.star}
degenerates at $E_2$, and we get an isomorphism
\[
H^*(BG) = \Sym(H^1(G)[-2]),
\]
from which (c) follows.
\end{example}

\begin{example}
Let $G$ be a connected algebraic group over $k$. Assume that for every prime
number $\ell$ dividing $n$, $H^i(G,\Z_{\ell})$ is torsion-free for all $i$.
Classical results due to Borel \cite{Borel3} can be adapted as follows.
\begin{enumerate}
\item $H^*(G)$ is the exterior algebra over a free $\Lambda$-module having
    a basis of elements of odd degree \cite[Propositions 7.2,
    7.3]{Borel3}.

\item In the spectral sequence of the fibration $\Spec k \to BG$,
\[
E^{ij}_2 = H^i(BG) \otimes H^j(G) \Rightarrow H^{i+j}(\Spec k),
\]
primitive and transgressive elements coincide \cite[Proposition
20.2]{Borel3}, and the transgression gives an isomorphism $d_{q+1}\colon
P^q\simto Q^{q+1}$ from the transgressive part $P^q= E_{q+1}^{0q}$ of
$H^q(G)\simeq E_2^{0q}$ to the quotient $Q^{q+1}= E_{q+1}^{q+1,0}$ of
$H^{q+1}(BG)\simeq  E_2^{q+1,0}$. Moreover, $Q^*$ is a free
$\Lambda$-module having a basis of elements of even degrees, and every
section of $H^*(BG)\to Q^*$ provides an isomorphism between $H^*(BG)$ and
the polynomial algebra $\Sym_\Lambda(Q^*)$ \cite[Th\'eor\`emes 13.1,
19.1]{Borel3}.
\end{enumerate}

Now assume that $G$ is a connected reductive group over $k$. Let $T$ be
\textit{the} maximal torus in $G$, and $W = \Norm_G(T)/T$ \textit{the} Weyl
group. Recall that $G$ is $\ell$-torsion-free if $\ell$ does not divide the
order of $W$, cf.\ \cite{Borel4}, \cite[Section 1.3]{SerreBour}. As in
\cite[Sommes trig., 8.2]{SGA4d}, the following results can be deduced from
the classical results on compact Lie groups by lifting $G$ to characteristic
zero.
\begin{itemize}
\item[(c)] The spectral sequence
\begin{equation}\label{e.4.9.star}
E^{ij}_2 = H^i(BG) \otimes
    H^j(G/T) \Rightarrow H^{i+j}(BT)
\end{equation}
degenerates at $E_2$, $E^{ij}_2$ being zero if $i$ or $j$ is
odd\footnote{The vanishing of $H^j(G/T)$ for $j$ odd follows for example
from the Bruhat decomposition of $G/B$ for a Borel $B$ containing $T$.}.
In particular, the homomorphism
\[ H^*(BT) \to H^*(G/T)
\]
induced by the projection $G/T \to BT$ is surjective. In other words, in
view of Theorem \ref{l.finite}, $H^*(G/T)$ is generated by the Chern
classes of the invertible sheaves $L_{\chi}$ obtained by pushing out the
$T$-torsor $G$ over $G/T$ by the characters $\chi \colon T \to \Gm$.

\item[(d)] The Weyl group $W$ acts on \eqref{e.4.9.star}, trivially on
    $H^*(BG)$, and $H^*(G/T)$ is the regular representation of $W$
    \cite[Lemme 27.1]{Borel3}. In particular, the homomorphism $H^*(BG)
    \to H^*(BT)$ induced by the projection $BT \to BG$ induces an
    isomorphism
\begin{equation}\label{e.GTW}
 H^*(BG) \simto H^*(BT)^W.
\end{equation}
\end{itemize}
\end{example}

\section{Finiteness of orbit types}\label{s.3}

Let $k$ be a field of characteristic $p\ge 0$, let $G$ be an algebraic group
over $k$, and let $A$ be a finite group. The presheaf of sets
$\cHom_\group(A,G)$ on $\AlgSp_{/k}$ is represented by a closed subscheme
$X$ of the product $\prod_{a\in A}G$ of copies of $G$ indexed by~$A$. In the
case where $A\simeq (\Z/\ell\Z)^r$ is an elementary abelian $\ell$-group of
rank $r$, $\cHom_\group(A,G)(T)$ can be identified with the set of commuting
$r$-tuples of $\ell$-torsion elements of $G(T)$. The group $G$ acts on $X$
by conjugation. Let $x\in X(k)$ be a rational point of $X$ and let $c\colon
G\to X$ be the $G$-equivariant morphism sending $g$ to $xg$, where $xg\colon
a\mapsto g^{-1}x(a)g$. Let $H=c^{-1}(x)\subset G$ be the \emph{inertia}
subgroup at $x$. The morphism $c$ decomposes into
\[G\to H\backslash G \xrightarrow{f} X,\]
where $f$ is an immersion \cite[III, \S~3, Proposition 5.2]{DG}. The
\emph{orbit} of $x$ under $G$ is the (scheme-theoretic) image of $f$, which
is a subscheme of $X$. The orbits of $X$ are disjoint with each other.

The following result is probably well known. It was communicated to us by
Serre.

\begin{theorem}[Serre]\label{t.Serre}
Assume that the order of $A$ is not divisible by $p$. Then the orbits of $X$
under the action of $G$ are open subschemes. Moreover, if $G$ is smooth,
then $X$ is smooth.
\end{theorem}

The condition on the order of $A$ is essential. For example, if $p>0$,
$A=\Z/p\Z$ and $G=\G_a$ is the additive group, then $G$ acts trivially on
$X\simeq G$.

Note that for any field extension $k'$ of $k$, if $Y$ is an orbit of $X$
under $G$, then $Y_{k'}$ is an orbit of $X_{k'}$ under $G_{k'}$.

\begin{cor}\label{c.Serre0}
The orbits are closed and the number of orbits is finite. Moreover, if
    $k$ is algebraically closed, then the orbits form a disjoint open
    covering of $X$.
\end{cor}

\begin{proof}
It suffices to consider the case when $k$ is algebraically closed. In this
case, rational points of $X$ form a dense subset \cite[Corollaire
10.4.8]{EGAIV}. Thus, by Theorem \ref{t.Serre}, the orbits form a disjoint
open covering of the quasi-compact topological space $X$. Therefore, the
orbits are also closed and the number of orbits is finite.
\end{proof}

\begin{cor}\label{c.Serre}
Let $G$ be an algebraic group over $k$ and let $\ell$ be a prime number
distinct from~$p$. There are finitely many conjugacy classes of elementary
abelian $\ell$-subgroups of~$G$. Moreover, if $k$ is algebraically closed
and $k'$ is an algebraically closed extension of $k$, then the natural map
$S_k\to S_{k'}$ from the set $S_k$ of conjugacy classes of elementary
abelian $\ell$-subgroups of~$G$ to the set $S_{k'}$ of conjugacy classes of
elementary abelian $\ell$-subgroups of~$G_{k'}$ is a bijection.
\end{cor}

\begin{proof}
By Corollary \ref{c.Serre0}, it suffices to show that the ranks of the
elementary abelian $\ell$-subgroups of~$G$ are bounded. For this, we may
assume $k$ algebraically closed, and $G$ smooth. As in the proof of Theorem
\ref{t.finite}, let $L$ be the maximal connected affine normal subgroup of
the identity component $G^0$ of~$G$. Let $d$ be the dimension of the abelian
variety $G^0/L$, and let $m$ be the maximal integer such that $\ell^m |
[G:G^0]$. Choose an embedding of $L$ into some $\GL_n$. Then every
elementary abelian subgroup of $G$ has rank $\le n+2d+m$.
\end{proof}

To prove the theorem, we need a lemma on tangent spaces. Let $S$ be an
algebraic space, and let $X$ be an $S$-functor, that is, a presheaf of sets
on $\AlgSp_{/S}$. Recall \cite[II 3.1]{SGA3} that the \emph{tangent bundle}
to $X$ is defined to be the $S$-functor
\[T_{X/S}=\cHom_S(\Spec(\cO_S[\epsilon]/(\epsilon^2)),X),\]
which is endowed with a projection to $X$. For every point $u\in X(S)$, the
\emph{tangent space} to $X$ at $u$ is the $S$-functor \cite[II 3.2]{SGA3}
\[T^u_{X/S}=T_{X/S}\times_{X,u}S.\]
Recall \cite[II 3.11]{SGA3} that, for $S$-functors $Y$ and $Z$, we have an
isomorphism
\[T_{\cHom_S(Y,Z)/S}\simeq \cHom_S(Y,T_{Z/S}).\]
For a morphism $f \colon Y\to Z$  of $S$-functors, this induces an
isomorphism
\begin{equation}\label{e.Lu}
T^f_{\cHom_S(Y,Z)/S} \simeq \cHom_{Z/S}((Y,f),T_{Z/S}).
\end{equation}
Assume that $Z$ is an $S$-group, that is, a presheaf of groups on
$\AlgSp_{/S}$. Then we have an isomorphism of schemes $T_{Z/S}\simeq
Z\times_S\Lie(Z/S)$, where $\Lie(Z/S)=T^1_{Z/S}$. Thus \eqref{e.Lu} induces
an isomorphism
\begin{equation}\label{e.Lu2}
T^f_{\cHom_S(Y,Z)/S} \simto \cHom_S(Y,\Lie(Z/S)).
\end{equation}
Furthermore, if $Y$ is an $S$-group and $f$ is a homomorphism of $S$-groups,
then the image of $T_{\cHom_{S\text{-}\group}(Y,Z)/S}^f$ by \eqref{e.Lu2} is
$\cZ^1_S(Y,\Lie(Z/S))$ \cite[II 4.2]{SGA3}, where $Y$ acts on $\Lie(Z/S)$ by
the formula $y\mapsto \Ad(f(y))$.\footnote{For compatibility with \cite[II
4.1]{SGA3}, we write the adjoint action as left action.}

\begin{lemma}\label{l.Serre}
Let $f\colon Y\to Z$ be a homomorphism of $S$-groups as above. Let $c\colon
Z\to \cHom_{S\text{-}\group}(Y,Z)$ be the morphism given by
  \[z\mapsto (y\mapsto(z^{-1}f(y)z)).\]
  Then the composition
  \[\Lie(Z/S) \xto{T^1_{c/S}} T^f_{\cHom_{S\text{-}\group}(Y,Z)/S} \to \cHom_S(Y,\Lie(Z/S))\]
  is given by $t\mapsto(y\mapsto \Ad(f(y))t-t)$, and the image is $\cB_S^1(Y,\Lie(Z/S))$.
\end{lemma}

\begin{proof}
  The exact sequence
  \[1\to \Lie(Z/S) \to T_{Z/S} \to Z \to 1\]
  has a canonical splitting, which allows one to identify $T_{Z/S}$ with the semidirect product $\Lie(Z/S)\rtimes Z$.
An element $(t,z)$ of the semidirect product (evaluated at an $S$-scheme
$S'$) corresponds to the image of $dR_z(t)\in T^z_{Z/S}(S')$, where
$R_z\colon Z\times_S S'\to Z\times_S{S'}$ is the right translation by $z$.
  Multiplication in the semidirect product is given by
  \[(t,z)(t',z')=(t+\Ad(z)t',zz').\]
  The image of $t$ under $T_{c/S}^1$ in $T^f_{\cHom_{S}(Y,Z)/S}\xto{\sim} \cHom_{Z/S}(Y,T_{Z/S})$ is $T_{c/S}(t,1)$, given by
  \[y\mapsto (t,1)^{-1} (0,f(y)) (t,1) = (\Ad(f(y))t-t, f(y)).\]
  Hence the image in $\cHom_S(Y,\Lie(Z/S))$ is $y\mapsto \Ad(f(y))t-t$.
\end{proof}

\begin{proof}[Proof of Theorem \ref{t.Serre}]
We may assume $k$ algebraically closed and $G$ smooth. As in the beginning
of Section~\ref{s.3}, let $u\colon A\to G$ be a rational point of $X$, let
$H$ be the inertia at~$u$, let $Y=H\backslash G$, and let $c\colon G\to X$
be the $G$-equivariant morphism sending $g$ to $ug$, which factorizes
through an immersion $j\colon Y\to X$. Since $H^1(A,\Lie(G))=0$ for any
action of $A$ on $\Lie (G)$, it follows from Lemma \ref{l.Serre} that $T^1_c
\colon \Lie(G)\to T^u_X$ is an epimorphism. Thus the map $T^{\{H\}}_j\colon
T^{\{H\}}_Y\to T^u_X$ is an isomorphism. Since $Y$ is smooth
\cite[VI${}_{\textrm{B}}$ 9.2]{SGA3}, $j$ is \'etale \cite[Corollaire
17.11.2]{EGAIV} and hence an open immersion at this point. In other words,
the orbit of $u$ contains an open neighborhood of $u$. Since the rational
points of $X$ form a dense subset \cite[Corollaire 10.4.8]{EGAIV}, the
orbits of rational points form an open covering of $X$, which implies that
$X$ is smooth.
\end{proof}

\section{Structure theorems for equivariant cohomology rings}\label{s.4}

Throughout this section $\kappa$ is a field and $k$ is an algebraically
closed field.

\begin{definition}\label{s.cofinal}
For a functor $F\colon \cC\to \cD$ and an object $d$ of $\cD$, let
$(d\downarrow F)=\cC\times_\cD\cD_{d/}$ (strict fiber product) be the
category whose objects are pairs $(c,\phi)$ of an object $c$ of $\cC$ and a
morphism $\phi \colon d\to F(c)$ in $\cD$, and arrows are defined in the
natural way. Recall that $F$ is said to be \emph{cofinal} if, for every
object $d$ of $\cD$, the category $(d\downarrow F)$ is nonempty and
connected.
\end{definition}

If $F$ is cofinal and $G\colon \cD\to \cE$ is a functor such that
$\varinjlim GF$ exists, then $\varinjlim G$ exists and the morphism
$\varinjlim GF\to \varinjlim G$ is an isomorphism \cite[Theorem
IX.3.1]{MacLane}.

\begin{lemma}\label{l.cofinal}
Let $F\colon \cC\to \cD$ be a full and essentially surjective functor. Then
$F$ is cofinal.
\end{lemma}

\begin{proof}
Let $d$ be an object of $\cD$. As $F$ is essentially surjective, there exist
an object $c$ of $\cC$ and an isomorphism $f\colon d\simto F(c)$ in $\cD$,
which give an object of $(d\downarrow F)$. As $F$ is full, for any morphism
$g\colon d\to F(c')$, with $c'$ an object of $\cC$, there exists a morphism
$h\colon c\to c'$ in $\cC$ such that $F(h)=gf^{-1}$, which gives a morphism
$(c,f)\to (c',g)$ in $(d \downarrow F)$.
\end{proof}

We now introduce some enriched categories, which will be of use in the
structure theorems, especially Theorem \ref{t.main2}.

\begin{definition}\label{s.enrich}
Let $\cD$ be a category enriched in the category $\AlgSp_{/\kappa}$ of
algebraic $\kappa$-spaces, with Cartesian product as the monoidal operation
\cite[Section 1.2]{Kelly}. For objects $X$ and $Y$ of $\cD$,
$\Hom_{\cD}(X,Y)$ is an algebraic $\kappa$-space and composition of
morphisms in $\cD$ is given by morphisms of algebraic $\kappa$-spaces. We
denote by $\cD(\kappa)$ the category having the same objects as $\cD$, in
which
\[\Hom_{\cD(\kappa)}(X,Y)=(\Hom_{\cD}(X,Y))(\kappa).\]
Assume that $\kappa$ is separably closed. We denote by $\cD(\pi_0)$ the
category having the same objects as $\cD$, in which
\[\Hom_{\cD(\pi_0)}(X,Y)=\pi_0(\Hom_{\cD}(X,Y)).\]
Note that, if $\Hom_{\cD}(X,Y)$ is of finite type, $\Hom_{\cD(\pi_0)}(X,Y)$
is a finite set. We have a functor
\[\eta\colon \cD(\kappa)\to \cD(\pi_0),\]
which is the identity on objects, and sends $f\in \Hom_{\cD}(X,Y)(\kappa)$
to the connected component containing it. Assume that $\Hom_{\cD}(X,Y)$ is
locally of finite type. If $\kappa$ is algebraically closed, or if for all
$X$, $Y$ in $\cD$, $\Hom_{\cD}(X,Y)$ is smooth over $\kappa$, then $\eta$ is
full, hence cofinal by Lemma \ref{l.cofinal}.
\end{definition}

\begin{construction}\label{s.GX}
Let $G$ be an algebraic group over $k$, let $X$ be an algebraic space of
finite presentation over~$k$, endowed with an action of $G$, and let $\ell$
be a prime number. We define a category enriched in the category $\FTSch{k}$
of schemes of finite type over $k$,
\[\cA_{G,X,\ell},\]
as follows. Objects of $\cA_{G,X,\ell}$ are pairs $(A,C)$ where $A$ is an
elementary abelian $\ell$-subgroup of $G$ and $C$ is a connected component
of the algebraic space of fixed points $X^A$ (which is a closed algebraic
subspace of $X$ if $X$ is separated). For objects $(A,C)$ and $(A',C')$ of
$\cA_{G,X}$, we denote by $\Trans_G((A,C), (A',C'))$ the \emph{transporter}
of $(A,C)$ into $(A',C')$, namely the closed subgroup scheme of $G$
representing the functor
\[
S\mapsto \{g \in G(S) \mid g^{-1}A_S g  \subset A'_S, C_Sg\supset C'_S\}.
\]
In fact, $\Trans_G((A,C), (A',C'))$ is a closed and open subscheme of the
scheme $\Trans_G(A,A')$ defined by the cartesian square
\[\xymatrix{\Trans_G(A,A')\ar[r]\ar[d] & \prod_{a\in A} A'\ar[d]\\
G\ar[r] &\prod_{a\in A} G}
\]
where the lower horizontal arrow is given by $g\mapsto (g^{-1}ag)_{a\in A}$.
Indeed, if we consider the morphism
\[F\colon \Trans_G(A,A')\times X^{A'}\to X^A\quad (g,x)\mapsto xg^{-1}\]
and the induced map
\[\phi\colon \pi_0(\Trans_G(A,A'))\to \pi_0(X^A)\quad \Gamma \mapsto \pi_0(F)(\Gamma,C'),
\]
then $\Trans_G((A,C), (A',C'))$ is the union of the connected components of
$\Trans_G(A,A')$ corresponding to $\phi^{-1}(C)$. We define
\[\Hom_{\cA_{G,X,\ell}}((A,C),(A',C')) \coloneqq \Trans_G((A,C),(A',C')).\]
Composition of morphisms is given by the composition of transporters
\[\Trans_G((A',C'),(A'',C''))\times \Trans_G((A,C),(A',C'))\to \Trans_G((A,C),(A'',C'')),\]
which is a morphism of $k$-schemes. When no confusion arises, we omit $\ell$
from the notation. We will denote $\cA_{G,\Spec(k)}$ by $\cA_G$.

For an object $(A,C)$ of $\cA_{G,X}$, we denote by $\Cent_G(A,C)$ its
\emph{centralizer}, namely the closed subscheme of $G$ representing the
functor
\[
S\mapsto \{g \in G(S) \mid C_Sg= C_S\text{ and $g^{-1}a g  =a$ for all $a\in A$}\}.
\]
For objects $(A,C)$, $(A',C')$ of $\cA_{G,X}$, we have natural injections
(cf.\ \cite[(8.2)]{Quillen2})
\begin{equation}\label{e.tranovercent}
\Cent_G(A,C)\backslash \Trans_G((A,C), (A',C')) \to \Cent_G(A) \backslash \Trans_G(A,A') \to \Hom(A,A').
\end{equation}
We let $\cA_{G,X}^\flat$ denote the category having the same objects as
$\cA_{G,X}$, but with morphisms defined by the left hand side of
\eqref{e.tranovercent}. We call the finite group
\begin{equation}\label{e.Weyl}
W_G(A,C) \coloneqq \Cent_G(A,C)\backslash \Trans_G((A,C),(A,C))
\end{equation}
the \emph{Weyl group} of $(A,C)$. This is a subgroup of the finite group
\[W_G(A)=\Cent_G(A)\backslash\Norm_G(A)\subset \Aut(A).\]
The functors
\[\cA_{G,X}(k)\to \cA_{G,X}(\pi_0)\to \cA_{G,X}^\flat\]
(the second one defined via \eqref{e.tranovercent}) are cofinal by Lemma
\ref{l.cofinal}.

Let $k'$ be an algebraically closed extension of $k$. We have a functor
$\cA_{G,X}(k)\to \cA_{G_{k'},X_{k'}}(k')$ carrying $(A,C)$ to $(A,C_{k'})$.
Since the map
\[\pi_0(\Trans_G((A,C),(A',C'))\to \pi_0(\Trans_{G_{k'}}((A,C_{k'}),(A',C'_{k'})))\]
is a bijection, this induces a functor $\cA_{G,X}(\pi_0)\to
\cA_{G_{k'},X_{k'}}(\pi_0)$.
\end{construction}

In the rest of the section we assume $\ell$ invertible in $k$.

\begin{lemma}\label{l.Afincat}
The category $\cA_{G,X}(\pi_0)$ is essentially finite, and the functor
$\cA_{G,X}(\pi_0)\to \cA_{G_{k'},X_{k'}}(\pi_0)$ is an equivalence. In
particular, $\cA_G(\pi_0)$ is essentially finite.
\end{lemma}

\begin{proof}
Let $S$ be a set of representatives of isomorphisms classes of objects of
$\cA_{G}(\pi_0)$. In other words, $S$ is a set of representatives of
conjugacy classes of elementary abelian $\ell$-subgroups of $G$. By
Corollary \ref{c.Serre}, this is a finite set. Let $T$ be the set of objects
$(A,C)$ of $\cA_{G,X}(\pi_0)$ such that $A\in S$. Then $T$ is a finite set.
The conclusion follows from the following facts:
\begin{enumerate}
\item For $(A,C)$ and $(A',C')$ in $\cA_{G,X}$,
    $\Hom_{\cA_{G,X}(\pi_0)}((A,C),(A',C'))$ is finite (Definition
    \ref{s.enrich}), and, by Construction \ref{s.GX},
    \[\Hom_{\cA_{G,X}(\pi_0)}((A,C),(A',C'))\simto
    \Hom_{\cA_{G_{k'},X_{k'}}(\pi_0)}((A,C_{k'}),(A',C'_{k'})).
    \]
\item The finite set $T$ is a set of representatives of isomorphism
    classes of objects of $\cA_{G,X}(\pi_0)$, and $\{(A,C_{k'})\mid
    (A,C)\in T\}$ is a set of representatives of isomorphism classes of
    objects of $\cA_{G_{k'},X_{k'}}(\pi_0)$.
\end{enumerate}
Indeed, (b) follows from the following obvious lemma.
\end{proof}

\begin{lemma}\label{l.fincat}
Let $B$, $C$ be sets endowed with equivalence relations denoted by $\simeq$
and let $f\colon B\to C$ be a map such that $b\simeq b'$ implies $f(b)\simeq
f(b')$. Let $S$ be a set of representatives of $C$. For every $s\in S$, let
$T_s$ be a set of representatives of $f^{-1}(s)$. Then $\bigcup_{s\in S}
T_s$ is a set of representatives of $B$ if and only if for every $b\in B$
and every $c\in S$ such that $f(b)\simeq c$, there exists $b'\in f^{-1}(c)$
such that $b\simeq b'$.
\end{lemma}

\begin{remark}
Let $G$ be an algebraic group over $k$ and let $T$ be a subtorus of $G$.
Then $W_G(T)=\Cent_G(T)\backslash \Norm_G(T)$ is a finite subgroup of
$\Aut(T)$. The inclusions
\[\Norm_G(T)\subset \Norm_G(T[\ell]), \quad \Cent_G(T)\subset \Cent_G(T[\ell])\]
induce a homomorphism $\rho\colon W_G(T)\to W_G(T[\ell])$. Via the
isomorphisms $\Aut(T)\simeq\Aut(M)$ and $\Aut(T[\ell])\simeq \Aut(M/\ell
M)$, where $M=\X^*(T)$, $\rho$ is compatible with the reduction homomorphism
$\Aut(M)\to \Aut(M/\ell M)$. If $T$ is a maximal torus, then $\rho$ is
surjective by the proof of \cite[1.1.1]{SerreBour}.

For $\ell>2$, $\rho$ is injective. In fact, for an element $g$ of
$\Ker(\Aut(M)\to \Aut(M/\ell M))$ and arbitrary $\ell$, the $\ell$-adic
logarithm $\log(g) \coloneqq \sum_{m=1}^{\infty} \frac{(-1)^{m-1}}{m}
(g-1)^{m}\in \ell \End(M)\otimes \Z_\ell$ is well defined. If $g^n=\id$ for
some $n\ge 1$, then $n\log (g)=\log(g^n)=0$, so that $\log(g)=0$. In the
case $\ell>2$, we then have $g=\exp\log(g)=\id$. For $\ell=2$, $\rho$ is not
injective in general. For example, if $G=\mathrm{SL}_2$ and $T$ is a maximal
torus, then $W_G(T)\simeq \Z/2$ and $W_G(T[2])=\{1\}$.

If $G=\GL_n$ and $T$ is a maximal torus, then $\rho$ is an isomorphism for
arbitrary $\ell$. In fact, in this case, $\Norm_G(T)=\Norm_G(T[\ell])$ and
$\Cent_G(T)=\Cent_G(T[\ell])$.
\end{remark}

\begin{notation}\label{s.H*}
We will sometimes omit the constant coefficient $\F_\ell$ from the notation.
We will sometimes write $H^*_G$ for $H^*(BG)=H^*(BG,\F_\ell)$.
\end{notation}

\begin{construction}\label{s.HAC}
Let $T=\Trans_G(A,A')$, let $g\in T(k)$, and let $c_g\colon A\to A'$ be the
map $a\mapsto g^{-1}ag$. In the above notation, the morphism $Bc_g\colon
BA\to BA'$ induces a homomorphism $\theta_g\colon H^*_{A'}\to H^*_A$. This
defines a presheaf $(H^*_A,\theta_g)$ on $\cA_G^\flat$, hence on
$\cA_{G,X}^\flat$.

If $(A,C)$ is an object of $\cA_{G,X}$, we have
\[
H^*([C/A]) = H^*_A \otimes H^*(C).
\]
The restriction $H^*([X/G]) \rightarrow H^*([C/A])$ induced by the inclusion
$(C,A)\to (X,G)$, composed with the projection
\begin{equation}\label{e.proj}
H^*([C/A])\to H^*_A
\end{equation}
induced by $H^*(C) \rightarrow H^0(C) = {\F}_{\ell}$, defines a homomorphism
\begin{equation}\label{e.Ac*}
(A,C)^* \colon H^*([X/G]) \rightarrow H^*_A .
\end{equation}
For $g\in \Trans((A,C),(A',C'))(k)\subset T(k)$, we have the following
2-commutative square of groupoids in the category $\AlgSp_{/U}$
(Construction \ref{s.Eq}):
\[\xymatrix{(C',A)_\bullet \ar[r]^{(\id,c_g)}\ar[d]_{(g^{-1},\id)}\drtwocell\omit{_} &(C',A')_\bullet\ar[d]\\
(C,A)_\bullet\ar[r] & (X,G)_\bullet}
\]
(with trivial action of $A$ and $A'$ on $C'$ and trivial action of $A$ on
$C$), where the 2-morphism is given by $g$. The corresponding 2-commutative
square of Artin stacks
\[\xymatrix{BA\times C' \ar[r]\ar[d] & BA'\times C'\ar[d]\\
BA\times C \ar[r] & [X/G]}
\]
induces by adjunction (Notation \ref{s.fCart}) the following commutative
square:
\[\xymatrix{H^*([X/G])\ar[r]\ar[d] & H^*([C/A])\ar[d]^{[g^{-1}/\id]^*}\\
H^*([C'/A'])\ar[r]^{[\id/c_g]^*} & H^*([C'/A]).}
\]
Composing with the projections \eqref{e.proj}, we obtain the following
commutative diagram:
\[\xymatrix{H^*([X/G])\ar[d]_{(A',C')^*}\ar[rd]^{(A,C)^*} \\
H^*_{A'}\ar[r]_{\theta} & H^*_{A}.}\]
Therefore the maps $(A,C)^*$ \eqref{e.Ac*} define a homomorphism
\begin{equation}\label{e.aGX}
a(G,X) \colon H^*([X/G]) \rightarrow \varprojlim_{\cA_{G,X}^\flat}
(H^*_A, \theta_g).
\end{equation}
Note that the right-hand side is the equalizer of
\[(j_1,j_2)\colon \prod_{(A,C)\in \cA_{G,X}} H_A^*\rightrightarrows \prod_{g\colon (A,C)\to (A',C')}H_A^*,\]
where $g$ runs through morphisms in $\cA_{G,X}^\flat$,
$j_1(h_{(A,C)})=(h_{(A,C)})_g$, $j_2(h_{(A,C)})=(\theta_g h_{(A',C')})_g$.
Moreover, by the finiteness results Corollary \ref{c.finite} and Lemma
\ref{l.Afincat}, the right-hand side of \eqref{e.aGX} is a finite
$H^*(BG)$-module, and, in particular, a finitely generated
$\F_\ell$-algebra.
\end{construction}

To state our main result for the map $a(G,X)$ \eqref{e.aGX}, we need to
recall the notion of uniform $F$-isomorphism. For future reference, we give
a slightly extended definition as follows.

\begin{definition}\label{s.grvec}
Let $\grvec$ be the category of graded $\F_\ell$-vector spaces. It is an
$\F_\ell$-linear $\otimes$-category. The commutativity constraint of
$\grvec$ follows Koszul's rule of signs, such that a (pseudo-)ring in
$\grvec$ is an anti-commutative graded $\F_\ell$-(pseudo-)algebra.

Let $\cC$ be a category. As a special case of Construction \ref{s.rtopos},
the functor category $\grvec^\cC:=\mathrm{Fun}(\cC^{op},\grvec)$ is a
$\F_\ell$-linear $\otimes$-category. The functor $\varprojlim_{\cC}\colon
\grvec^\cC \to \grvec$ is the right adjoint to the unital $\otimes$-functor
$\grvec \to \grvec^\cC$, thus has a right unital $\otimes$-structure. If $u
\colon R \rightarrow S$ is a homomorphism of pseudo-rings in $\grvec^\cC$,
we say that $u$ is a \emph{uniform $F$-injection} (resp.\ \emph{uniform
$F$-surjection}) if there exists an integer $n\ge 0$ such that for any
object $i$ of $\cC$ and any homogeneous element (or, equivalently, any
element) $a$ in the kernel of $u_i$ (resp.\ in $S_i$), $a^{\ell^n} = 0$
(resp.\ $a^{\ell^n}$ is in the image of $u_i$). Note that $a^{\ell^n}=0$ for
some $n \ge 0$ is equivalent to $a^m=0$ for some $m\ge 1$. We say $u$ is a
\emph{uniform $F$-isomorphism} if it is both a uniform $F$-injection and a
uniform $F$-surjection. These definitions apply in particular to $\grvec$ by
taking $\cC$ to be a discrete category of one object, in which case the
notion of a uniform $F$-isomorphism coincides with the definition in
\cite[Section~3]{Quillen1}.
\end{definition}

The following result is an analogue of Quillen's theorem (\cite[Theorem
6.2]{Quillen1}, \cite[Theorem 8.10]{Quillen2}):

\begin{theorem}\label{c.main}
Let $X$ be a separated algebraic space of finite type over $k$, and let $G$
be an algebraic group over $k$ acting on $X$. Then the homomorphism $a(G,X)$
\eqref{e.aGX} is a uniform $F$-isomorphism (Definition \ref{s.grvec}).
\end{theorem}

\begin{remark}\label{r.cohBA}
Let $A$ be an elementary abelian $\ell$-group of rank~$r\ge 0$. We identify
$H^1(BA,{\mathbb{F}}_{\ell})$ with $\check A = \Hom(A,{\mathbb{F}}_{\ell})$.
Recall \cite[Section~4]{Quillen1} that we have a natural identification of
${\mathbb{F}}_{\ell}$-graded algebras
\[
H^*(BA,{\mathbb{F}}_{\ell}) = \begin{cases}\Sym(\check A) & \text{if $\ell = 2$} \\
\wedge(\check A) \otimes \Sym(\beta \check A)  & \text{if $\ell >
2$,}\end{cases}\label{(7.1.1)}
\]
where $\Sym$ (resp.\ $\wedge$) denotes a symmetric (resp.\ exterior) algebra
over $\F_\ell$, and $\beta\colon \check A\to H^2(BA,\F_\ell)$ is the
Bockstein operator. In particular, if $\{x_1,\dots, x_r\}$ is a basis of
$\check A$ over ${\mathbb{F}}_{\ell}$, then
\[
H^*(BA,{\mathbb{F}}_{\ell}) = \begin{cases}{\mathbb{F}}_{\ell}[x_1,\dots,x_r] & \text{if $\ell = 2$} \\
\wedge(x_1,\dots,x_r) \otimes {\mathbb{F}}_{\ell}[y_1,\dots,y_r]  & \text{if $\ell > 2$}\end{cases}\label{(7.1.2)}
\]
where $y_i = \beta x_i$.
\end{remark}

\begin{cor}\label{c.PS}
With $X$ and $G$ as in Theorem \ref{c.main}, let $K \in
D^b_c([X/G],\F_{\ell})$. The Poincar\'e series
\[
\PS_t(H^*([X/G],K)) = \sum_{i \ge 0} \dim_{\F_{\ell}}H^i([X/G],K)t^i
\]
is a rational function of $t$ of the form $P(t)/\prod_{1 \le i \le
n}(1-t^{2i})$, with $P(t) \in \Z[t]$. The order of the pole of
$\PS_t(H^*([X/G]))$ at $t = 1$ is the maximum rank of an elementary abelian
$\ell$-subgroup $A$ of $G$ such that $X^A \neq \emptyset$.
\end{cor}

\begin{proof}
By Theorem \ref{t.finite}, $H^*([X/G],K)$ is a finitely generated module
over $H^*([X/G])$, which is a finitely generated algebra over $\F_{\ell}$.
Therefore the Poincar\'e series $\PS_t(H^*([X/G],K))$ is a rational function
of $t$, and the order of the pole at $t = 1$ of $\PS_t(H^*([X/G]))$ is equal
to the dimension of the commutative ring $H^{2*}([X/G])$. To show that
$\PS_t(H^*([X/G],K))$ is of the form given in Corollary \ref{c.PS}, recall
(Remark \ref{r.finite}) that we have shown in the proof of Theorem
\ref{t.finite} that $H^*([X/G],K)$ is a quotient $H^*(BH)$-module of
$H^*([X/H],f^*K)$ for a certain affine subgroup $H$ of $G$, $f$ denoting the
canonical morphism $[X/H] \to [X/G]$. Embedding $H$ into some $\GL_n$ and
applying Corollary \ref{t.finite2}, we deduce that $H^*([X/G],K)$ is a
finite $H^*(B\mathrm{GL}_n)$-module. Since $H^*(B\mathrm{GL}_n)\simeq
\F_\ell[c_1,\dots,c_n]$, where $c_i$ is of degree $2i$ (Theorem
\ref{l.finite}), $\PS_t(M^*)$ is of the form $P(t)/\prod_{1 \le i \le
n}(1-t^{2i})$ with $P(t) \in \Z[t]$ for every finite graded
$H^*(B\GL_n)$-module $M^*$. The last assertion of Corollary \ref{c.PS} is
derived from Theorem \ref{c.main} as in \cite[Theorem 7.7]{Quillen1}. One
can also see it in a more geometric way, observing that the reduced spectrum
of $H^{\varepsilon *}([X/G])$ (where $\varepsilon = 1$ if $\ell = 2$ and 2
otherwise) is homeomorphic to an amalgamation of standard affine spaces
$\underline{A} = \Spec (H^{\varepsilon *}_A)_{\mathrm{red}}$ associated with
the objects $(A,C)$ of $\mathcal{A}_{G,X}$ (see Construction \ref{c.amalg}).
\end{proof}

\begin{example}
Let $G$ be a connected reductive group over $k$ with no $\ell$-torsion, and
let $T$ be a maximal torus of $G$. Let $\iota\colon \cA'\hookrightarrow
\cA_G^\flat$ be the full subcategory spanned by $T[\ell]$. The functor
$\iota$ is cofinal. Indeed, for every object $A$ of $\cA_G^\flat$, since $A$
is toral, there exists a morphism $c_g\colon A\to T[\ell]$ in $\cA_G^\flat$.
Moreover, for morphisms $c_{g}\colon A\to T[\ell]$, $c_{g'}\colon A\to
T[\ell]$ in $\cA_G^\flat$, there exists an isomorphism $c_{h}\colon
T[\ell]\to T[\ell]$ such that $c_h c_{g}=c_{g'}$ in $\cA_G^\flat$, by
\cite[1.1.1]{SerreBour} applied to the conjugation $c_{g^{-1}g'}\colon
c_g(A)\to c_{g'}(A)$. Let $W= W_G(T)$. The map $a(G,\Spec(k))$ can be
identified with the injective $F$-isomorphism
\[H^*_G\simeq (H^*_T)^W \to (H^*_{T[\ell]})^W\]
induced by restriction (where the isomorphism is \eqref{e.GTW}). In
particular,
\[\varprojlim_{A\in
\cA_{G}^\flat}(H^{\varepsilon*}_A)_{\red}\simeq \Sym(T[\ell]\spcheck)^W,\]
where $\varepsilon=1$ if $\ell=2$ and $\varepsilon=2$ if $\ell>2$. Moreover,
for $\ell>2$, $a(G,\Spec(k))$ induces an isomorphism $H^{2*}_G\simeq
((H^{2*}_{T[\ell]})_{\red})^W$.
\end{example}

\begin{example}
Let $X=X(\Sigma)$ be a toric variety over $k$ with torus $T$, where $\Sigma$
is a fan in $N\otimes \R$ and $N=\X_*(T)$. We identify $T[\ell]$ with
$N\otimes \mu_\ell$. The inertia $I_\sigma\subset T$ of the orbit $O_\sigma$
corresponding to a cone $\sigma\in \Sigma$ is $N_\sigma\otimes \Gm$, where
$N_\sigma$ is the sublattice of $N$ generated by $N\cap \sigma$, so that
$A_\sigma=I_\sigma[\ell]\simeq N_\sigma \otimes \mu_\ell$. The latter can be
identified with the image of $N\cap \sigma$ in $N\otimes \F_\ell$. This
defines an object $(A_\sigma, C_\sigma)$ of $\cA_{T,X}$, where $C_\sigma$ is
the connected component of $X^{A_\sigma}$ containing $O_\sigma$. The functor
$\Sigma\to \cA_{T,X}^\flat$ is cofinal. Thus we have a canonical isomorphism
\[\varprojlim_{A\in \cA_{T,X}^\flat} (H^{\varepsilon*}_A)_{\red}\simeq \varprojlim_{\sigma\in \Sigma} (H^{\varepsilon*}_{A_\sigma})_{\red}.\]
Note that $(H^*_{A_\sigma})_{\red}$ can be canonically identified with
$\Sym(M_\sigma)\otimes \F_\ell$, where $M_\sigma=M/(M\cap \sigma^\perp)$ and
$\Sym(M_\sigma)$ is the algebra of integral polynomial functions on
$\sigma$. In particular, we have a canonical isomorphism
\begin{equation}\label{e.toric}
\varprojlim_{A\in \cA_{T,X}^\flat}(H^{\varepsilon*}_A)_{\red} \simeq
\PP^*(\Sigma)\otimes \F_\ell,
\end{equation}
where
\[\PP^*(\Sigma)=\{f\colon \Supp(\Sigma) \to \R \mid (f\res \sigma) \in \Sym(M_\sigma)\text{ for each }\sigma\in \Sigma\}\]
is the algebra of piecewise polynomial functions on $\Sigma$.  Recall that
Payne established an isomorphism from the integral equivariant Chow
cohomology ring $A_T^*(X)$ of Edidin and Graham \cite[2.6]{EG} onto
$\PP^*(\Sigma)$ \cite[Theorem~1]{Payne}. Combining Theorem \ref{c.main} and
\eqref{e.toric}, we obtain a uniform $F$-isomorphism
\[H^*([X/T],\F_\ell)\to \PP^*(\Sigma)\otimes \F_\ell.\]
If $X$ is smooth, this is an isomorphism, and $\PP^*(\Sigma)$ is isomorphic
to the Stanley-Reisner ring of~$\Sigma$ \cite[Section~4]{BDP}.
\end{example}

In the rest of this section, we state an analogue of Theorem \ref{c.main}
with coefficients.

\begin{construction}\label{c.6.15}
Let $G$ be an algebraic group over $k$, $X$ an algebraic $k$-space endowed
with an action of $G$, and $K \in D^+_\cart([X/G],\F_{\ell})$.

If $A$, $A'$ are elementary abelian $\ell$-subgroups of $G$ and $g \in G(k)$
conjugates $A$ into $A'$ (i.e.\ $g^{-1}Ag \subset A'$), $A$ acts trivially
on $X^{A'}$ via $c_g = A \to A'$ (where $c_g$ is the conjugation $s \mapsto
g^{-1}sg$), and we have an equivariant morphism $(1,c_g)\colon (X^{A'},A)
\to (X,G)$, where $1$ denotes the inclusion $X^{A'}\subset X$, inducing
\[
[1/c_g] \colon [X^{A'}/A] = BA \times X^{A'} \to [X/G].
\]
We thus have, for all $q$, a restriction map
\[
H^q([X/G],K) \to H^q([X^{A'}/A],[1/c_g]^*K).
\]
On the other hand, we have a natural projection
\[
\pi \colon [X^{A'}/A] = BA \times X^{A'} \to X^{A'},
\]
hence an edge homomorphism for the corresponding Leray spectral sequence
\[
H^q([X^{A'}/A],[1/c_g]^*K) \to H^0(X^{A'},R^q\pi_*[1/c_g]^*K).
\]
By composition we get a homomorphism
\begin{equation}\label{e.6.15.1}
a^q(A,A',g) \colon
H^q([X/G],K) \to H^0(X^{A'},R^q\pi_*[1/c_g]^*K).
\end{equation}
Since $R^*\pi_*\F_\ell=\bigoplus_{q} R^q\pi_*\F_\ell$ is a constant sheaf of
value $H^*(BA,\F_\ell)$, $R^*\pi_*[1/c_g]^*K=\bigoplus_q R^q\pi_*[1/c_g]^*K$
is endowed with a $H^*(BA,\F_\ell)$-module structure by Constructions
\ref{s.tensor} and \ref{s.rtopos}, which induces a $H^*(BG,\F_\ell)$-module
structure via the ring homomorphism $[1/c_g]^*\colon H^*(BG,\F_\ell)\to
H^*(BA,\F_\ell)$. The map $a(A,A',g)=\bigoplus_q a^q(A,A',g)$ is
$H^*(BG,\F_\ell)$-linear.

If $(Z,Z',h)$ is a second triple consisting of elementary abelian
$\ell$-subgroups $Z$, $Z'$, and $h \in G(k)$ such that $c_h \colon Z \to
Z'$, the datum of elements $a$ and $b$ of $G(k)$ such that $g = ahb$ and
$c_a \colon A \to Z$, $c_b \colon Z' \to A'$, defines a commutative diagram
\begin{equation}\label{e.6.15.2}
\xymatrix{A \ar[r]^{c_g} \ar[d]_{c_a}& A' \\
Z \ar[r]^{c_h} &Z', \ar[u]_{c_b}}
\end{equation}
hence a morphism $[b^{-1}/c_a] \colon [X^{A'}/A] \to [X^{Z'}/Z]$, fitting
into a 2-commutative diagram
\begin{equation}\label{e.6.15.3}
\xymatrix{{} & [X^{A'}/A] \ar[d]^{[b^{-1}/c_a]} \ar[dl]_{[1/c_g]} \ar[r]^-{\pi} & X^{A'} \ar[d]^{b^{-1}}\\
[X/G]\urtwocell\omit{<2>} & [X^{Z'}/Z] \ar[l]^{[1/c_h]} \ar[r]^-{\pi} & X^{Z'},}
\end{equation}
where the 2-morphism of the triangle is induced by $b$. Consider the
homomorphism
\begin{equation}\label{e.6.15.4}
(a,b)^* \colon H^0(X^{Z'},R^q\pi_*[1/c_h]^*K) \to H^0(X^{A'},(b^{-1})^*R^q\pi_*[1/c_h]^* K)\to H^0(X^{A'},R^q\pi_*[1/c_g]^*K),
\end{equation}
where the first map is adjunction by $b^{-1}$ and the second map is base
change map for the square in \eqref{e.6.15.3}. This fits into a commutative
triangle
\begin{equation}\label{e.6.15.5}
\xymatrix{H^q([X/G],K) \ar[d] \ar[dr] & {} \\
H^0(X^{Z'},R^q\pi_*[1/c_h]^*K) \ar[r]^{(a,b)^*} & H^0(X^{A'},R^q\pi_*[1/c_g]^*K),}
\end{equation}
where the vertical and oblique maps are given by \eqref{e.6.15.1}. Denote by
\begin{equation}\label{e.6.15.6}
\mathcal{A}_G(k)^{\natural}
\end{equation}
the following category. Objects of $\mathcal{A}_G(k)^{\natural}$ are triples
$(A,A',g)$ as above, morphisms $(A,A',g) \to (Z,Z',h)$ are pairs $(a,b) \in
G(k) \times G(k)$ such that $g = ahb$ and $c_a \colon A \to Z$, $c_b \colon
Z' \to A'$. Via the maps $(a,b)^*$ \eqref{e.6.15.4}, the groups
$H^0(X^{A'},R^q\pi_*[1/c_g]^*K)$ form a projective system indexed by
$\mathcal{A}_G(k)^{\natural}$, and by the commutativity of \eqref{e.6.15.5}
we get a homomorphism
\begin{equation}\label{e.aE}
 a^q_G(X,K) \colon H^q([X/G],K) \to R^q_G(X,K),
\end{equation}
where
\begin{equation}\label{e.RE}
R^q_G(X,K) \coloneqq \varprojlim_{(A,A',g) \in \mathcal{A}_G(k)^{\natural}} H^0(X^{A'},R^q\pi_*[1/c_g]^*K).
\end{equation}
Since $\bigoplus_q (a,b)^*$ is $H^*(BG,\F_\ell)$-linear, $R^*_G(X,K)
\coloneqq \bigoplus_{q}R^q_G(X,K)$ is endowed with a structure of
$H^*(BG,\F_\ell)$-module. The map
\begin{equation}\label{e.aEstar}
a_G(X,K) = \bigoplus_q a^q_G(X,K) \colon H^*([X/G],K) \to R^*_G(X,K)
\end{equation}
induced by \eqref{e.aE} is a homomorphism of $H^*(BG,\F_\ell)$-modules. If
$K$ is a (pseudo-)ring in $D^+_\cart([X/G],\F_\ell)$, $R^*_G(X,K)$ is a
$\F_\ell$-(pseudo-)algebra and $a_G(X,K)$ is a homomorphism of
$\F_\ell$-(pseudo-)algebras.
\end{construction}

\begin{theorem}\label{t.main2}
Let $G$ be an algebraic group over $k$, $X$ a separated algebraic space of
finite type over $k$ endowed with an action of $G$, and $K\in
D^+_c([X/G],\F_\ell)$.
\begin{enumerate}
\item $R^q_G(X,K)$ is a finite-dimensional $\F_{\ell}$-vector space for
    all $q$; if $K \in D^b_c([X/G],\F_{\ell})$, $R^*_G(X,K)$ is a finite
    module over $H^*(BG,\F_{\ell})$.

\item If $K$ is a pseudo-ring in $D^+_c([X/G],\F_{\ell})$ (Construction
    \ref{c.Dring}), the kernel of the homomorphism $a_G(X,K)$
    \eqref{e.aEstar} is a nilpotent ideal of $H^*([X/G],K)$. If, moreover,
    $K$ is commutative, then $a_G(X,K)$ is a uniform $F$-isomorphism
    (Definition \ref{s.grvec}).
\end{enumerate}
\end{theorem}

\begin{remark}\label{s.bemol}
The projective limit in \eqref{e.aE} is the equalizer of the double arrow
\[(j_1,j_2)\colon \prod_{A\in \cA_G}\Gamma(X^{A},R^q\pi_{A*}[1/c_1]^*K) \rightrightarrows \prod_{(A,A',g)\in \cA_G(k)^\natural}\Gamma(X^{A'},R^q\pi_{(A,A',g)*} [1/c_g]^*K),\]
where $\pi_A=\pi_{(A,A,1)}$, $[1/c_1]\colon [X^A/A]\to [X/G]$, $j_1$ is
induced by $(1,g)\colon (A,A',g)\to (A,A,1)$ and $j_2$ is induced by
$(g,1)\colon (A,A',g)\to (A',A',1)$.

This is a consequence of the following general fact (applied to
$\cC=\cA_{G}(k)$). Let $\cC$ be a category. Define a category $\cC^\natural$
as follows. The objects of $\cC^\natural$ are the morphisms $A\to A'$ of
$\cC$. A morphism in $\cC^\natural$ from $A\to A'$ to $Z\to Z'$ is a pair of
morphisms $(A\to Z, Z'\to A')$ in $\cC$ such that the following diagram
commutes:
\[\xymatrix{A\ar[r]\ar[d] &A'\\
Z\ar[r]& Z'.\ar[u]}
\]
Let $\cF$ be a presheaf of sets on $\cC^\natural$.  Then the sequence
  \[\Gamma(\widehat{\cC^\natural}, \cF)\to \prod_{A\in \cC}\cF(\id_A)\rightrightarrows \prod_{(a\colon A\to A')\in \cC^\natural}\cF(a)\]
is exact. Here the two projections are induced by $(\id_A,a)\colon a\to
\id_A$ and $(a,\id_{A'})\colon a\to \id_{A'}$, respectively.

Indeed, because the two compositions are equal, we have a map $s\colon
\Gamma(\widehat{\cC^\natural},\cF)\to K$, where $K$ is the equalizer of the
double arrow. It is straightforward to check that the map $K\to \prod_{a\in
\cC^\natural} \cF(a)$ factors through $\Gamma(\widehat{\cC^\natural},\cF)$
to give the inverse of $s$.

Note that this statement generalizes the calculation of \emph{ends}
$\int_{A\in \cC}F(A,A)$ \cite[Section IX.5]{MacLane} of a functor $F$ from
$\cC^\op \times \cC$ to the category of sets. More generally, for any
category $\cD$ and any functor $F\colon \cC^\op \times \cC\to \cD$,
$\int_{A\in \cC} F(A,A)$ can be identified with the limit
$\varprojlim_{a\colon A\to A'} F(A,A')$ indexed by $\cC^\natural$.
\end{remark}

\begin{remark}
For $K=\F_\ell$, the commutative diagram
\[\xymatrix{H^*([X/G],\F_\ell)\ar[r]\ar[rd]
& \prod_{(A,C)\in \cA_{G,X}}H^*_A\ar@<.5ex>[r]\ar@<-.5ex>[r]\ar[d]^{\simeq} &\prod_{g\colon (A,C)\to (A',C')} H^*_A\ar[d]^{\simeq}\\
& \prod\limits_{A\in \cA_G}\Gamma(X^{A},R^q\pi_{A*}\F_\ell) \ar@<.5ex>[r]\ar@<-.5ex>[r]
&\prod\limits_{(A,A',g)\in \cA_G(k)^\natural}\Gamma(X^{A'},R^q\pi_{(A,A',g)*} \F_\ell)}
\]
induces a commutative diagram
\[\xymatrix{H^*([X/G],\F_\ell)\ar[r]^{a(G,X)}\ar[rd]_{a_G(X,\F_\ell)} & \varprojlim_{\cA_{G,X}^\flat} H^*_A\ar[d]^\simeq\\
& R^*_G(X,\F_\ell).}
\]
Therefore Theorem \ref{t.main2} generalizes Theorem \ref{c.main}.
\end{remark}

Part (b) of Theorem \ref{t.main2} will be proved as a corollary of a more
general structure theorem (Theorem \ref{p.str}). Part (a) will follow from
the next lemma.

\begin{lemma}\label{l.6.18}
Let $\cE_G$ be the category enriched in $\FTSch{k}$ having the same objects
as $\cA_G(k)^\natural$ and in which $\Hom_{\cE_G}((A,A',g), (Z,Z',h))$ is
the subscheme of $G\times G$ representing the presheaf of sets on
$\AlgSp_{/k}$:
\[S\mapsto \{(a,b)\in (G\times G)(S)\mid a^{-1}A_Sa\subset Z_S, b^{-1}Z'_Sb\subset A'_S, g=ahb\}\]
(so that by definition $\mathcal{E}_G(k)= \mathcal{A}_G(k)^{\natural}$).
\begin{enumerate}
\item The functor $F\colon\mathcal{E}_{G}(\pi_0)\to \cA_G(\pi_0)^\natural$
    carrying $(A,A',g)$ to $(A,A',\gamma)$, where $\gamma$ is the
    connected component of $\Trans_G(A,A')$ containing $g$, is an
    equivalence of categories. In particular, $\cE_G(\pi_0)$ is equivalent
    to a finite category, and for every algebraically closed extension
    $k'$ of $k$, the natural functor $\cE_G(\pi_0)\to \cE_{G_{k'}}(\pi_0)$
    is an equivalence of categories.

\item The projective system $H^0(X^{A'},R^q\pi_*[1/c_g]^*K)$ indexed by
    $(A,A',g) \in \mathcal{A}_G(k)^{\natural}$ factors through
    $\mathcal{E}_{G}(\pi_0)$.
\end{enumerate}
\end{lemma}

\begin{remark}
The projective system in Lemma \ref{l.6.18} (b) does not factor through
$(\cA_G^\flat)^\natural$ in general. Indeed, if $G$ is a finite discrete
group of order prime to $\ell$, then $\cA_G(\pi_0)$ and
$\cA_G(\pi_0)^\natural$ are both connected groupoids of fundamental group
$G$, while $\cA_G^\flat$ is a simply connected groupoid. If $K\in
\Mod_c(BG,\F_\ell)$, then the projective system in Lemma \ref{l.6.18} (b)
can be identified with the $\F_\ell$-representation of $G$ corresponding to
$K$.
\end{remark}

The proof of Lemma \ref{l.6.18} will be given after Remark \ref{r.homotinv}.
We will exploit the fact that the family of stacks $[X^{A'}/A]$
parameterized by $(A,A',g) \in \mathcal{A}_G(k)^{\natural}$ underlies a
family ``algebraically parameterized'' by $\mathcal{E}_G$. To make sense of
this, the following general framework will be convenient.

\begin{definition}\label{d.dstack}
Let $\mathcal{D}$ be a category enriched in $\AlgSp_{/\kappa}$ (Definition
\ref{s.enrich}). By a \emph{family of Artin $\kappa$-stacks parameterized by
$\mathcal{D}$}, or, for short,  an \emph{Artin $\mathcal{D}$-stack}, we mean
a collection
\[X=(X_A,x_{A,B},\sigma_A, \gamma_{A,B,C})_{A,B,C\in\cD},\]
where $X_A$ is an Artin stack over~$\kappa$, $x_{A,B}\colon X_A\times
\Hom_\cD(A,B)\to X_B$ is a morphism of Artin stacks over~$\kappa$,
$\sigma_A$ and $\gamma_{A,B,C}$ are 2-morphisms:
\[\xymatrix{X_A\ar[r]^-{\id_{X_A}\times i}\ar[rd]_{\id} & X_A\times\Hom_\cD(A,A)\ar[d]^{x_{A,A}}\\
&X_A\ultwocell\omit{<2>\sigma_A}}\]
\[\xymatrix{X_A\times \Hom_\cD(A,B)\times \Hom_\cD(B,C)\ar[d]_{x_{A,B}\times \id_{\Hom_\cD(B,C)}}\ar[r]^-{\id_{X_A}\times c}\ar@{}[rd]_(.4){\gamma_{A,B,C}} &  X_A\times \Hom_\cD(A,C)\ar[d]^{x_{A,C}}\\
X_B\times \Hom_\cD(B,C)\ar[r]^{x_{B,C}} & X_C,\ultwocell\omit{}}\]
satisfying identities of 2-morphisms expressing the unit and associativity
axioms. Here $i\colon \Spec (\kappa)\to \Hom_\cD(A,A)$ is the unit section
and $c\colon \Hom_\cD(A,B)\times \Hom_\cD(B,C)\to \Hom_\cD(A,C)$ is the
composition.

A morphism $f\colon X\to Y$ of Artin $\cD$-stacks is a collection
$((f_A)_{A\in \cD},(\phi_{A,B})_{A,B\in \cD})$, where $f_A\colon X_A\to Y_A$
is a morphism of Artin stacks over $\kappa$ and $\phi_{A,B}$ is a
2-morphism:
\[\xymatrix{X_A\times_{\Spec(\kappa)}\Hom_{\cD}(A,B)\ar[r]^-{x_{A,B}}\ar[d]_{f_A\times \id}\drtwocell\omit{}\ar@{}[rd]^(.6){\phi_{A,B}} & X_B\ar[d]^{f_B}\\
Y_A\times_{\Spec(\kappa)}\Hom_{\cD}(A,B)\ar[r]^-{y_{A,B}} & Y_B}
\]
satisfying certain identities of 2-morphisms with respect to the unit
section $i$ and the composition~$c$.
\end{definition}

\begin{definition}
Let $\Lambda$ be a commutative ring and let $X$ be an Artin $\cD$-stack. We
define a category $D_\cart(X,\Lambda)$ as follows. An object of
$D_\cart(X,\Lambda)$ is a collection $((K_A)_{A\in
\cD},(\alpha_{A,B})_{A,B\in \cD})$, where $K_A\in D_\cart(X_A,\Lambda)$,
$\alpha_{A,B}\colon x_{A,B}^*K_B\to p^*K_A$, $p\colon
X_A\times\Hom_\cD(A,B)\to X_A$ is the projection, such that the following
diagrams commute
\[\xymatrix{i^*x_{A,A}^* K_A\ar[r]^{i^*\alpha_{A,A}}\ar[rd]_{\sigma_A^*}^\simeq &i^*p^* K_A\ar[d]^{\simeq} & x_{A,C}^*
K_C\ar[r]^{\alpha_{A,C}}\ar[d]_{\gamma^*_{A,B,C}}^\simeq
& p^* K_A\\
&K_A &x_{A,B}^*x_{B,C}^* K_{C} \ar[r]^{\alpha_{B,C}} &p^* x_{A,B}^*
K_B. \ar[u]_{\alpha_{A,B}}}
\]
A morphism $K\to L$ in $D_\cart(X,\Lambda)$ is a collection $(K_A\to
L_A)_{A\in \cD}$ of morphisms in $D_\cart(X_A,\Lambda)$ commuting with
$\alpha_{A,B}$.  If $S$ is an Artin stack over $\kappa$, we denote by
$S_\cD$ the constant Artin $\cD$-stack. If $\kappa$ is separably closed and
$\Hom_{\cD}(A,B)$ is noetherian for every $A$ and every $B$ in $\cD$, then
$\Mod_\cart(\Spec(\kappa)_\cD,\Lambda)$ is equivalent to the category of
projective systems of $\Lambda$-modules indexed by $\cD(\pi_0)$. Indeed, in
this case, $\alpha_{A,B}\colon p^*K_B\to p^* K_A$ is a morphism between
constant sheaves on $\Hom_{\cD}(A,B)$, and has to be constant on every
connected component of $\Hom_{\cD}(A,B)$.
\end{definition}

\begin{remark}
If $\mathcal{D}$ is discrete (i.e.\ induced from a usual category) and $X$
is a $\mathcal{D}$-scheme, i.e.\ a functor from $\mathcal{D}$ to the
category of $\kappa$-schemes, the category $D_{\cart}(X,\Lambda)$ consists
of families of objects $K_A \in D(X_A,\Lambda)$ and compatible transition
maps $X_f^*K_B \to K_A$ for $f\colon A \to B$, and should not be confused
with the derived category of sheaves of $\Lambda$-modules on the total
\'etale topos of $X$ over $\mathcal{D}$.
\end{remark}

\begin{construction}
Let $f = ((f_A)_{A \in \mathcal{D}}, (\phi_{A,B})_{A,B\in \mathcal{D}})$ be
a morphism of Artin $\mathcal{D}$-stacks. The functors $f_A^*$ induce a
functor $f^*\colon D_\cart(Y,\Lambda)\to D_\cart(X,\Lambda)$. On the other
hand, for $K\in D_\cart(X,\Lambda)$ we have a diagram
\begin{equation}\label{e.homotinv}
\xymatrix{y^*_{A,B}Rf_{B*}K_B \ar[d] & p^* Rf_{A*}K_A\ar[d]\\
R(f_{A}\times \id)_*x_{A,B}^*K_B\ar[r]^{\alpha_{A,B}} & R(f_{A}\times \id)_*p^* K_A}
\end{equation}
where the left (resp.\ right) vertical arrow is  base change for the square
$\phi_{A,B}$ (resp.\ for the obvious cartesian square).

Assume that $\Lambda$ is annihilated by an integer invertible in $\kappa$,
and that the condition (a) (resp.\ (b)) below holds:
\begin{enumerate}
  \item $\Hom_{\cD}(A,B)$ is smooth over $\kappa$ for all objects $A$, $B$
      in $\cD$;
  \item $f_A$ is quasi-compact and quasi-separated and $K_A\in D^+_\cart$
      for every object $A$ of $\cD$.
\end{enumerate}
Then the right vertical arrow is an isomorphism by smooth base change
(resp.\ generic base change (Remark \ref{r.gbc})) from $\Spec (\kappa)$ to
$\Hom_{\cD}(A,B)$, and thus the diagram \eqref{e.homotinv} defines a map
$y_{A,B}^*Rf_{B*}K_B \to p^*Rf_{A*}K_A$. These maps endow $(Rf_{A*}K_A)$
with a structure of object of $D_{\cart}(Y,\Lambda)$. We thus get a functor
\[Rf_*\colon D_\cart(X,\Lambda)\to D_\cart(Y,\Lambda)\quad (\text{resp.\ $D^+_\cart(X,\Lambda)\to D^+_\cart(Y,\Lambda)$}).
\]
The adjunctions
$\id_{D_\cart(X_A,\Lambda)}\to Rf_{A*}f_A^*$ induce a natural transformation
$\id \to Rf_* f^*$.
\end{construction}

\begin{remark}\label{r.homotinv}
The construction of $Rf_*$ above encodes the homotopy-invariance of \'etale
cohomology \cite[XV Lemme 2.1.3]{SGA7}. More precisely, assume $\kappa$
separably closed. Let $Y, Y'$ be two Artin stacks over $\kappa$, $L\in
D_\cart(Y,\Lambda)$, $L'\in D_\cart(Y',\Lambda)$. A morphism $c\colon
(Y,L)\to (Y',L')$ is a pair $(g,\phi)$, where $g\colon Y\to Y'$, $\phi\colon
g^*L'\to L$. Following \cite[XV Section 2.1]{SGA7}, we say that two
morphisms $c_0, c_1\colon (Y,L)\to (Y',L')$ are \emph{homotopic} if there
exists a connected scheme $T$ of finite type over $\kappa$, two points
$0,1\in T(\kappa)$, a morphism $(Y\times_{\Spec(\kappa)} T, \pr_1^* L)\to
(Y,L')$ inducing $c_0$ and $c_1$ by taking fibers at $0$ and $1$,
respectively. This is equivalent to the existence of an Artin $\cD_T$-stack
$X$ and an object $K\in D_\cart(X,\Lambda)$ such that $X_A=Y$, $X_{A'}=Y'$,
$K_A=L$, $K_{A'}=L'$ and inducing $c_0$ and $c_1$ by taking fibers at $0$
and $1$. Here $\cD_T$ is the $\FTSch{\kappa}$-enriched category with
$\Ob(\cD_T)=\{A,A'\}$,
$\Hom_{\cD_T}(A,A)=\Hom_{\cD_T}(A',A')=\Spec(\kappa)$,
$\Hom_{\cD_T}(A',A)=\emptyset$ and $\Hom_{\cD_T}(A,A')=T$. If $c_0$ and
$c_1$ are homotopic, then $c_0^*= c_1^*\colon H^*(Y',L')\to H^*(Y,L)$. To
prove this, we may assume that $T$ is a smooth curve as in \cite[XV Lemme
2.1.3]{SGA7}. Let $a\colon X\to \Spec(\kappa)_{\cD_T}$ be the projection. By
the above, $R^* a_* K$ is a projective system of graded $\Lambda$-modules
indexed by $\cD_T(\pi_0)$, and $c_0^*=c_1^*$ is the image of the nontrivial
arrow of $\cD_T(\pi_0)$.
\end{remark}

\begin{proof}[Proof of Lemma \ref{l.6.18}.]
By construction, $F$ is essentially surjective. Consider the morphism of
schemes $\phi\colon \Hom_{\cE_G}((A,A',g),(Z,Z',h))\to
\Hom_{\cA_G}(Z',A')=\Trans_G(Z',A')$ given by $(a,b)\mapsto b$. It fits into
the following Cartesian diagram
\[\xymatrix{\Hom_{\cE_G}((A,A',g),(Z,Z',h)) \ar[d]\ar[r]^-{\phi} & \Trans_G(Z',A')\ar[d]\\
\{t\in\Hom(Z',A')\mid t(c_h(Z))\supset c_g(A) \}\ar@{^{(}->}[r] & \Hom(Z',A').}
\]
In particular, $\phi$ is an open and closed immersion and induces an
injection on
\[\Hom_{\cE_G(\pi_0)}((A,A',g),(Z,Z',h))\to
\Hom_{\cA_G(\pi_0)}(Z',A').\]
In other words, the composite functor $p_2\circ
F\colon \cE_G(\pi_0)\to \cA_G(\pi_0)^{\op}$ is faithful, where $p_2\colon
\cA_G(\pi_0)^\natural \to \cA_G(\pi_0)^{\op}$. Therefore, $F$ is faithful.
To show that $F$ is full, let $(\alpha,\beta)\colon F(A,A',g)\to F(Z,Z',h)$
be a morphism in $\cA_G(\pi_0)^\natural$. Choose $b\in \beta(k)\subset
G(k)$. Then we have a Cartesian diagram
\[\xymatrix{\Trans_G(A,Z) \ar[r]^\psi\ar[d] & \Trans_G(A,A')\ar[d]\\
\Hom(A,Z)\ar@{^{(}->}[r] &\Hom(A,A'),}
\]
where $\psi\colon a\mapsto ahb$. In particular, $\psi$ is an open and closed
immersion. The map $\pi_0(\Trans_G(A,Z))\to \pi_0(\Trans_G(A,A'))$ induced
by $\psi$ carries $\alpha$ to $\gamma=\alpha\eta\beta$, where $\gamma\in
\pi_0(\Trans_G(A,A'))$ and $\eta\in \pi_0(\Trans_G(Z,Z'))$ are the connected
components of $g$ and $h$, respectively. Thus there exists $a\in
\alpha(k)\subset G(k)$ such that $g=\psi(a)=ahb$. Then $(a,b)\colon
(A,A',g)\to (Z,Z',h)$ is a morphism in $\cE_G(k)=\cA_G(k)^\natural$, and
induces a morphism $\tau$ in $\cE_G(\pi_0)$ such that
$F(\tau)=(\alpha,\beta)$. Therefore, $F$ is an equivalence of categories.
The second assertion of (a) follows from this and Lemma \ref{l.Afincat}.

Let us prove (b). For $(A,A',g)$ and $(Z,Z',h)$ in $\mathcal{E}_G(k)$,
consider the scheme
\[
T = \Hom_{\mathcal{E}_G}((A,A',g),(Z,Z',h))
\]
and the tautological section $t = (\underline{a},\underline{b}) \in T(T)$.
Then, if $[X^{A'}/A]_T$ (resp.\ $[X^{Z'}/Z]_T$) denotes the product of
$[X^{A'}/A]$ (resp.\ $[X^{Z'}/Z]$) with $T$ over $\Spec(k)$, $t$ defines a
morphism of stacks $[\underline{b}^{-1}/c_{\underline{a}}] \colon
[X^{A'}/A]_T \to [X^{Z'}/Z]_T$ over $T$, whose fiber at $(a,b)$ is
$[b^{-1}/c_a]$. These morphisms are compatible with composition of morphisms
up to 2-morphisms, and define a structure of $\mathcal{E}_G$-stack
(Definition \ref{d.dstack}) on the family of stacks $[X^{A'}/A]$ for
$(A,A',g) \in \mathcal{E}_G(k)$. Moreover, we have a diagram over $T$
\begin{equation}\label{e.6.18.3}
\xymatrix{{} & [X^{A'}/A]_T \ar[d]^{[\underline{b}^{-1}/c_{\underline{a}}]} \ar[dl]_{[1/c_g]} \ar[r]^{\pi} & X^{A'}_T \ar[d]^{\underline{b}^{-1}}\\
[X/G]_T & [X^{Z'}/Z]_T \ar[l]^{[1/c_h]} \ar[r]^{\pi} & X^{Z'}_T,}
\end{equation}
where the 2-morphism of the triangle is induced by $\underline{b}$. The
fiber of \eqref{e.6.18.3} at $(a,b)$ is \eqref{e.6.15.3}.  Therefore we get
morphisms of Artin $\mathcal{E}_G$-stacks
\[
[X/G]_{\mathcal{E}_G} \from ([X^{A'}/A])_{(A,A',g)} \xrightarrow{\pi} (X^{A'})_{(A,A',g)} .
\]
Thus the system $H^0(X^{A'},R^q\pi_*[1/c_g]^* K)$ indexed by $(A,A',g)\in
\cA_G(k)^\natural$ can be extended to an object of
$\Mod_\cart(\Spec(k)_{\cE_G},\F_\ell)$, which amounts to a system indexed by
$\cE_G(\pi_0)$. More concretely, the morphism $(a,b)^*$ \eqref{e.6.15.4} is
the stalk at $(a,b)$ of a morphism of constant sheaves on $T$
\begin{equation}\label{e.6.18.4}
(\underline{a},\underline{b})^* \colon H^0(X^{Z'},R^q\pi_*[1/c_h]^*K)_T \to
H^0(X^{A'},R^q\pi_*[1/c_g]^*K)_T,
\end{equation}
defined by $(\underline{a},\underline{b})$ via \eqref{e.6.18.3}. Therefore
it depends only on the connected component of $(a,b)$ in $T$.
\end{proof}

We need the following lemma for the proof of Theorem \ref{t.main2} (a).

\begin{lemma}\label{l.fintriv}
Let $Y$ be an algebraic space over $k$, and let $A$ be a finite discrete
group. Let $L \in D^b_c([Y/A],\F_{\ell})$, where $A$ acts trivially on~$Y$.
Let $\pi \colon [Y/A] = BA \times Y \to Y$ be the second projection.
Consider the structure of $H^*(BA,\F_\ell)$-module on $R^*\pi_*L$ given by
Constructions \ref{s.tensor} and \ref{s.rtopos}, as $R^*\pi_*\F_\ell$ is a
constant sheaf of value $H^*(BA,\F_\ell)$. Then $R^*\pi_*L$ is a sheaf of
constructible $H^*(BA,\F_\ell)$-modules.
\end{lemma}

\begin{proof}
We may assume $L$ concentrated in degree zero. Suppose first that $L$ is
locally constant. Then $R^*\pi_*L$ is a locally constant, constructible
sheaf of $H^*(BA,\F_\ell)$-modules. Indeed, by definition there is an
\'etale covering $(U_\alpha)$ of $Y$ such that $L\res [U_\alpha/A]$
(considered as a sheaf of $\F_\ell[A]$-modules on $U_\alpha$) is a constant
$\F_{\ell}[A]$-module of finite dimension over $\F_{\ell}$ of value
$L_{\alpha}$. Then $R^*\pi_*L \res U_\alpha$ is a constant
$H^*(BA,\F_\ell)$-module of value $H^*(BA,L_\alpha)$. By Theorem
\ref{t.finite}, $H^*(BA,L_{\alpha})$ is a finite $H^*(BA,\F_\ell)$-module,
so the lemma is proved in this case. In general, we may assume $Y$ to be an
affine scheme. Take a finite stratification $Y = \bigcup Y_{\alpha}$ into
disjoint locally closed constructible subsets such that $L \res Y_{\alpha}$
is locally constant, or equivalently, that $L\res [Y_\alpha/A]$ is locally
constant. Then, if $\pi_{\alpha} = \pi \res [Y_{\alpha}/A] \to Y_{\alpha}$,
$(R^*\pi_*L )\res Y_{\alpha} \simeq R\pi_{\alpha *}(L \res Y_{\alpha})$ by
the finiteness of $A$, and we conclude by the preceding case.
\end{proof}

\begin{proof}[Proof of Theorem \ref{t.main2} (a).]
By Lemma \ref{l.6.18} (b) we can rewrite $R^q_G(X,K)$ in the form
\[
R^q_G(X,K) \coloneqq \varprojlim_{(A,A',g) \in \mathcal{E}_G(\pi_0)} H^0(X^{A'},R^q\pi_*[1/c_g]^*K).
\]
As $\mathcal{E}_G(\pi_0)$ is essentially finite (Lemma \ref{l.6.18} (a)) and
$R^q\pi_*[1/c_g]^*K$ is constructible, the first assertion follows. Let us
now prove the second assertion. As $\mathcal{E}_G(\pi_0)$ is equivalent to a
finite category, it is enough to show that, for all $(A,A',g)$,
$H^0(X^{A'},R^*\pi_* [1/c_g]^* K)$ is a finite $H^*(BG,\F_\ell)$-module. As
$A$ acts trivially on $X^{A'}$, $R^*\pi_*[1/c_g]^*K$ is a constructible
sheaf of $H^*(BA,\F_\ell)$-modules by Lemma \ref{l.fintriv}. Therefore
$H^0(X^{A'},R^*\pi_*[1/c_g]^*K)$ is a finite $H^*(BA,\F_\ell)$-module, thus,
by Corollary \ref{c.finite}, a finite $H^*(BG,\F_\ell)$-module.
\end{proof}

\section{Points of Artin stacks}\label{s.5}

In this section we discuss two kinds of points of Artin stacks which will be
of use to us:
\begin{enumerate}
\item \emph{geometric points}, which generalize the usual geometric points
    of schemes,

\item \emph{$\ell$-elementary points}, which depend on a prime number
    $\ell$, and are adapted to the study of the maps $a(G,X) $
    \eqref{e.aGX} and $a_G(X,K)$ (Theorem \ref{t.main2}).
\end{enumerate}

The statement of the main structure theorem on Artin stacks (Theorem
\ref{p.str}) requires only the notion (b). The notion (a) is a technical
tool used in the proof.

\begin{definition}\label{s.point}
Let $\mathcal{X}$ be a Deligne-Mumford stack. By a \emph{geometric point} of
$\mathcal{X}$ we mean a morphism $x \to \mathcal{X}$, where $x$ is the
spectrum of a separably closed field. The geometric points of $\mathcal{X}$
form a category
\[
P_{\mathcal{X}},
\]
where a morphism from $x \to \mathcal{X}$ to $y \to \mathcal{X}$ is defined
as an $\mathcal{X}$-morphism $\mathcal{X}_{(x)} \to \mathcal{X}_{(y)}$ of
the corresponding strict henselizations \cite[Remarque 6.2.1]{LMB}. The
category $P_{\mathcal{X}}$ is essentially $\mathcal{U}$-small.  One shows as
in \cite[VIII Th\'eor\`eme 7.9]{SGA4} that the functor $(x \to \mathcal{X})
\mapsto (\mathcal{F} \mapsto \mathcal{F}_{x})$ from $P_{\mathcal{X}}$ to the
category of points of the \'etale topos $\mathcal{X}_{\et}$ is an
equivalence of categories.

When $\mathcal{X}$ is a scheme, $P_{\mathcal{X}}$ is the usual category of
geometric points of $\mathcal{X}$. If $\mathcal{X} = \Spec  k$, $k$ a field,
$P_{\mathcal{X}}$ is a connected groupoid whose fundamental group is
isomorphic to the Galois group of $k$.

As $P_\cX$ is an essentially $\cU$-small category, we have a morphism of
topoi
\begin{equation}\label{e.7.1.1}
p \colon \widehat{P_{\cX}} \to \cX_{\et},
\end{equation}
where $\widehat{P_{\cX}}$ denotes the topos of presheaves on $P_\cX$. For a
sheaf $\mathcal{F}$ on $\mathcal{X}$, $p^*\mathcal{F}$ is the presheaf $(x
\to \mathcal{X}) \mapsto \mathcal{F}_x$ on $P_{\mathcal{X}}$, and $p_*$
applied to a presheaf $(K_x)_{x \in P_{\mathcal{X}}}$ is the sheaf whose set
of sections on $U$ is $\varprojlim_{x \in P_{U}} K_x$. In particular we have
an adjunction map
\begin{equation}\label{e.adj}
b_{\mathcal{X},\mathcal{F}} \colon \mathcal{F} \to p_*p^*\mathcal{F},
\end{equation}
which is a monomorphism, as $\cX_{\et}$ has enough points and
$p^*b_{\cX,\cF}$ is a split monomorphism (this fact holds of course more
generally for any topos $\cX$ with an essentially small conservative family
of points $P_\cX$, cf.\ \cite[IV 6.7]{SGA4}).
\end{definition}

\begin{prop}\label{p.sp}
Let $\cX$ be a locally noetherian Deligne-Mumford stack, $\Lambda$ a
noetherian commutative ring, $\cF$ a constructible sheaf of
$\Lambda$-modules on $\cX$. Then the adjunction map $b_{\cX,\cF}\colon
\cF\to p_*p^* \cF$ \eqref{e.adj} is an isomorphism. In particular, the
homomorphism
  \begin{equation}\label{e.sp}
  \phi\colon \cF(\cX)\to \varprojlim_{x\in P_\cX} \cF_x
  \end{equation}
  is an isomorphism.
\end{prop}

\begin{proof}
If $f\colon\cY\to\cX$ is a morphism of Deligne-Mumford stacks, the square of
topoi
  \[\xymatrix{\widehat{P_\cY} \ar[r]^{p_\cY}\ar[d]_{P_f} & \cY_{\et}\ar[d]^{f_\et}\\
  \widehat{P_\cX} \ar[r]^{p_{\cX}} & \cX_{\et}}\]
  commutes and induces base change morphisms
  \begin{equation}\label{e.spbc1}
    p_\cX^* f_* \to P_{f*} p_\cY^*
  \end{equation}
  and
  \begin{equation}\label{e.spbc2}
    f^* p_{\cX*} \to p_{\cY*} P_f^*
  \end{equation}
and commutative diagrams
\begin{equation}\label{e.spbc3}
  \xymatrix{f_* \cG \ar[r]^{f_* b_{\cY,\cG}}\ar[d]_{b_{\cX,f_*\cG}} & f_* p_{\cY*} p_{\cY}^* \cG\ar[d]^\simeq &f^*\cF\ar[r]^{b_{\cY,f^*\cF}}\ar[d]_{f^*b_{\cX,\cF}} & p_{\cY *} p_\cY^*f^*\cF\ar[d]^\simeq \\
  p_{\cX*}p_\cX^*f_*\cG \ar[r]^{\eqref{e.spbc1}} & p_{\cX*} P_{f*} p_\cY^* \cG &f^* p_{\cX*} p_\cX^* \cF\ar[r]^{\eqref{e.spbc2}} & p_{\cY*} P_f^* p_{\cX}^* \cF .}
\end{equation}
If $f$ is a closed immersion, \eqref{e.spbc1} is an isomorphism. If $f$ is
\'etale, \eqref{e.spbc2} is an isomorphism.

Let $i\colon \cZ\to \cX$ be a closed immersion, and let $j\colon \cU\to \cX$
be the complementary open immersion. Then the following diagram with exact
rows commutes (where we write $p$ for $p_{\cX}$):
  \[\xymatrix{0\ar[r]& j_!j^*\cF \ar[r]\ar[d]_{b_{\cX,j_!j^*\cF}} & \cF\ar[r]\ar[d]^{b_{\cX,\cF}} & i_* i^* \cF\ar[r]\ar[d]^{b_{\cX,i_*i^*\cF}} & 0\\
  0\ar[r] & p_*p^* j_!j^*\cF\ar[r] & p_*p^*\cF\ar[r] & p_*p^*i_*i^*\cF.}
  \]
Thus, to show that $b_{\cX,\cF}$ is an isomorphism, it suffices to show that
both $b_{\cX,i_*i^*\cF}$ and $b_{\cX,j_!j^*\cF}$ are isomorphisms. By the
square on the left of \eqref{e.spbc3} applied to $\cG=i^*\cF$,
$b_{\cX,i_*i^*\cF}$ is an isomorphism if $b_{\cZ,i^*\cF}$ is an isomorphism.
On the other hand, the following diagram commutes:
\begin{equation}\label{e.sp23}
  \xymatrix{j^*\cF\ar[d]_{b_{\cU,j^*\cF}} & j^*j_!j^*\cF \ar[l]_\sim\ar[r]^{j^*b_{\cX,j_!j^*\cF}}\ar[d]^{b_{\cU,j^*j_!j^*\cF}} & j^*p_*p^*j_!j^*\cF\ar[d]^{\eqref{e.spbc2}}_\simeq \\
  p_{\cU*}p_\cU^*j^*\cF & p_{\cU*}p_\cU^*j^*j_!j^*\cF \ar[l]_\sim\ar[r]^\sim & p_{\cU*}P_j^*p^*j_!j^*\cF.}
\end{equation}
  We now prove that
  \begin{equation}\label{e.zero}
  i^*(p_*p^*j_!j^*\cF)=0.
  \end{equation}
By the commutativity of \eqref{e.sp23}, this will imply that
$b_{\cX,j_!j^*\cF}$ is an isomorphism if $b_{\cU,j^*\cF}$ is an isomorphism.
For any geometric point $z\to \cZ$,
  \begin{equation}\label{e.splim}
  (p_*p^*j_!j^*\cF)_z \simeq \varinjlim_{U\in N_\cX(z)^\op} \varprojlim_{u\in P_U}(j_!j^*\cF)_u,
  \end{equation}
where $N_\cX(z)$ is the category of \'etale neighborhoods of $z$ in $\cX$
that are quasi-compact and quasi-separated schemes. Let $U$ be any such
neighborhood. Take a finite stratification $(U_\alpha)_{\alpha\in A}$ of $U$
by connected locally closed constructible subschemes such that the
restrictions $\cF\res U_\alpha$ are locally constant. Let
$P_{U,(U_\alpha)_{\alpha\in A}}$ be the category obtained from $P_U$ by
inverting all arrows in the full subcategories $P_{U_\alpha}$. Geometric
points of the same stratum are isomorphic in $P_{U,(U_\alpha)_{\alpha\in
A}}$. Let $B\subset A$ be the subset of indices $\alpha$ such that there
exists a morphism from a geometric point of $U_\alpha$ to $z$ in
$P_{U,(U_\alpha)_{\alpha\in A}}$. Let $V=\bigcup_{\alpha\in B} U_\alpha$.
Since the geometric points of $V$ are closed under generization in $U$, $V$
is an open subset of $U$. Since specialization maps on the same stratum are
isomorphisms, the projective system $((j_!j^*\cF)_v)_{v\in P_V}$ factors
uniquely through a projective system $((j_!j^*\cF)_x)_{x\in P_{V,
(U_\alpha)_{\alpha\in B}}}$ (where on each stratum $U_\alpha$, $\alpha\in B$
all specialization maps are isomorphisms) and
  \[\varprojlim_{v\in P_V} (j_!j^*\cF)_v \simeq \varprojlim_{x\in P_{V, (U_\alpha)_{\alpha\in B}}} (j_!j^*\cF)_x
  \]
by Lemma \ref{l.loclim} below. Note that $P_V$ contains $z$ and that for any
object $x$ of $P_{V,(U_\alpha)_{\alpha\in B}}$ there exists a morphism from
$x$ to $z$. Therefore, as $(j_!j^*\cF)_z=0$, this limit is zero. This
implies that the full subcategory of $N_\cX(z)^\op$ consisting of the
neighborhoods $U$ such that $\varprojlim_{v\in P_U}(j_!j^*\cF)_v=0$ is
cofinal. It follows that the limit \eqref{e.splim} is zero and hence
\eqref{e.zero} holds, as claimed. To sum up, we have shown that
$b_{\cX,\cF}$ is an isomorphism if both $b_{\cZ,i^*\cF}$ and
$b_{\cU,j^*\cF}$ are isomorphisms.

By induction, we may therefore assume $\cF$ locally constant. Using the
square on the right of \eqref{e.spbc3}, we may assume $\cF$ constant. In
this case it suffices to show that \eqref{e.sp} is an isomorphism. We may
further assume that $\cX$ is connected and noetherian. Then $P_\cX$ is a
connected category and the assertion is trivial.
\end{proof}

\begin{lemma}\label{l.loclim}
Let $\cC$ be a category and let $S$ be a set of morphisms in $\cC$. If we
denote by $F\colon \cC\to S^{-1}\cC$ the localization functor, then $F$ and
$F^\op$ are cofinal (Definition \ref{s.cofinal}).
\end{lemma}

\begin{proof}
It suffices to show that $F$ is cofinal. Let $X$ be an object of
$S^{-1}\cC$, let $Y$ be an object of $\cC$ and let $f\colon X\to FY$ be a
morphism in $S^{-1}\cC$. Then $f=t_ns_n^{-1}\dots t_1s_1^{-1}$ with $t_i$ in
$\cC$ and $s_i\in S$. Using $t_i$ and $s_i$, $f$ can be connected to
$\id_X\colon X\to FX$ in $(X\downarrow S^{-1}\cC)$.
\end{proof}

\begin{remark}
If, in Proposition \ref{p.sp}, the sheaf $\cF$ is not assumed constructible,
then the monomorphism $\phi$ is not an isomorphism in general, as shown by
the following example. Let $X$ be a scheme of dimension $\ge 1$ of finite
type over a separably closed field $k$  and let $\mathcal{F}=\bigoplus_{x
\in |X|} i_{x*}\Lambda$, where $|X|$ is the set of closed points of $X$ and
let $i_x \colon \{x\} \to X$ be the inclusion. Then $\Gamma(X,\mathcal{F})
\simeq \Lambda^{(|X|)}$ (by commutation of $\Gamma(X,-)$ with filtered
inductive limits). On the other hand, for $x \in P_X$, $\mathcal{F}_x =
\Lambda$ if the image of $x$ is a closed point, and $\mathcal{F}_x = 0$
otherwise, hence $\varprojlim_{x \in P_X} \Lambda \simeq \Lambda^{|X|}$. The
monomorphism $\varphi$ in Proposition \ref{p.sp} can be identified with the
inclusion $\Lambda^{(|X|)} \subset \Lambda^{|X|}$, which is not an
isomorphism, as $|X|$ is infinite.
\end{remark}

\begin{remark}
In the situation of Proposition \ref{p.sp}, the morphism
  \[R\Gamma(\cX,\cF)\to R\varprojlim_{x\in P_\cX} \cF_x\]
is not an isomorphism in general. In fact, if $\cX=\Spec (k)$, then the left
hand side computes the continuous cohomology of the Galois group $G$ of $k$
while the right hand side computes the cohomology of $G$ as a discrete
group.
\end{remark}

\begin{definition}\label{d.7.4}
Let $\mathcal{X}$ be an Artin stack. By a \emph{geometric point} of
$\mathcal{X}$ we mean a morphism $a \colon S \to \mathcal{X}$, where $S$ is
a strictly local scheme. If $a \colon S \to \mathcal{X}$ and $b \colon S \to
\mathcal{X}$ are geometric points of $\cX$, a morphism $(a \colon S \to
\mathcal{X}) \to (b \colon T \to \mathcal{X})$ is a morphism $u \colon S \to
T$ together with a 2-morphism
\begin{equation}\label{e.triST}
  \xymatrix{S\ar[r]\ar[dr] & T\ar[d]\\& \cX.\ultwocell\omit{<2>}}
\end{equation}
We thus get a full subcategory $\mathcal{P}'_{\mathcal{X}}$ of
$\AlgSp_{/\cX}$ (Notation \ref{s.smtop}). We define the \emph{category of
geometric points} of $\mathcal{X}$ as the category
\[
\mathcal{P}_{\mathcal{X}} = M_{\mathcal{X}}^{-1}\mathcal{P}'_{\mathcal{X}},
\]
localization of $\mathcal{P}'_{\mathcal{X}}$ by the set $M_{\mathcal{X}}$ of
morphisms $(a \to b)$ in $\mathcal{P}'_{\mathcal{X}}$ sending the closed
point of $S$ to the closed point of $T$.
\end{definition}

Although $\cP'_\cX$ is a $\cU$-category and not essentially small in
general, we will show in Proposition \ref{p.small} that $\cP_\cX$ is
essentially small. The next proposition shows that the definition above is
consistent with Definition \ref{s.point}.

\begin{prop}\label{l.DMpt}
For any Deligne-Mumford stack $\cX$, the functor $P_\cX\to \cP'_\cX$
sending every geometric point $x\to \cX$ to the strict henselization
$\cX_{(x)}\to \cX$ induces an equivalence of categories $\iota\colon
P_\cX\to \cP_\cX$.
\end{prop}

\begin{proof}
Consider the functor $F'\colon \cP'_\cX\to P_\cX$ sending $S\to \cX$ to its
closed point $s\to \cX$. For any morphism in $\cP'_\cX$ as in
\eqref{e.triST}, its image under $F'$ is the induced morphism $\cX_{(s)}\to
\cX_{(t)}$, where $s$ and $t$ are the closed points of $S$ and $T$,
respectively. The functor $F\colon \cP_\cX\to P_\cX$ induced by $F'$ gives a
quasi-inverse to $\iota\colon P_\cX\to \cP_\cX$. In fact
$F\iota=\id_{P_\cX}$ and we have a natural isomorphism $\id_{\cP_\cX}\to
\iota F$ given by the morphism $S\to \cX_{(s)}$ in $M_\cX$ for $S\to \cX$ in
$\cP_\cX$ of closed point $s$.
\end{proof}

\begin{remark}\label{s.cP}
The reason why we do not consider the category of points $\Point(\cX_\sm)$
of the smooth topos $\cX_{\sm}$ is that already in the case $\cX$ is an
algebraic space, the functor $\Point(\cX_\sm)\to \Point(\cX_\et)$ induced by
the morphism of topoi $\epsilon\colon \cX_\sm\to \cX_\et$ is not an
equivalence. For example, if $U\to \cX$ is a smooth morphism and $y$ is a
geometric point of $U$ lying above a geometric point $x$ of $\cX$ such that
the image of $y$ in the fiber $U\times_\cX x$ is not a closed point, then
the points $\tilde x\colon \cF\mapsto (\cF_\cX)_x$ and $\tilde y\colon
\cF\mapsto (\cF_U)_y$ of $\cX_\sm$ are not equivalent, but have equivalent
images in $\Point(\cX_\et)$. Indeed, if we denote by $\epsilon^!\colon
\cX_\et\to \cX_\sm$ the right adjoint of $\epsilon_*$, then the stalk of
$\epsilon^! \cG$ is $\cG_x$ at $\tilde x$, but is $e$ at $\tilde y$.
\end{remark}

\begin{prop}\label{p.small}
Let $\cX$ be an Artin stack, and let $\tilde\cP'_\cX$ be the full
subcategory of $\cP'_\cX$ consisting of morphisms $S\to \cX$, such that
$S\to \cX$ is the strict henselization of some smooth atlas $X\to \cX$ at
some geometric point of $X$. Let $\tilde M_\cX=M_\cX\cap \Ar(\tilde
\cP'_\cX)$. Then the inclusion $\tilde \cP'_\cX\subset \cP'_\cX$ induces an
equivalence of categories $\tilde M_\cX^{-1} \tilde \cP'_\cX\to \cP_\cX$.
\end{prop}

Note that $\tilde \cP'_\cX$ and hence $\tilde \cP_\cX$ are essentially
small. Thus Proposition \ref{p.small} shows that $\cP_\cX$  is essentially
small,

\begin{proof}
We write $\tilde\cP_\cX=\tilde M_\cX^{-1}\tilde \cP'_\cX$. For $x\colon S\to
\cX$ in $\cP'_\cX$, let $A_x$ be the full subcategory of
$(\AlgSp_{/\cX})_{x/}$ consisting of diagrams
  \begin{equation}\label{e.small}
  \xymatrix{S\ar[r]\ar[rd]_x & X\ar[d]^p\\
  &\cX \ultwocell\omit{<2>}}
  \end{equation}
such that $p$ is a smooth atlas. Then $A_x$ is nonempty since every smooth
surjection to $S$ admits a section \cite[Corollaire 17.16.3 (ii)]{EGAIV}.
Moreover, $A_x$ admits finite nonempty products. Consider the functor
$F_x\colon A_x\to \tilde \cP'_\cX$ sending \eqref{e.small} to the strict
localization $X_{(s)}\to \cX$ at the closed point $s$ of $S$. For any pair
of morphisms $(f,g)\colon X\rightrightarrows Y$ with the same source and
target in $A_x$, $F_x(f)$ and $F_x(g)$ have the same image in $\tilde
\cP_\cX$. Indeed, $f\res S=g\res S$ implies $F_x(f)t=F_x(g)t$, where $t\in
\tilde M_\cX$ is the inclusion of the closed point of $X_{(s)}$. Thus there
exists a unique functor $G_x$ making the following diagram commutative
  \[\xymatrix{A_x\ar[r]^{F_x}\ar[d] & \tilde \cP'_\cX\ar[d]\\
  \lvert A_x \rvert \ar[r]^{G_x} & \tilde \cP_\cX}\]
where $\lvert A_x \rvert$ is the simply connected groupoid having the same
objects as $A_x$. This construction is functorial in $x$, in the sense that
for $x\to y$ in $\cP'_\cX$, we have a natural transformation
  \[\xymatrix{\lvert A_y \rvert\ar[d]\ar[rd]^{G_y}\\
  \lvert A_x\rvert\ar[r]_{G_x} &  \tilde \cP_\cX.\ultwocell\omit{<-2>}}\]
Choosing an object $X$ in  $A_x$ for every $x$, we obtain a functor
$\cP'_\cX \to \tilde \cP_\cX$ sending $x$ to $X_{(s)}$. This functor factors
through $\cP_\cX\to \tilde\cP_\cX$ and defines a quasi-inverse of $\tilde
\cP_\cX\to \cP_\cX$.
\end{proof}

\begin{remark}\label{s.homeo}
For any morphism $f\colon \cX\to \cY$ of Artin stacks, composition with $f$
defines a functor $\cP'_f\colon \cP'_\cX\to \cP'_\cY$, which induces
$\cP_f\colon \cP_\cX \to \cP_\cY$.
\begin{enumerate}
\item If $f$ is a schematic universal homeomorphism, then $\cP_f$ is an
    equivalence of categories. In fact, for any object $T\to \cY$ of
    $\cP'_\cY$, the base change $S=T\times_\cY \cX\to T$ is a schematic
    universal homeomorphism, so that $S$ is a strictly local scheme by
    \cite[Proposition 18.8.18 (i)]{EGAIV}. The functor $\cP'_{\cY}\to
    \cP'_{\cX}$ carrying $T\to \cY$ to $T\times_\cY \cX \to \cX$ carries
    $M_\cY$ to $M_\cX$ and induces a quasi-inverse of $\cP_f$.

\item For morphisms $\cX\to \cY$ and $\cZ\to\cY $ of Artin stacks, the
    functor $\cP'_{\cX\times_\cY\cZ}\to \cP'_\cX\times_{\cP'_\cY}\cP'_\cZ$
    is an equivalence of categories.
\end{enumerate}
\end{remark}

\begin{example}\label{e.PBG}
Let $k$ be a separably closed field, and let $G$ be an algebraic group over
$k$. Then $\cP_{BG}$ is a connected groupoid whose fundamental group is
isomorphic to $\pi_0(G)$.

To prove this, by Remark \ref{s.homeo} (a), we may assume $k$ algebraically
closed and $G$ smooth. Then, for every object $S\to BG$ of $\cP'_{BG}$, the
corresponding $G_{S}$-torsor is trivial and we fix a trivialization. For any
strictly local scheme $S$ over $\Spec(k)$, we denote by $p_S\colon S\to BG$
the object of $\cP_{BG}$ corresponding to the trivial $G_S$-torsor and by
$a_S\colon S\to \Spec(k)$ the projection. By the definition of $BG$
\eqref{e.1.5.2}, morphisms $p_S\to p_T$ in $\cP'_{BG}$ correspond
bijectively to pairs $(f,r)$, where $f\colon S\to T$ is a morphism of
schemes and $r\in G(S)$. We denote the morphism corresponding to $(f,r)$ by
$\theta(f,r)$. If $s$ is the closed point of $S$, $r(s) \in G(s)$ belongs to
the inverse image of a unique connected component of $G$, denoted $[r]$. Let
$\Pi$ be the groupoid with one object and fundamental group $\pi_0(G)$. The
above construction defines a full functor $\cP'_{BG}\to \Pi$ sending
$\theta(f,r)$ to $[r]$, which induces a functor still denoted by $F\colon
\cP_{BG}\to \Pi$. Since $\theta(a_S,r)\colon p_S\to p_{\Spec(k)}$ is in
$M_{BG}$ and $\theta(f,r)\theta(a_T,1)=\theta(a_S,r)$, $\theta(f,r)$ is an
isomorphism in $\cP_{BG}$. Thus $\cP_{BG}$ is a connected groupoid. To show
that $F$ is an equivalence of categories, it suffices to check that for all
$r\in G^0(S)$, $\theta(a_S,r)\equiv\theta(a_S,1)$. Here $\equiv$ stands for
equality in $\cP_{BG}$. For this, we may assume that $S$ is a point, say
$S=\Spec(k')$. We regard $r\colon \Spec(k')\to G$ as a geometric point
of~$G$. Since $G^0$ is irreducible, $X=G_{(1)}\times_G G_{(r)}$ is nonempty.
Let $x$ be a geometric point of $X$, and let $t\in G(G)$ be the tautological
section. Then
  \[\theta(a_{G_{(1)}},t)\theta(s_1,1)=\theta(a_{\Spec(k)},1)=\theta(a_{G_{(1)}},1) \theta(s_1,1),\]
where $s_1\colon \Spec(k)\to G_{(1)}$ is the closed point. It follows that
$\theta(a_{G_{(1)}},t)\equiv\theta(a_{G_{(1)}},1)$,
$\theta(a_x,t)\equiv\theta(a_x,1)$, and hence
$\theta(a_{G_{(r)}},t)\equiv\theta(a_{G_{(r)}},1)$. Therefore, if $s_r\colon
\Spec(k')\to G_{(r)}$ denotes the closed point, we have
  \[\theta(a_{\Spec(k')},r)=\theta(a_{G_{(r)}},t)\theta(s_r,1)\equiv\theta(a_{G_{(r)}},1)\theta(s_r,1)=\theta(a_{\Spec(k')},1).\]
\end{example}

\begin{construction}\label{c.p}
Let $\mathcal{X}$ be a locally noetherian Artin stack. If $\cF$ is a
cartesian sheaf on $\mathcal{X}$, then the presheaf
\[
\fp' \cF\colon (a \colon S \to \mathcal{X}) \mapsto \Gamma(S,a^*\cF) \simeq  \cF_s
\]
(where $s$ is the closed point of $S$) on $\cP'_\cX$ defines a presheaf on
$\mathcal{P}_{\mathcal{X}}$, which will denote by $\fp \cF$. We thus get an
exact functor
\begin{equation}\label{e.p*}
\fp\colon \Sh_\cart(\cX)\to \hat\cP_\cX.
\end{equation}
If $\cX$ is a Deligne-Mumford stack, then $p^*\simeq \iota^*\fp$, where
$p\colon \hat P_\cX\to\cX_\et$ is the projection \eqref{e.7.1.1} and
$\iota\colon P_\cX\to \cP_\cX$ the equivalence of Proposition \ref{l.DMpt}.
\end{construction}

The following result generalizes Proposition \ref{p.sp}.

\begin{prop}\label{p.spA}
Let $\cX$ be an Artin stack, let $\Lambda$ be a noetherian commutative ring,
and let $\cF$ be a constructible sheaf of $\Lambda$-modules on $\cX$. Then
the map
  \begin{equation}\label{e.spA}
  \Gamma(\cX,\cF) \to \varprojlim_{x\in \cP_\cX} \cF_x
  \end{equation}
defined by the restriction maps $\Gamma(\cX,\cF)\to (\fp \cF)(x)=\cF_x$ is
an isomorphism.
\end{prop}

The proof will be given after a couple of lemmas.

\begin{lemma}\label{l.limit}
Let $F\colon \cC\to \cD$ be a functor between small categories.
Assume that for any morphism $f\colon X\to Y$ in $\cD$, there exists
a morphism $a\colon A\to B$ in $\cC$ and a commutative square in
$\cD$ of the following form:
  \[\xymatrix{X\ar[r]^f\ar[d]_{\simeq} & Y\ar[d]^{\simeq}\\
  F(A)\ar[r]^{F(a)} & F(B).}\]
Then $F$ is of descent for presheaves. More precisely, for any presheaf
$\cF$ on $\cD$, with the notation of \cite[IV 4.6]{SGA4}, the sequence
  \[\cF\to F_*F^* \cF \rightrightarrows F_{2*}F_2^* \cF\]
is exact, where $\cC\times_\cD \cC$ is the 2-fiber product,
$F_2\colon \cC\times_\cD \cC \to \cD$ is the projection, and the
double arrow is induced by the two projections from $\cC\times_\cD
\cC$ to $\cC$. In particular, the sequence
  \begin{equation}\label{e.limit}
  \Gamma(\hat{\cD},\cF)\to \Gamma(\hat{\cC}, F^* \cF)\rightrightarrows \Gamma(\widehat{\cC\times_\cD \cC}, F_2^* \cF)
  \end{equation}
is exact.
\end{lemma}

\begin{proof}
For any $X$ in $\cD$, $\cF(X)\to (F_*F^*\cF)(X)\rightrightarrows
(F_{2*}F_2^* \cF)(X)$ is \eqref{e.limit} applied to the functor $F'\colon
\cC_{/X}\to \cD_{/X}$ induced by $F$ and the presheaf $\cF\res(\cD_{/X})$.
Since $F'$ also satisfies the assumption of the lemma, it suffices to prove
that \eqref{e.limit} is exact. By definition,
$\Gamma(\hat{\mathcal{C}},F^*\cF)$ consists of families $s = (s_X) \in
\varprojlim_{X \in \mathcal{C}}\mathcal{F}(F(X))$. Similarly,
$\Gamma(\widehat{\mathcal{C}\times_{\mathcal{D}}
\mathcal{C}},F_2^*\mathcal{F}) = \varprojlim_{(Y,Z,\alpha) \in
\mathcal{C}\times_{\mathcal{D}} \mathcal{C}} \mathcal{F}(F_2(Y,Z,\alpha))$.
Let $E$ be the equalizer of the double arrow in \eqref{e.limit}. We
construct $\epsilon\colon E\to \Gamma(\hat\cD,\cF)$ as follows. Let $s\in
E$. For any object $X$ of $\cD$, put $\epsilon(s)_X=\cF(e)(s_A)\in \cF(X)$,
for a choice of $e \colon X \simto F(A)$. This does not depend on the choice
of $e$, because if $e'\colon X\simto F(A')$, then $(A,A',e'e^{-1})$ defines
an object of $\cC\times_\cD \cC$, and $s\in E$ implies
$\cF(e)(s_A)=\cF(e')(s_{A'})$. For any morphism $f\colon X\to Y$ in $\cD$,
the hypothesis implies that $\cF(f)(\epsilon(s)_Y)=\epsilon(s)_X$. This
finishes the construction of $\epsilon$ It is straightforward to check that
$\epsilon$ is an inverse of $\Gamma(\hat\cD,\cF)\to E$.
\end{proof}

\begin{lemma}\label{l.P}
Let $f\colon \cX\to \cY$ be a smooth surjective morphism of Artin stacks. If
$\cU$ is a universe containing $\cP'_\cX$ and $\cP'_\cY$, then the functor
$\cP'_f\colon \cP'_\cX \to \cP'_\cY$ satisfies the condition of Lemma
\ref{l.limit} for $\cU$.
\end{lemma}

\begin{proof}
Let $(h,\alpha)\colon (S,u)\to (T,v)$ be a morphism in $\cP'_\cY$. Since
$\cX\times_\cY T$ is an Artin stack smooth over $T$, it admits a section,
giving rise to the following 2-commutative diagram
  \[\xymatrix{S\ar[r]^h\ar@/_/[rrd]_u &T\ar[r]^g\ar[rd]_v & \cX\ar[d]^f\\
  &&\cY.\ulltwocell\omit{<1.5>\alpha}\ultwocell\omit{<2>\beta}}\]
Then the following diagram commutes
  \[\xymatrix{(S,u)\ar[r]^{(h,\alpha)}\ar[d] & (T,v)\ar[d]\\
  \cP'_f((S,gh))\ar[r]^{\cP'_f((h,\id))} & \cP'_f(( T,g)),}
  \]
where the left (resp.\ right) vertical arrow is the isomorphism
$(\id_S,\beta\alpha\colon u \to fgh)$ (resp.\ $(\id_T,\beta\colon v\to
fg)$).
\end{proof}

\begin{proof}[Proof of Proposition \ref{p.spA}]
Note that $\varprojlim_{x\in \cP_\cX}\cF_x\to \varprojlim_{x\in
\cP'_\cX}\cF_x$ is an isomorphism by Lemma \ref{l.loclim}.  Let $f\colon
X\to \cX$ be a smooth atlas. The following diagram commutes:
\[
  \xymatrix{\Gamma(\cX,\cF) \ar[d]\ar[r] & \Gamma(X,f^*\cF)\ar@<.5ex>[r]\ar@<-.5ex>[r]\ar[d] & \Gamma(X\times_\cX X,g^*\cF)\ar[d]\\
  \varprojlim_{x\in \cP'_\cX} \cF_x\ar[r] & \varprojlim_{x\in \cP'_X} \cF_x\ar@<.5ex>[r]\ar@<-.5ex>[r] & \varprojlim_{x\in \cP'_{X\times_\cX X}} \cF_x.}
\]
Here $g\colon  X\times_\cX X \to \cX$ and the double arrows are induced by
the two projections from $X\times_\cX X$ to~$X$. The top row is exact by the
definition of a sheaf. The bottom row is exact by Lemmas \ref{l.limit},
\ref{l.P} and Remark \ref{s.homeo} (b). The middle and right vertical arrows
are isomorphisms by Propositions \ref{p.sp} and \ref{l.DMpt}. It follows
that the left vertical arrow is also an isomorphism.
\end{proof}

\begin{example}
For $\mathcal{X} = BG$ as in Example \ref{e.PBG}, $\mathcal{F}$ corresponds
(by Corollaries \ref{c.XG} and \ref{l.BG}) to a $\Lambda$-module of finite
type $M$ equipped with an action of $\pi_0(G)$. Thus
$\Gamma(BG,\mathcal{F})$  is the module of invariants $M^{\pi_0(G)}$. By
Example \ref{e.PBG}, $\varprojlim_{x \in \mathcal{P}_{BG}} \mathcal{F}_x$ is
the set of zero cycles $Z^0(\pi_0(G),M)$, and \eqref{e.spA} is the
tautological isomorphism.

If $G$ is finite, the isomorphism \eqref{e.spA} extends to an isomorphism
\[R\Gamma(BG,\cF)\simto R\varprojlim_{x\in \cP_{BG}}\cF_x=R\Gamma(B\pi_0(G),M).\]
However, this no longer holds for $G$ general, as the example of $G=\Gm$ and
$\cF=\Lambda$ already shows (Theorem \ref{l.finite}).
\end{example}

In the rest of this section, we fix a prime number $\ell$.

\begin{definition}\label{d.7.15}
Let $\mathcal{X}$ be an Artin stack. By an \textit{$\ell$-elementary point}
of $\mathcal{X}$ we mean a \textit{representable} morphism $x \colon
\mathcal{S} \to \mathcal{X}$, where $\mathcal{S}$ is isomorphic to a
quotient stack $[S/A]$, where $S$ is a strictly local scheme  endowed with
an action of an elementary abelian $\ell$-group $A$ acting trivially on the
closed point of $S$. If $x \colon [S/A] \to \mathcal{X}$, $y \colon [T/B]
\to \mathcal{X}$ are $\ell$-elementary points of $\mathcal{X}$, a morphism
from $x$ to $y$ is an isomorphism class of pairs $(\varphi,\alpha)$, where
$\varphi \colon [S/A] \to [T/B]$ is a morphism and $\alpha \colon x \to
y\varphi$ is a 2-morphism. An isomorphism between two pairs
$(\varphi,\alpha)\to(\psi,\beta)$ is a 2-morphism $c\colon \varphi\to \psi$
such that $\beta=(y*c)\circ \alpha$. We thus get a category
$\cC'_{\cX,\ell}$, full subcategory of $\Rep{\cX}$ (Remark \ref{s.rep}).
\end{definition}

\begin{prop}\label{p.7.16}
Let $\mathcal{X}$ be an Artin stack.
\begin{enumerate}
\item Let $x\colon \mathcal{S} = [S/A] \to \mathcal{X}$ be an
    $\ell$-elementary point of $\mathcal{X}$, let $s$ be the closed point
    of $S$, and let $\varepsilon$ be the composition $s \to S \to
    \mathcal{S}$. Then $\mathrm{Aut}_{\mathcal{S}(s)}(\varepsilon) = A$,
    and the morphism $x$ induces an injection
\[
\mathrm{Aut}_{\mathcal{S}(s)}(\varepsilon)  \hookrightarrow \mathrm{Aut}_{\mathcal{X}(s)}(x).
\]

\item Let $x \colon [S/A] \to \mathcal{X}$, $y \colon [T/B] \to
    \mathcal{X}$ be $\ell$-elementary points of $\mathcal{X}$,  and let
    $(\varphi,\alpha) \colon x \to y$ be a morphism in $\cC'_{\cX,\ell}$.
    Then there exists a pair $(f,u)$, where $u \colon A \to B$ is a group
    monomorphism and  $f \colon S \to T$ is a $u$-equivariant morphism of
    $\cX$-schemes, such that the morphism of $\cX$-stacks
    $(\varphi,\alpha)$ is induced by the morphism of groupoids $(f,u)
    \colon (S,A)_{\bullet} \to (T,B)_{\bullet}$ over $\cX$. If $(f,u)$ is
    such a pair and $r\in B$, then $(fr,u)$ is also such a pair. If
    $(f_1,u_1)$ and $(f_2,u_2)$ are two such pairs, then $u_1=u_2$ and
    there exists a unique $r\in B$ such that $f_1=f_2r$.

\item Assume that $\mathcal{X} = [X/G]$ for an algebraic space $X$ over a
    base algebraic space $U$, endowed with an action of a smooth group
    algebraic space $G$ over $U$. Then every $\ell$-elementary point
    $x\colon [S/A]\to [X/G]$ lifts to a morphism of $U$-groupoids
    $(x_0,i)\colon (S,A)_\bullet\to (X,G)_\bullet$, where $x_0 \colon S
    \to X$ and $i\colon S \times A \to G$. Moreover, in the situation of
    (b), if $(x_0,i)$, $(y_0,j)$, $(f,u)$ are liftings of $x$, $y$,
    $\varphi$ to $U$-groupoids, respectively, then there exists a unique
    2-morphism of $U$-groupoids (Proposition \ref{p.eq}) lifting $\alpha$
\begin{equation}\label{e.7.16.1}
\xymatrix{(S,A)_{\bullet} \ar[rd]_{(x_0,i)} \ar[r]^{(f,u)} & (T,B)_{\bullet} \ar[d]^{(y_0,j)} \\
&(X,G)_{\bullet}\ultwocell\omit{<2>},}
\end{equation}
given by $r\colon S\to G$ satisfying $x_0(z) = (y_0f)(z)r(z)$ and $i(z,a)
= r(z)^{-1}j(f(z),u(a))r(z a)$.
\end{enumerate}
\end{prop}

\begin{proof}
(a) The first assertion follows from the definition of $[S/A]$ (Notation
\ref{s.wedge}), and the second one from the assumption that $x$ is
representable, hence faithful.

(b) Applying Proposition \ref{p.eqstr} to the groupoids  $(S,A)_{\bullet}$
and $(T,B)_{\bullet}$  over $\cX$, we get a pair $(f,u)$, with  $u \colon S
\times A \to B$ given by Proposition \ref{p.eq} (a), such that $[f/u] =
(\varphi,\alpha)$. The morphism $u$ is constant on $S$, hence induced by a
homomorphism, still denoted $u$, from $A$ to $B$. Since $\varphi$ is
representable, $u$ is a monomorphism. Such a pair $(f,u)$ is unique up to a
unique 2-isomorphism. If $(f_1,u_1)$ and $(f_2,u_2)$ are two choices, a
2-isomorphism from $(f_1,u_1)_{\bullet}$ to $(f_2,u_2)_{\bullet}$ is given
by $r \colon S \to B$ (Proposition \ref{p.eq} (b)), which is necessarily
constant, of value denoted again $r \in B$. Then we have $f_1 = f_2r$ and
$ru_1 = u_2r$, hence $u_1 = u_2$.

(c)  The existence of the liftings follows from Proposition \ref{p.eqstr}
applied to the three groupoids. The description of the morphisms and the
2-morphism of groupoids comes from Proposition \ref{p.eq} (b).
\end{proof}

\begin{remark}
As the referee points out, Definition \ref{d.7.15} is related to the
$\ell$-torsion inertia stack $I(\cX,\ell)$ considered in \cite[Proposition
3.1.3]{AGV}. Indeed, $\cP'_{I(\cX,\ell)}$ can be identified with the
subcategory of $\cC'_{\cX,\ell}$ spanned by $\ell$-elementary points of the
form $S\times BA\to \cX$ and morphisms inducing $\id_A$, where $A=\Z/\ell$.
\end{remark}

\begin{definition}
For a stack of the form $\mathcal{S} = [S/A]$ as in Definition \ref{d.7.15},
the group $A$ is, in view of Proposition \ref{p.7.16} (a), uniquely
determined by $\mathcal{S}$ (up to an isomorphism). We define the
\emph{rank} of $\mathcal{S}$ to be the rank of $A$, and for an
$\ell$-elementary point $x \colon \mathcal{S} \to \mathcal{X}$, we define
the \textit{rank} of $x$ to be the rank of $\mathcal{S}$. $\ell$-elementary
points of rank zero are just geometric points (Definition \ref{d.7.4}). The
full subcategory of $\mathcal{C}'_{\mathcal{X}}$ (Definition \ref{d.7.15})
spanned by $\ell$-elementary points of rank zero is the category
$\mathcal{P}'_{\mathcal{X}}$ (Definition \ref{d.7.4}).
\end{definition}

\begin{definition}
We define the \emph{category of $\ell$-elementary points of $\mathcal{X}$}
to be the category
\begin{equation}\label{e.7.18.2}
\mathcal{C}_{\mathcal{X},\ell}= N_{\mathcal{X},\ell}^{-1}\mathcal{C}'_{\mathcal{X},\ell}
\end{equation}
deduced from $\mathcal{C}'_{\mathcal{X},\ell}$ by inverting the set
$N_{\mathcal{X},\ell}$ of morphisms given by pairs $(f,u)$ (Proposition
\ref{p.7.16} (b)) such that $f\colon S\to T$ carries the closed point of $S$
to the closed point of $T$ and $u$ is a group isomorphism. When no ambiguity
can arise, we will remove the subscript $\ell$ from the notation.
\end{definition}

Although $\cC'_{\cX,\ell}$ is only a $\cU$-category, we will see that
$\cC_{\cX,\ell}$ is essentially small if $\cX$ is a Deligne-Mumford stack of
finite inertia or a global quotient stack (Proposition \ref{l.DMlpt} and
Remark \ref{r.small}).

We may interpret $\cC_{\cX,\ell}$ with the help of an auxiliary category
$\bar\cC'_{\cX,\ell}$ as follows.

\begin{construction}\label{c.C}
Objects of $\bar\cC'_{\cX,\ell}$ are pairs $(x,A)$ such that $x\colon S\to
\cX$ is a geometric point of $\cX$, $A$ is an elementary abelian
$\ell$-group acting on $x$ by $\cX$-automorphisms with trivial action on the
closed point of $S$, and the morphism $[S/A]\to \cX$ is representable.
Morphisms of $\bar\cC'_{\cX,\ell}$ are pairs $(f,u)\colon (x,A)\to (y,B)$,
where $u\colon A\to B$ is a homomorphism and $f\colon x\to y$ is an
equivariant morphism in $\cP'_\cX$. Note that $u$ is necessarily a
monomorphism. By definition, $\bar\cC'_{\cX,\ell}$ is a full subcategory of
$\Eq(\Rep{\cX})$.

We have a natural functor $\rho'\colon\bar \cC'_{\cX,\ell} \to
\cC'_{\cX,\ell}$ sending $(x,A)$ to $[S/A]\to \cX$.  By Proposition
\ref{p.7.16} (a) and (b), the functor is full and essentially surjective,
and in particular cofinal (Lemma \ref{l.cofinal}). If $\varpi' \colon
\mathcal{P}'_{\mathcal{X}} \hookrightarrow \mathcal{C}'_{\mathcal{X},\ell}$
is the inclusion functor, and $\bar \varpi'\colon\cP'_\cX\to
\bar\cC'_{\cX,\ell}$ is the functor sending $x$ to $(x,\{1\})$, which is
also fully faithful, we have a 2-commutative diagram
\begin{equation}\label{e.varpitri}
\xymatrix{& \bar\cC'_{\cX,\ell}\ar[d]^{\rho'}\\
\cP'_\cX\ar[r]^{\varpi'}\ar[ur]^{\bar\varpi'} & \cC'_{\cX,\ell}.}
\end{equation}

Let
\begin{equation}
\bar \cC_{\cX,\ell}=\bar N_{\cX,\ell}^{-1} \bar\cC'_{\cX,\ell}
\end{equation}
be the category deduced from $\bar \cC'_{\cX,\ell}$ by inverting the set
$\bar N_{\cX,\ell}$ of morphisms $(f,u)\colon (S,A)\to (T,B)$ such that $f$
sends the closed point $s$ of $S$ to the closed point $t$ of $T$ and $u
\colon A \to B$ is an isomorphism. The diagram \eqref{e.varpitri} induces a
diagram
\begin{equation}\label{e.varpitri0}
\xymatrix{& \bar\cC_{\cX,\ell}\ar[d]^\rho\\
\cP_\cX\ar[r]^{\varpi}\ar[ur]^{\bar\varpi} & \cC_{\cX,\ell}.}
\end{equation}
\end{construction}

The functor $\rho$ is essentially surjective, and its effects on morphisms
can be described as follows. Let $(x,A)$ and $(y,B)$ be objects of
$\bar\cC_{\cX}$. The action of $B$ on $(y,B)$ by automorphisms in
$\bar\cC'_{\cX}$ induces an action of $B$ on $(y,B)$ by automorphisms in
$\bar\cC_{\cX}$, and, in turn, an action of $B$ on
$\Hom_{\bar\cC_{\cX}}((x,A),(y,B))$. This action is compatible with
composition in the sense that if $f\colon (x,A)\to (y,B)$, $g\colon (y,B)\to
(z,C)$ are morphisms of $\bar\cC_\cX$, and $b\in B$, then
$g\circ(fb)=(g(\theta(g)(b)))\circ f$, where $\theta\colon \bar\cC_\cX\to
\cA$ is the functor induced by the functor $\bar\cC'_\cX\to \cA$ carrying
$(x,A)$ to $A$. Here $\cA$ denotes the category whose objects are elementary
abelian $\ell$-groups and whose morphisms are monomorphisms.

\begin{prop}\label{l.varpi}\leavevmode
\begin{enumerate}
\item The functor $\rho$ induces a bijection
\begin{equation}\label{e.locquot}
\Hom_{\bar \cC_\cX}((x,A),(y,B))/B \simto
\Hom_{\cC_\cX}(\rho(x,A),\rho(y,B)).
\end{equation}

\item The functors $\bar\varpi$ and $\varpi$ are fully faithful.
\end{enumerate}
\end{prop}

\begin{proof}
(a) Indeed, consider the quotient category $\cC^\sharp_\cX$ having the same
objects as $\bar \cC_\cX$ with morphisms defined by
\[\Hom_{\cC^\sharp_\cX}((x,A),(y,B))=\Hom_{\bar \cC_\cX}((x,A),(y,B))/B,\]
and the quotient functor $\rho^\sharp \colon\bar\cC_\cX\to \cC^\sharp_\cX$.
By the universal properties of $\rho$, $\rho^\sharp$, and the localization
functors $\bar \cC'_\cX\to \bar \cC_\cX$, $\cC'_\cX\to \cC_\cX$, we obtain
an equivalence between $\cC^\sharp_\cX$ and $\cC_\cX$, compatible with
$\rho$ and $\rho^\sharp$.

(b) It follows from (a) that $\rho$ induces an equivalence of categories
from the full subcategory of $\bar\cC_{\cX}$ spanned by the image of
$\cP_{\cX}$ to the full subcategory of $\cC_{\cX}$ spanned by the image of
$\cP_{\cX}$. Thus it suffices to show that $\bar\varpi$ is fully faithful.
The functor $\bar\cC'_{\cX}\to \cP'_\cX$ sending $(x,A)$ to $x$ is a
quasi-retraction of $\bar\varpi'$, and induces a quasi-retraction of
$\bar\varpi$. Here, by a quasi-retraction of a functor $F$, we mean a
functor $G$ endowed with a natural isomorphism $GF\simeq \id$.  Thus
$\bar\varpi$ is faithful. Let us show that $\bar\varpi$ is full. Let $x,x'$
be geometric points of $\mathcal{X}$. By definition, any morphism $f\colon
x\to x'$ in $\bar\cC_\cX$ is of the form $(t_n,v_n)(s_n,u_n)^{-1}\dots
(t_1,v_1)(s_1,u_1)^{-1}$, where $(t_i,v_i)\colon (x_i,A_i)\to
(y_{i+1},B_{i+1})$ is in $\bar\cC'_\cX$ and $(s_i,u_i)\colon (x_i,A_i)\to
(y_i,B_i)$ is in $\bar N_\cX$ for $1\le i\le n$, $y_1=x$, $y_{n+1}=x'$,
$B_1=B_{n+1}=\{1\}$. Then $u_i\colon A_i\to B_i$ is an isomorphism and
$v_i\colon A_i\to B_{i+1}$ is a monomorphism. Thus $A_i=B_i=\{1\}$.
Moreover, $t_i$ is in $\cP'_\cX$ and $s_i$ is in $M_\cX$. It follows that
$f=\bar\varpi(a)$, where $a=t_ns_n^{-1}\dots t_1s_1^{-1}$ is in $\cP_\cX$.
\end{proof}

\begin{remark}\label{r.Cfunc}
For any \emph{representable} morphism $f\colon \cX\to \cY$ of Artin stacks,
composition with $f$ induces functors $\cC_f\colon \cC_\cX\to \cC_\cY$ and
$\bar\cC_f\colon \bar\cC_\cX\to \bar\cC_\cY$. As in Remark \ref{s.homeo}
(a), $\cC_f$ and $\bar\cC_f$ are equivalences of categories if $f$ is a
schematic universal homeomorphism.
\end{remark}

\begin{definition}
Morphisms in the categories $\bar \cC_{\cX,\ell}$ and
$\mathcal{C}_{\mathcal{X},\ell}$ are in general difficult to describe. When
$\mathcal{X}$ is a Deligne-Mumford stack of finite inertia, the categories
$\bar\cC_{\cX,\ell}$ and $\cC_{\mathcal{X},\ell}$ admit simpler
descriptions, as in Proposition \ref{l.DMpt}. Let us call a \emph{DM
$\ell$-elementary point of $\mathcal{X}$} a pair $(x,A)$, where $x\colon
s\to \cX$ is a geometric point of $\mathcal{X}$ and $A$ an $\ell$-elementary
abelian subgroup of $\mathrm{Aut}_{\cX(s)}(x)$. Define a morphism from
$(x\colon s\to \cX,A)$ to $(y\colon t\to \cX,B)$ to be an
$\mathcal{X}$-morphism $\cX_{(x)}\to \cX_{(y)}$ such that $f(A)\subset B$,
where $f\colon \Aut_{\cX(s)}(x)\to \Aut_{\cX(t)}(y)$ is defined as follows.
Note that $I_{(y)}\coloneqq I_{\cX}\times_{\cX} \cX_{(y)}$ is finite and
unramified over $\cX_{(y)}$, thus is a finite disjoint union of closed
subschemes of $\cX_{(y)}$ by \cite[Corollaire 18.4.7]{EGAIV}. For $a \in
\Aut_{\cX(s)}(x)$, the point $s\to I_{(y)}$ given by $a$ lies in same
component as the point $t\to I_{(y)}$ given by $f(a)$. We thus get a
category $\bar C_{\cX,\ell}$. We define the \emph{category of DM
$\ell$-elementary points of $\cX$} to be the category $C_{\cX,\ell}$ having
the same objects as $\bar C_{\cX,\ell}$ and such that
$\Hom_{C_{\cX,\ell}}((x,A),(y,B))=\Hom_{\bar C_{\cX,\ell}}((x,A),(y,B))/B$.
We omit the subscript $\ell$ from the notation when no ambiguity arises.
\end{definition}

Note that for $(x,A)\in \bar C_\cX$, the morphism $[\cX_{(x)}/A]\to \cX$ is
representable.

\begin{prop}\label{l.DMlpt}
Let $\cX$ be a Deligne-Mumford stack of finite inertia. Then the functor
$C_\cX\to \bar\cC'_\cX$ carrying $(x,A)$ to $(\cX_{(x)}\to \cX,A)$ induces
an equivalence of categories $\iota\colon \bar C_\cX\to \bar \cC_\cX$ and,
in turn, an equivalence of categories $\cC_\cX\to C_\cX$.
\end{prop}

\begin{proof}
A quasi-inverse of $\iota$ is induced by the functor $\bar\cC'_\cX\to \bar
C_\cX$ carrying $(S\to \cX,A)$ to $(s\to \cX,A)$, where $s$ is the closed
point of $S$.
\end{proof}

In the sequel, for $\mathcal{X}$  a Deligne-Mumford stack of finite inertia,
we will often identify the categories $\mathcal{C}_{\mathcal{X}}$ and
$C_{\mathcal{X}}$ by the equivalence of Proposition \ref{l.DMlpt} and call
$DM$ $\ell$-elementary points just $\ell$-elementary points.

\begin{construction}
Let $\mathcal{X}$ be an Artin stack. Let $\mathcal{F}$ be a cartesian sheaf
on $\mathcal{X}$. If $x \colon [S/A] \to \mathcal{X}$ is an
$\ell$-elementary point of $\mathcal{X}$, let $\mathcal{F}_x \coloneqq
x^*\mathcal{F}$, and
\[
\Gamma(x,\mathcal{F}_x) \coloneqq \Gamma([S/A],\mathcal{F}_x) \simeq \Gamma(BA,\mathcal{F}_s) \simeq \mathcal{F}_s^A.
\]
If $(\varphi,\alpha) \colon [S/A] \to [Y/B]$ is a morphism in
$\mathcal{C}'_{\mathcal{X}}$, we have a natural map $\Gamma(x,\mathcal{F}_x)
\to \Gamma(y,\mathcal{F}_y)$ given by restriction, and in this way we get a
presheaf $\fq'\cF\colon x \mapsto \Gamma(x,\mathcal{F}_x)$ on
$\mathcal{C}'_{\mathcal{X}}$, which factors through a presheaf $\fq \cF$ on
$\mathcal{C}_{\mathcal{X}}$. The canonical restriction maps
$\Gamma(X,\mathcal{F}) \to \Gamma(x,\mathcal{F}_x)$ yield a map
\begin{equation}\label{e.7.22.1}
\Gamma(X,\mathcal{F}) \to \varprojlim_{x \in
\mathcal{C}_{\mathcal{X}}} \Gamma(x,\mathcal{F}_x).
\end{equation}
If $x \colon [S/A] \to \mathcal{X}$ is an elementary point of rank zero,
i.e.\ a geometric point of $X$ (Definition \ref{d.7.15}),
$\Gamma(x,\mathcal{F}_x) = \mathcal{F}_x$, and by restriction via $\varpi
\colon \mathcal{P}_{\mathcal{X}} \hookrightarrow \mathcal{C}_{\mathcal{X}}$,
the presheaf $\fq \cF$ induces the presheaf $\mathfrak{p}\mathcal{F}$
(Construction \ref{c.p}). Therefore we have a commutative diagram
\begin{equation}\label{e.CPtri}
\xymatrix{\Gamma(X,\mathcal{F}) \ar[dr] \ar[r] & \varprojlim_{x \in \mathcal{C}_{\mathcal{X}}} \Gamma(x,\mathcal{F}_x) \ar[d] \\
{} & \varprojlim_{x \in \mathcal{P}_{\mathcal{X}}} \mathcal{F}_x,}
\end{equation}
where the horizontal (resp.\ oblique) map is \eqref{e.7.22.1} (resp.\
\eqref{e.spA}), and the vertical one is restriction via $\varpi$.
\end{construction}

\begin{prop}
Let $\cX$ be a locally noetherian Artin stack, $\Lambda$ a noetherian
commutative ring, and $\cF$ a constructible sheaf of $\Lambda$-modules on
$\cX$. Then \eqref{e.7.22.1} is an isomorphism.
\end{prop}

\begin{proof}
The oblique map of \eqref{e.CPtri} is an isomorphism by Proposition
\ref{p.spA}. By Lemma \ref{l.loclim}, the vertical map is obtained by
applying the functor
$\Gamma(\widehat{\cC'_\cX},-)=\varprojlim_{\cC'_\cX}(-)$ to the adjunction
map
\[\alpha\colon \fq' \cF \to \varpi'_*\varpi'^*\fq'\cF=\varpi'_* \fp' \cF,\]
where $\varpi'\colon \cP'_\cX\to \cC'_\cX$. Thus it suffices to show that
$\alpha$ is an isomorphism. Here
\[(\varpi'_*,\varpi'^*)\colon \widehat{\cP'_\cX}\to \widehat{\cC'_\cX}\]
is the morphism of topoi defined by $(\varpi'^*\cE)(z)=\cE(\varpi'(z))$ and
$(\varpi'_*\cG)(x)=\varprojlim_{(t,\phi)\in (\varpi'\downarrow x)}\cG_t$,
where for an $\ell$-elementary point $x\colon [S/A]\to \cX$,
$(\varpi'\downarrow x)$ is the category of pairs $(t,\phi)$, where $t$ is a
geometric point of $\cX$ and $\phi\colon \varpi' t\to x$ is a morphism in
$\cC'_\cX$, which is equivalent to $\cP'_{[S/A]}$. Let $\cA$ be the groupoid
with one object~$*$ and fundamental group $A$. Consider the functor $F\colon
\cA\to (\varpi'\downarrow x)$ sending $*$ to $(x\varepsilon,\varepsilon)$,
where $\varepsilon\colon S\to [S/A]$, and $a\in A$ to the morphism
$x\varepsilon\to x\varepsilon$ induced by the action of $a$. For any object
$(t,\phi)$ of $(\varpi'\downarrow x)$, the category $((t,\phi)\downarrow F)$
is a simply connected groupoid. Therefore, $F$ is cofinal and
\[\alpha(x)\colon \Gamma(x,\cF_x)\to \varprojlim_{(t,\phi)\in (\varpi'\downarrow x)}\cF_t\simto \varprojlim_{\cA}\cF_{x
\varepsilon} \simeq \cF_s^A
\]
is an isomorphism. Here $s$ is the closed point of $S$.
\end{proof}

In the next section we study higher cohomological variants of
\eqref{e.7.22.1}.

\section{A generalization of the structure theorems to Artin stacks}\label{s.6}

In this section we fix an algebraically closed field $k$  and a prime number
$\ell$ invertible in $k$.

\begin{construction}
Let $\mathcal{X}$ be an Artin stack, and let $K \in
D_{\cart}(\mathcal{X},\F_{\ell})$. For $q \in \Z$, consider the presheaf of
$\F_{\ell}$-vector spaces on $\mathcal{C}_{\mathcal{X}}$ \eqref{e.7.18.2}
\[
(x \colon [S/A] \to \mathcal{X}) \mapsto H^q([S/A],K_x) \simeq H^q(BA,K_s)
\]
(where $K_x \coloneqq x^*K$ and $s$ is the closed point of $S$), and let
\begin{equation}\label{e.R}
R^q(\mathcal{X},K) \coloneqq \varprojlim_{(x \colon \mathcal{S} \to \mathcal{X}) \in \mathcal{C}_{\mathcal{X}}} H^q(\mathcal{S},K_x).
\end{equation}
The restriction maps $H^q(\mathcal{X},K) \to H^q(\mathcal{S},K_x)$ define a
map
\begin{equation}\label{e.a}
a^q_{\mathcal{X},K} \colon H^q(\mathcal{X},K) \to R^q(\mathcal{X},K).
\end{equation}
We denote by $a_{\mathcal{X},K}$ the direct sum of these maps:
\begin{equation}\label{e.astar}
a_{\mathcal{X},K} = \bigoplus_q a^q_{\mathcal{X},K}\colon H^*(\mathcal{X},K) \to R^*(\mathcal{X},K).
\end{equation}
If $K$ has a (pseudo-)ring structure (Construction \ref{c.Dring}), then both
sides of \eqref{e.astar} are $\F_{\ell}$-(pseudo-)algebras, and
$a_{\mathcal{X},K}$ is a homomorphism of $\F_\ell$-(pseudo-)algebras.
\end{construction}

\begin{definition}\label{d.quot}
We say that an Artin stack $\cX$ over $k$ is a \emph{global quotient stack}
if $\cX$ is equivalent to a stack of the form $[X/G]$ for $X$ a
\emph{separated} algebraic space of finite type over $k$ and $G$ an
algebraic group over $k$. We say that an Artin stack $\cX$ of finite
presentation over $k$ \emph{has a stratification by global quotients} if
there exists a stratification of $\cX_\red$ by locally closed substacks such
that each stratum is a global quotient stack.
\end{definition}

Recall that an Artin stack over $k$ is of finite presentation if and only if
it is quasi-separated and of finite type over $k$. Note that our Definition
\ref{d.quot} differs from \cite[Definition 2.9]{EHKV} and \cite[Definition
3.5.3]{Kresch} because we allow quotients by non-affine algebraic groups.

The following theorem is our main result.

\begin{theorem}\label{p.str}
Let $\cX$ be an Artin stack of finite presentation over $k$ admitting a
stratification by global quotients, $K\in D^+_c(\cX,\F_\ell)$.
\begin{enumerate}
\item $R^q(\cX,K)$ is a finite-dimensional $\F_\ell$-vector space for all
    $q$. Moreover, $R^*(\cX,\F_\ell)$ is a finitely generated
    $\F_\ell$-algebra and, for $K$ in $D^b_c(\cX,\F_\ell)$, $R^*(\cX,K)$
    is a finitely generated $R^*(\cX,\F_\ell)$-module.

\item If $K$ is a pseudo-ring in $D^+_c(\cX,\F_\ell)$, then $\Ker
    a_{\cX,K}$ \eqref{e.a} is a nilpotent ideal of $H^*(\cX,K)$. If,
    moreover, $K$ is commutative and $\cX$ is a Deligne-Mumford stack with
    finite inertia or a global quotient stack, then $a_{\cX,K}$ is a
    uniform $F$-isomorphism (Definition \ref{s.grvec}).
\end{enumerate}
\end{theorem}

\begin{remark}
\begin{enumerate}
\item A non separated scheme of finite presentation over $k$ is not a
    global quotient stack in the sense of Definition \ref{d.quot} in
    general. Michel Raynaud gave the example of an affine plane with
    doubled origin. More generally, if $Y$ is a separated smooth scheme of
    finite type over $k$, and $Y'$ is obtained by gluing two copies of
    $Y$, $Y^{(1)}$ and $Y^{(2)}$, along the complement of a nonempty
    closed subset of codimension $\ge 2$, then, for any algebraic group
    $G$, every $G$-torsor $X$ over $Y'$ is non separated. To see this, we
    may assume $G$ smooth. By \'etale localization on $Y$, we may further
    assume that $X$ admits a section $s_i$ over $Y^{(i)}$, $i=1,2$. Assume
    that $X$ is separated. The restrictions of $s_1$ and $s_2$ to
    $V=Y^{(1)}\cap Y^{(2)}$ provide a section of $G\times V$, which
    extends by Weil's extension theorem (see \cite[Theorem 4.4.1]{BLR} for
    a generalization) to a section of $G\times Y^{(2)}$. Via this section,
    $s_1$ and $s_2$ can be glued to give a trivialization of $X$ over
    $Y'$, contradicting the separation assumptions.

\item Recall \cite[Proposition 3.5.9]{Kresch} that, if for every geometric
    point $\eta \to \cX$, the inertia $I_\eta=\eta\times_{\cX} I_{\cX}$ is
    affine, where $I_{\cX}=\cX\times_{\Delta_\cX,\cX\times \cX,\Delta_\cX}
    \cX$, then $\cX$ has a stratification by global quotients in the sense
    of \cite[Definition 3.5.3]{Kresch}, and a fortiori in the sense of
    Definition \ref{d.quot}.

\item On the other hand, the fact that $\cX$ has a stratification by
    global quotients in the sense of Definition \ref{d.quot} imposes
    restrictions on its inertia groups. In fact, if $k$ has characteristic
    zero, then, for any geometric point $\eta=\Spec(K) \to \cX$ with $K$
    algebraically closed, $I_\eta^0/(I_{\eta})_\aff$ is an abelian variety
    over $K$ defined over $k$. Here $I_\eta^0$ is the identity component
    of $I_\eta$ and $(I_\eta)_{\aff}$ is the largest connected affine
    normal subgroup of $I_\eta$. Indeed, if $\cX=[X/G]$, then $I_\eta$ is
    a subgroup of $G\otimes_k K$, so that $I_\eta^0/(I_{\eta})_\aff$ is
    isogenous to an abelian subvariety of $(G^0/G_{\aff})\otimes_k K$,
    hence is defined over $k$ (for an abelian variety $A$ over $k$,
    torsion points of order invertible in $k$ of $A\otimes_k K$ are
    defined over $k$ as $k$ is algebraically closed).

\item For an Artin stack $\cX$ of finite presentation over $k$ and a
    commutative ring $K$ in $D^b_c(\cX,\F_\ell)$, we do not know whether
    $H^*(\cX,K)$ is a finitely generated $\F_\ell$-algebra or whether
    $a_{\cX,K}$ is a uniform $F$-isomorphism in general, even under the
    assumption that $\cX$ has a stratification by global quotients. It may
    be the case that to treat the general case we would need to
    reformulate the theory in a relative setting.
\end{enumerate}
\end{remark}

The proof of Theorem \ref{p.str} will be given in Section~\ref{s.7}. In the
rest of this section we show that Theorem \ref{p.str} (b) implies Theorem
\ref{t.main2} (b).

\begin{construction}\label{c.8.4}
Let $G$ be an algebraic group over $k$ and $X$ an algebraic space over $k$
endowed with an action of $G$ (here we do not assume $X$ to be of finite
type over $k$). To show that Theorem \ref{p.str} (b) implies Theorem
\ref{t.main2} (b), we will proceed in two steps.
\begin{enumerate}[(1)]
\item For $K \in D^+_\cart([X/G],\F_{\ell})$ we will construct a
    homomorphism
\begin{equation}\label{e.RtoR}
\alpha \colon R^*([X/G],K) \to R^*_G(X,K),
\end{equation}
which will be a homomorphism of $\F_{\ell}$-(pseudo-)algebras if $K$ has a
(pseudo-)ring structure, and whose composition with $a_{[X/G],K} \colon
H^*([X/G],K) \to R^*([X/G],K)$ will be $a_G(X,K)$ \eqref{e.aE}.

\item We will show that $\alpha$ is an isomorphism.
\end{enumerate}

Let us construct $\alpha$. Recall that
\[
R^q([X/G],K) = \varprojlim_{(x \colon \mathcal{S} \to [X/G]) \in \mathcal{C}_{[X/G]}} H^q(\mathcal{S},K_x),
\]
and $R^q_G(X,K)= \varprojlim_{(A,A',g) \in \mathcal{A}_G(k)^{\natural}}
H^0(X^{A'},R^q\pi_*[1/c_g]^*K)$ \eqref{e.RE}. We first compare the
categories $\mathcal{A}_G(k)^{\natural}$ and $\mathcal{C}_{[X/G]}$ by means
of a third category $\mathcal{C}_{X,G}$ mapping to them by functors $E$
and~$\Pi$:
\begin{equation}\label{e.8.4.2}
\xymatrix{{} & \mathcal{C}_{X,G} \ar[dl]_{E} \ar[dr]^{\Pi} & {} \\
\mathcal{C}_{[X/G]} & {} & \mathcal{A}_G(k)^{\natural}.}
\end{equation}
The category $\mathcal{C}_{X,G}$ is cofibered over
$\mathcal{A}_G(k)^{\natural}$ by $\Pi$. The fiber category of
$\mathcal{C}_{X,G}$ at an object $(A,A',g)$ of $\mathcal{A}_G(k)^{\natural}$
is the category of points $P_{X^{A'}}$ of the fixed point space of $A'$ in
$X$. If $(a,b) \colon (A,A',g) \to  (Z,Z',h)$ is a morphism in
$\mathcal{A}_G(k)^{\natural}$ (cf.\ \eqref{e.6.15.2}), we define the pushout
functor $P_{b^{-1}} \colon P_{X^{A'}} \to P_{X^{Z'}}$ to be the functor
induced by $b^{-1} \colon X^{A'} \to (X^{A'})b^{-1} = X^{bA'b^{-1}} \subset
X^{Z'}$. If $x \colon s \to X^{A'}$ is a geometric point of $X^{A'}$, let
$E_{(A,A',g)}(x) \colon [s/A] \to [X/G]$ be the $\ell$-elementary point of
$[X/G]$ defined by the composition
\[
E_{(A,A',g)}(x) \colon [s/A] \xrightarrow{[x/A]} [X^{A'}/A] \xrightarrow{[1/c_g]} [X/G].
\]
For $(x \colon s \to X^{A'}) \in P_{X^{A'}}$, $(y \colon t \to X^{Z'}) \in
P_{X^{Z'}}$, let $u \colon x \to y$ be a morphism in $\mathcal{C}_{X,G}$
above $(a,b) \colon (A,A',g) \to (Z,Z',h)$. The morphism
\[
E(u) \colon E_{(A,A',g)}(x) \to E_{(Z,Z',h)}(y)
\]
is defined as follows. By definition, $u$ is a commutative square
\[
\xymatrix{X^{A'} \ar[d]_{b^{-1}} & (X^{A'})_{(x)} \ar[d]^f \ar[l]_x \\
X^{Z'} & (X^{Z'})_{(y)} \ar[l]_y,}
\]
where the horizontal arrows denote by abuse of notation the morphisms
induced by strict localizations. It gives the (2-commutative) square on the
right of the diagram
\begin{equation}\label{e.8.4.2b}
\xymatrix{&\ar[ld]_{[1/c_g]}[X^{A'}/A] \ar[d]^{[b^{-1}/c_a]} & [(X^{A'})_{(x)}/A] \ar[l]_{[x/A]} \ar[d]^{[f/c_a]} \\
[X/G]\urtwocell\omit{<2>}& \ar[l]^{[1/c_h]}[X^{Z'}/Z] & [(X^{Z'})_{(y)}/Z] \ar[l]_{[y/Z]},}
\end{equation}
whose composition with the 2-morphism (given by $b$) in the left triangle of
\eqref{e.8.4.2b} (appearing in \eqref{e.6.15.3}) is the morphism $E(u)$.
This defines the functor $E$ in \eqref{e.8.4.2}.

Fix $q \in \Z$. Denote by $H^q(K_{\bullet})$ the projective system $((\xi
\colon \mathcal{S} \to [X/G]) \mapsto H^q(\mathcal{S},K_{\xi}))$ on
$\mathcal{C}_{[X/G]}$, whose projective limit is $R^q([X/G],K)$ \eqref{e.R}.
In other words, $R^q([X/G],K) =
\Gamma(\widehat{\mathcal{C}_{[X/G]}},H^q(K_{\bullet}))$. We have an inverse
image map
\begin{equation}\label{e.8.4.3}
\Gamma(\widehat{\mathcal{C}_{[X/G]}},H^q(K_{\bullet})) \to \Gamma(\widehat{\mathcal{C}_{X,G}},E^*H^q(K_{\bullet}))
\simeq \Gamma(\widehat{\cA_G(k)^\natural},\Pi_*E^*H^q(K_\bullet)).
\end{equation}
By the cofinality lemma (Lemma \ref{l.cofcof}) below,
\[
(\Pi_*E^*H^q(K_\bullet))_{(A,A',g)} \simeq \varprojlim_{x \in P_{X^{A'}}}H^q([x/A],K_x).
\]
By Proposition \ref{p.sp} (applied to the algebraic space $X^{A'}$), we have
a natural isomorphism
\[
\varprojlim_{x \in P_{X^{A'}}}H^q([x/A],K_x) \simto H^0(X^{A'},R^q\pi_*([1/c_g]^*K))
\]
where $\pi \colon [X^{A'}/A] = BA \times X^{A'} \to X^{A'}$ is the
projection, and $[1/c_g] \colon [X^{A'}/A] \to [X/G]$ is the morphism in
\eqref{e.6.15.3}. Finally, we find a natural isomorphism
\[
\Gamma(\widehat{\cA_G(k)^\natural},\Pi_*E^*H^q(K_\bullet)) \simto \varprojlim_{(A,A',g) \in \mathcal{A}_G(k)^{\natural}} H^0(X^{A'},R^q\pi_*([1/c_g]^*K)),
\]
which, by the definition of $R^q_G(X,K)$ \eqref{e.RE}, can be rewritten
\begin{equation}\label{e.8.4.4}
\Gamma(\widehat{\cA_G(k)^\natural},\Pi_*E^*H^q(K_\bullet)) \simto R^q_G(X,K).
\end{equation}
The composition of \eqref{e.8.4.3} and \eqref{e.8.4.4} yields the desired
map $\alpha$ \eqref{e.RtoR}.
\end{construction}

\begin{lemma}\label{l.cofcof}
Let $\Pi\colon \cC\to \cE$ be a cofibered category, let $e$ be an object of
$\cE$, and let $\Pi_e$ be the fiber category of $\Pi$ above $e$. Then the
functor $F\colon \Pi_e\to (\Pi\downarrow e)$ is cofinal. In
  particular, for every presheaf $\cF$ on $\cC$, $(\Pi_* \cF)(e)\simto \varprojlim_{c\in \Pi_e} \cF(c)$.
\end{lemma}

\begin{proof}
For every object $(c,f\colon \Pi c\to e)$ of $(\Pi\downarrow e)$, $(c,f)\to
F(f_*c)$ is an initial object of $((c,f)\downarrow F)$. Thus
$((c,f)\downarrow F)$ is connected.
\end{proof}

\begin{prop}\label{p.RR}
Under the assumptions of Construction \ref{c.8.4}, the functor $E$ is
cofinal. In particular, \eqref{e.RtoR} is an isomorphism.
\end{prop}

\begin{cor}
Theorem \ref{p.str} (b) implies Theorem \ref{t.main2} (b).
\end{cor}

\begin{proof}[Proof of Proposition \ref{p.RR}]
The second assertion follows from the first assertion and the construction
of \eqref{e.RtoR}. To show the first assertion, since the functor
$\cC_{X,G_{\red}} \to \cC_{X,G}$ is an isomorphism and the functor
$\cC_{[X/G_{\red}]} \to \cC_{[X/G]}$ is an equivalence of categories by
Remark \ref{r.Cfunc}, we may assume $G$ smooth.

Let $N$ be the set of morphisms in $\mathcal{\cC}_{X,G}$ whose image under
$E$ is an isomorphism in $\mathcal{C}_{[X/G]}$. Then $E$ factors as
\[
\cC_{X,G}\to \cB \coloneqq N^{-1}\mathcal{\cC}_{X,G} \xrightarrow{F} \mathcal{C}_{[X/G]}.
\]
By Lemma \ref{l.loclim}, $\cC_{X,G}\to \cB$ is cofinal. Thus it suffices to
show that $F$ is cofinal. We will show that:
\begin{enumerate}
\item $F$ is essentially surjective;

\item $F$ is full.
\end{enumerate}
This will imply that $F$ is cofinal by Lemma \ref{l.cofinal}. For the proof
it is convenient to use the following notation. For an object $x$ of
$P_{X^{A'}}$ above an object $(A,A',g)$ of $\mathcal{A}_G(k)^{\natural}$, we
will denote the resulting object of $\cB$ by the notation
\[
(x,(A,A',g)).
\]

Let us prove (a). For every $\ell$-elementary point $\xi \colon [S/A] \to
[X/G]$, we choose an algebraic closure $\bar s$ of the closed point $s$ of
$S$ and we let $\bar\xi$ denote the composite $[\bar s/A]\to
[S/A]\xrightarrow{\xi}[X/G]$. We say that a lifting
\[\sigma=(a\in X(\bar
s), \alpha\in \cHom(A,G)(\bar s),\iota\colon [a/\alpha]\simeq \bar\xi)
\]
of $\bar\xi$ (Proposition \ref{p.7.16} (c)) is \emph{rational} if $\alpha\in
\cHom(A,G)(k)$. Recall that $\alpha$ is injective. Here $\cHom(A,G)$ is the
scheme of group homomorphisms from $A$ to $G$ (Section~\ref{s.3}). A
rational lifting $\sigma$ of $\bar \xi$ defines an object
\[\omega_\sigma=\omega_{a,\alpha}=(a,(\alpha(A),\alpha(A),1))\]
of $\cB$ and an isomorphism
\[\psi_{\xi,\sigma}\colon F(\omega_\sigma)\to \bar \xi \to \xi\]
in $\cC_{[X/G]}$. By Corollary \ref{c.Serre0}, every element of
$\cHom(A,G)(\bar s)$ is conjugate by an element of $G(\bar s)$ to an element
of $\cHom(A,G)(k)$. Thus every $\bar \xi$ admits a rational lifting. It
follows that $F$ is essentially surjective.

Let us prove (b). For any object $\mu=(x,(A,A',g))$ of $\cB$,
$\sigma_\mu=(\bar x,c_g\colon A\to G,\id)$ is a rational lifting of
$\overline{F(\mu)}$ and $\psi_{F(\mu),\sigma_\mu}=F(m_\mu)$, where
\[m_\mu\colon \omega_{\sigma_\mu}=(\bar x,(g^{-1}Ag,g^{-1}Ag,1))\to (x,(A,A',g))=\mu\]
is the inverse of the obvious morphism in $N$ above the morphism
$(g,1)\colon (g^{-1}Ag,g^{-1}Ag,1)\from (A,A',g)$ of $\cA_G(k)^\natural$.
Now if $\mu$ and $\nu$ are objects of $\cB$ and $f\colon F(\mu)\to F(\nu)$
is a morphism in $\cC_{X,G}$, then $f=F(m_\nu u m_\mu^{-1})$, where $u$ is
obtained from the following lemma applied to $f$, $\sigma=\sigma_\mu$,
$\tau=\sigma_\nu$. Thus $F$ is full.
\end{proof}

\begin{lemma}
Let $\xi \colon [S/A] \to [X/G]$, and let $\eta \colon [T/B] \to [X/G]$ be
$\ell$-elementary points of $[X/G]$. For every morphism $f \colon \xi \to
\eta$ in $\mathcal{C}_{[X/G]}$, every rational lifting $\sigma$ of
$\bar\xi$, and every rational lifting $\tau$ of $\bar\eta$, there exists a
morphism $u \colon \omega_{\sigma} \to \omega_{\tau}$ in $\mathcal{B}$
making the following diagram commute:
\[
\xymatrix{F(\omega_{\sigma}) \ar[r]^{{F}(u)} \ar[d]_{\psi_{\xi,\sigma}} & F(\omega_{\tau}) \ar[d]^{\psi_{\eta,\tau}} \\
\xi \ar[r]^f & \eta.}
\]
\end{lemma}

\begin{proof}
Given a triple $(f,\sigma,\tau)$ as in the lemma, we say that
$L(f,\sigma,\tau)$ holds if there exists $u$ satisfying the condition of the
lemma. Given $f\colon \xi\to \eta$, we say that $L(f)$ holds if for every
rational lifting $\sigma$ of $\bar\xi$ and every rational lifting $\tau$ of
$\bar\eta$, $L(f,\sigma,\tau)$ holds.

\emph{Step 1. First reductions.} If $L(f\colon \xi\to \eta,\sigma,\tau)$ and
$L(g\colon \eta\to \zeta,\tau,\kappa)$ hold, where $\sigma$, $\tau$,
$\kappa$ are rational liftings of $\xi$, $\eta$, $\zeta$, respectively, then
$L(gf,\sigma,\kappa)$ holds, where $(gf,\sigma,\kappa)$ is the composed
triple $(gf,\sigma,\kappa)=(g,\tau,\kappa)(f,\sigma,\tau)$. In particular,
if $L(f \colon \xi \to \eta)$ and $L(g \colon \eta \to \zeta)$ hold, then
$L(gf)$ holds. Moreover, if $L(f)$ holds for an isomorphism $f$, then
$L(f^{-1})$ holds. Thus we may assume that $f$ is a morphism of
$\cC'_{[X/G]}$. Then $f=([h/\gamma],\theta)$, where $(h \colon S \to
T,\gamma \colon A \to B)$ is an equivariant morphism and $\theta\colon
\xi\to \eta''\coloneqq \eta[h/\gamma]$ is a 2-morphism. Note that $f$ can be
decomposed as
\[\xi\xrightarrow{f_1} \eta''\xrightarrow{f_2}\eta'\xrightarrow{f_3}\eta,\]
where $\eta'=\eta[\id_T/\gamma]$, $f_1=(\id_{[S/A]},\theta)$,
$f_2=([h/\id_A],\id_{\eta''})$, $f_3=([\id_T/\gamma],\id_{\eta'})$, as shown
by the diagram
\[\xymatrix{
[S/A]\ar@{=}[r]\xlowertwocell[rrd]{}_\xi{^\theta} &
[S/A]\ar[r]^{[h/\id_A]}\ar[rd]_{\eta''} &
[T/A]\ar[r]^{[\id_T/\gamma]}\ar[d]^{\eta'}
& [T/B]\ar[ld]^{\eta}\\
&&[X/G].}
\]

\emph{Step 2. $L(f)$ holds for any morphism of the form
$f=(\id_{[S/A]},\theta)$, and in particular $L(f_1)$ holds.} Let
$\sigma=(a,\alpha,\iota)$ and $\tau=(b,\beta,\epsilon)$ be rational liftings
of $\bar\xi$ and $\bar \eta$, respectively. Via the liftings, $\theta$ is
given by $g\in J(\bar s)$, where $J=\Trans_G(\beta(A),\alpha(A))$, and
$a=bg$. Let $g'\in J(k)$ be a rational point of the connected component of
$J$ containing $g$. Then $h\coloneqq g'^{-1}g\in H(\bar s)$, where $H$ is
the identity component of $\Norm_G(\alpha(A))$. Let $e$ be the generic point
of $H_{\bar s}$. Note that $\cP_{[H_{\bar s}/H_{\bar s}]}$ is equivalent to
$P_{\bar s}$, and hence is a simply connected groupoid. Thus the morphism in
$\cP_{[H_{\bar s}/H_{\bar s}]}$ induced by the diagram $1\from e\to h$ in
$P_{H_{\bar s}}$ can be identified with the 2-morphism $i_1\to i_h$ given by
$h$, where $i_1,i_h\colon \bar s\to [H_{\bar s}/H_{\bar s}]$ are the
morphisms induced by $1$ and $h$, respectively. Then we can take $u$ to be
the morphism
\[(a,(\alpha(A),\alpha(A),1))\xrightarrow{v} (bg',(\alpha(A),\alpha(A),1))\xrightarrow{w} (b,(\beta(A),\beta(A),1)).\]
in $\cB$, where $v$ is given by the diagram $1\from e\to h$ in $P_{H_{\bar
s}}$ via the $H$-equivariant morphism $H_{\bar s}\to X^A$ carrying $1$ to
$a$ (and carrying $h$ to $bg'$), and $w$ is the obvious morphism of
$\cC_{X,G}$ above the morphism $(g'^{-1},g')\colon
(\alpha(A),\alpha(A),1)\to (\beta(A),\beta(A),1)$ of $\cA_G(k)^\natural$.

\emph{Step 3. If $L(f,\sigma,\tau)$ holds for a triple $(f,\sigma,\tau)$,
then $L(f)$ holds.} Indeed, if $\sigma'$ and $\tau'$ are rational liftings
of $\bar\xi$ and $\bar\eta$, respectively, then, by Step 2,
$L(\id_\xi,\sigma',\sigma)$ and $L(\id_\eta,\tau,\tau')$ hold, so
$L(f,\sigma',\tau')$ holds because
$(f,\sigma',\tau')=(\id_\eta,\tau,\tau')(f,\sigma,\tau)(\id_\xi,\sigma',\sigma)$.

\emph{Step 4. $L(f_3)$ holds.} Indeed, a rational lifting
$\tau=(b,\beta,\epsilon)$ of $\bar\eta$ induces a rational lifting of
$\overline{\eta'}$, and with respect to these liftings we can take $u$ to be
the morphism in $\cB$ induced by the diagram in $\cC_{X,G}$
\[(b,(A,A,1))\from (b,(A,B,1)) \to (b,(B,B,1))\]
above the diagram in $\cA_G(k)^\natural$
\[(A,A,1)\xleftarrow{(\id_A,\gamma)} (A,B,1)\xrightarrow{(\gamma,\id_A)} (B,B,1).\]

\emph{Step 5. $L(f_2)$ holds.} By Proposition \ref{p.7.16} (c), $\eta'$ can
be lifted to a morphism of groupoids $(b,\alpha)$, where $b\colon T\to X$,
and $\alpha\colon T\times A\to G$ is a crossed homomorphism, which restricts
to a homomorphism $T^A \times A\to G$, corresponding to a morphism, denoted
by $\alpha \res T^A$, from the (strictly local) scheme $T^A$ to the scheme
$\cHom(A,G)$ of group homomorphisms from $A$ to $G$ (Section~\ref{s.3}). We
will first show that, up to replacing $T^A$ by a finite radicial extension,
$\alpha \res T^A$ is conjugate to a \emph{$k$-rational} point of
$\cHom(A,G)$. For this, recall (Corollary \ref{c.Serre0}) that the orbits of
$G$ acting by conjugation on $\cHom(A,G)$ form a finite cover by open and
closed subschemes. Let $C \subset \cHom(A,G)$ be the orbit containing the
image of $\alpha \res T^A$. Choose a $k$-rational point $\alpha'\in
\cHom(A,G)(k)$ of $C$. Then the homomorphism $g \mapsto c_g(\alpha') =
g^{-1}\alpha'g$ from $G$ onto $C$ factors through an isomorphism
\[
H\backslash G \simto C,
\]
for a subgroup $H$ of $G$. Let $T'$ be defined by the cartesian square
\[
\xymatrix{T' \ar[r] \ar[d] & H_{\mathrm{red}} \backslash G \ar[d] \\
T^A \ar[r]^{\alpha \res T^A} & C.}
\]
Since the projection $G \to H_{\mathrm{red}} \backslash G$ is smooth, the
upper horizontal arrow can be lifted to a morphism $g \colon T' \to G$. Then
$c_{g^{-1}}(\alpha\res T') \colon T' \to \cHom(A,G)$ is the constant map of
value $\alpha'$. Let $\pi\colon T'\to T^A\hookrightarrow T$ be the
composite. We obtain a lifting $(b\pi,\alpha')\colon (T',A)\to (X,G)$ of
$[T'/A]\to [X/G]$, which induces rational liftings of $\overline{\eta'}$ and
$\overline{\eta''}$. With respect to these liftings, we can take $u$ to be
the morphism $\bar s\to \bar t$ in $P_{X^A}$ above $(A,A,1)$.
\end{proof}

\begin{remark}\label{r.small}
The categories $\cC_{X,G}$ and hence $N^{-1}\cC_{X,G}$ are essentially
small. It follows from (a) and (b) in the proof of Proposition \ref{p.RR}
that $\cC_{[X/G]}$ is essentially small.
\end{remark}

\section{K\"unneth formulas}\label{s.Kunneth}

The main results of this section are the K\"unneth formulas of Propositions
\ref{p.ss} and \ref{p.Kunneth}. One may hope for more general formulas
involving derived categories of modules over derived rings. We will not
tackle this question. Instead, we use an elementary approach, based on
module structures on spectral sequences, described in Construction
\ref{s.ssring} and Lemma \ref{l.ss}.

\begin{construction}\label{s.ssring}
Let $(\cC,T)$ be an additive category with translation. For objects $M$ and
$N$ in~$\cC$, the extended homomorphism group is the graded abelian group
$\Hom^*(M,N)$ with $\Hom^n(M,N)=\Hom(M,T^n N)$. The extended endomorphism
ring $\End^*(M)=\Hom^*(M,M)$ is a graded ring and $\Hom^*(M,N)$ is a
$(\End^*(N),\End^*(M))$-bimodule. Let $A^*$ be a graded ring.  A left
$A^*$-module structure on an object $M$ of $\cC$ is by definition a
homomorphism $\lambda_M \colon A^*\to \End^*(M)$ of graded rings. More
precisely, such a structure is given by morphisms $\lambda_a\colon M\to T^n
M$, $a\in A^n$, $n\in \Z$ such that $\lambda_{a+b}=\lambda_a+\lambda_b$ for
$a,b\in A^n$ and the diagram
  \[\xymatrix{M\ar[r]^{\lambda_b}\ar[rd]_{\lambda_{ab}} &T^nM\ar[d]^{T^n \lambda_a}\\
  & T^{m+n} M}\]
commutes for $a\in A^m$, $b\in A^n$. A morphism $M\to M'$ in $\cC$, with $M$
and $M'$ endowed with $A^*$-module structures, is said to preserve the
$A^*$-module structures if it commutes with all $\lambda_a$, $a\in A^n$,
$n\in \Z$. Let $B^*$ be a graded right $A^*$-module. A morphism $B^*
\otimes_{A^*}M\to N$ is by definition a homomorphism $B^*\to \Hom^*(M,N)$ of
right $A^*$-modules. More precisely, it is given by a family of morphisms
$f_b\colon M\to T^nN$, $b\in B^n$, $n\in \Z$ in $\cC$ such that
$f_{b+c}=f_b+f_c$ for $b,c\in B^n$ and the diagram
  \[\xymatrix{M\ar[r]^{\lambda_a}\ar[rd]_{f_{ba}} & T^mM\ar[d]^{T^m f_b}\\
  & T^{m+n} N}\]
commutes for $a\in A^m$, $b\in B^n$. We thus get a functor $N\mapsto
\Hom(B^*\otimes_{A^*} M,N)$ from $\cC$ to the category of abelian groups,
contravariant in $M$. In the category of graded abelian groups with
translation given by shifting, the notion of left $A^*$-module coincides
with the usual notion of graded left $A^*$-module and the above functor is
represented by the usual tensor product. Let $F\colon (\cC,T)\to (\cC',T)$
be a functor of additive categories with translation \cite[Definition 10.1.1
(ii)]{KashiwaraSCat}. A left $A^*$-module structure on $M$ induces a left
$A^*$-module structure on $FM$ and a morphism $B^*\otimes_{A^*}M\to N$
induces a morphism $B^*\otimes_{A^*} FM\to FN$.

Let $\cD$ be a triangulated category, and let $\cA$ be an abelian category.
We consider the additive categories of spectral objects $\SpOb(\cD)$,
$\SpOb(\cA)$ of type $\tilde \Z$ \cite[II 4.1.2, 4.1.4, 4.1.6]{Verdier}.
Here $\tilde \Z$ is the category associated to the ordered set $\Z\cup \{\pm
\infty\}$. For $m\in \Z$, $(X,\delta)\in \SpOb(\cD)$, $(H,\delta)\in
\SpOb(\cA)$, we put
\[(X,\delta)[m]=(X[m], (-1)^m\delta[m]),\quad (H^n,\delta^n)_n[m]=(H^{n+m},(-1)^m\delta^{n+m})_n.\]
For $a\in \Z\cup\{\infty\}$, let $\SpSeq_a(\cA)$ be the category of spectral
sequences $E_a\Rightarrow H$ in $\cA$. We define
\[(E_a^{pq}\Rightarrow H^n)[m] = (E_a^{p+m,q}\Rightarrow H^{n+m})\]
by multiplying all $d_r$ by $(-1)^m$. We endow $\SpOb(\cD)$, $\SpOb(\cA)$
and $\SpSeq_a(\cA)$ with the translation functor $[1]$. The resulting
categories with translation are covariant in $\cD$ and $\cA$ for exact
functors. If $H\colon \cD\to \cA$ is a cohomological functor, the induced
functor $\SpOb(\cD)\to \SpOb(\cA)$ commutes with translation. For $b\ge a$,
the restriction functor $\SpSeq_a(\cA)\to \SpSeq_b(\cA)$ commutes with
translation. Using the notation of \cite[II (4.3.3.2)]{Verdier}, we obtain a
functor $\SpOb(\cA)\to \SpSeq_2(\cA)$, which also commutes with translation.
A left $A^*$-module structure on an object of $\SpSeq_a(\cA)$ induces left
$A^*$-module structures on $H^*$ and $E_r^{*q}$ for all $q\in \Z$ and $r\in
[a,\infty]$. If we put $G_qH^n=F^{n-q} H^n$, so that the abutment is of the
form $E_\infty^{pq}\simto\gr^G_q H^{p+q}$, then $G_q$ preserves the
$A^*$-module structure. The differentials $d_r^{*q}\colon E_r^{*q}\to
E_r^{*+r,q-r+1}$ and the abutment $E_\infty^{*q}\simto \gr^G_q H^*$ are
$A^*$-linear. A morphism $B^*\otimes_{A^*}(E_a\Rightarrow H) \to
(E'_a\Rightarrow H')$ induces morphisms on $E_r^{*q}$, $H^*$, $G_qH^*$,
$\gr^G_q H^*$, compatible with $d_r$, abutment, the projection $G_q\to
\gr^G_q$ and the inclusions $G_{q-1}H^*\to G_{q}H^*\to H^*$.
\end{construction}

\begin{lemma}\label{l.ss}
Let $H$ and $H'$ be filtered graded abelian groups, $H$ endowed with a left
$A^*$-module structure. We let $G$ denote the (increasing) filtrations.
Assume that $G_qH^*=G_qH'^*=0$ for $q$ small enough and $H^n=\bigcup_{q\in
\Z} G_q H^n$, $H'^n=\bigcup_{q\in \Z}G_q H'^n$ for all $n$. Let
$B^*\otimes_{A^*}H \to H'$ be a morphism such that the homomorphism
$B^*\otimes_{A^*}\gr^G_q H^*\to \gr^G_qH'^{*}$ is an isomorphism for all
$q$. Then the homomorphism $B^*\otimes_{A^*}H^*\to {H'}^*$ is an
isomorphism.
\end{lemma}

\begin{proof}
Since $G_qH^*=G_q{H'}^*=0$ for $q$ small enough, one shows by induction that
the morphism of exact sequences
  \[\xymatrix{& B^*\otimes_{A^*}G_{q-1} H^*\ar[r]\ar[d] & B^*\otimes_{A^*}G_q H^* \ar[r]\ar[d] & B^*\otimes_{A^*}\gr^G_q H^*\ar[r]\ar[d]^{\simeq} & 0\\
  0\ar[r] & G_{q-1} {H'}^* \ar[r] & G_q {H'}^* \ar[r] & \gr^G_q {H'}^* \ar[r] & 0}\]
is an isomorphism. Then we apply the hypotheses $\varinjlim_{q\in \Z}G_q H^*
= H^*$, $\varinjlim_{q\in \Z}G_q {H'}^* = {H'}^*$ and the fact that tensor
product commutes with colimits.
\end{proof}

\begin{construction}\label{s.ttopos}
Let
\begin{equation}\label{e.ttopos}
  \xymatrix{X'\ar[r]^{h} \ar[d]_{f'} &X\ar[d]^f\\
  Y'\ar[r]^g &Y}
\end{equation}
be a 2-commutative square of commutatively ringed topoi, $K\in D(\cO_{Y'})$,
$L\in D(\cO_{X})$. An element $s\in H^m(Y',K)$ corresponds to a morphism
$\cO_{Y'}\to K[m]$ in $D(\cO_{Y'})$, and an element $t\in H^n(X,L)$
corresponds to a morphism $\cO_{X}\to L[n]$ in $D(\cO_{X})$. Then
  \[L{f'}^*s\otimes^L_{\cO_{X'}} Lh^* t\colon \cO_{X'}\to L{f'}^*K\otimes^L_{\cO_{X'}} L{h}^*L\]
is a morphism in $D(\cO_{X'})$. This defines a graded map
  \[H^*(Y',K)\times H^*(X,L)\to H^*(X',L{f'}^*K\otimes^L_{\cO_{X'}} L{h}^*L)\]
which is $H^*(Y,\cO_{Y})$-bilinear, hence induces a homomorphism
\begin{equation}\label{e.tensor}
  H^*(Y',K)\otimes_{H^*(Y,\cO_Y)}H^*(X,L)\to H^*(X',L{f'}^*K\otimes^L_{\cO_{X'}} L{h}^*L),
\end{equation}
which is a homomorphism of $(H^*(Y',\cO_{Y'}), H^*(X,\cO_X))$-bimodules.
\end{construction}

\begin{construction}\label{c.ttopos}
Let $f\colon X\to Y$ be a morphism of commutatively ringed topoi, and let
$L\in D(\cO_Y)$, $K\in D(\cO_X)$. We consider the second spectral object
$(L,\delta)$ associated to $L$ \cite[III 4.3.1, 4.3.4]{Verdier}, with
$L(p,q)=\tau^{[p,q-1]} L$. For $s\in H^n(Y,\cO_Y)$ corresponding to
$\cO_Y\to \cO_Y[n]$, the functor $s\otimes^L_{\cO_Y}-$ induces a morphism of
spectral objects $(L,\delta)\to (L,\delta)[n]$. This endows $(L,\delta)$
with a structure of $H^*(Y,\cO_Y)$-module (Construction \ref{s.ssring}). For
$t\in H^n(X,K)$ corresponding to $\cO_X\to K[n]$, the functor
$t\otimes^L_{\cO_X}-$ induces a morphism of spectral objects
$Lf^*(L,\delta)\to K\otimes^L_{\cO_X}Lf^*(L,\delta)[n]$. This defines a
morphism
\[H^*(X,K)\otimes_{H^*(Y,\cO_Y)}Lf^*(L,\delta) \to K\otimes^L_{\cO_X}Lf^*(L,\delta).\]
Applying $Rf_*$ and composing with the adjunction $\id_{D(\cO_Y)} \to
Rf_*Lf^*$, we get a morphism
\[H^*(X,K)\otimes_{H^*(Y,\cO_Y)}(L,\delta)\to Rf_*(K\otimes^L_{\cO_X}Lf^*(L,\delta)).\]
Further applying the cohomological functor $H^0(Y,-)$, we obtain a morphism
\[H^*(X,K)\otimes_{H^*(Y,\cO_Y)}(E_2\Rightarrow H)\to (E'_2\Rightarrow
H'),\] where the two spectral sequences are
\begin{gather}
  E_2^{pq}=H^p(Y,\cH^q L)\Rightarrow H^{p+q}(Y,L),\\
  {E'}_2^{pq}=H^p(X,K\otimes^L_{\cO_X} Lf^* \cH^q L) \Rightarrow H^{p+q}(X,K\otimes_{\cO_X}^L Lf^*L).\label{e.secss}
\end{gather}
By construction, the induced morphisms on $E_2^{*q}$ and on $H^*$ coincide
with \eqref{e.tensor} for \eqref{e.ttopos} given by $\id_f$.
\end{construction}

The results of Constructions \ref{s.ttopos} and \ref{c.ttopos} have obvious
analogues for Artin stacks and complexes in $D_\cart(-,\Lambda)$, where
$\Lambda$ is a commutative ring.

\begin{prop}\label{p.ss}
Let
  \[\xymatrix{\cX'\ar[r]^{h} \ar[d]_{f'} &\cX\ar[d]^f\\
  \cY'\ar[r]^g &\cY}\]
be a 2-commutative square of Artin stacks. Let $K\in
D^+_\cart(\cY',\F_\ell)$ and $L\in D^+_\cart(\cX,\F_\ell)$. Suppose that
\begin{enumerate}
  \item The Leray spectral sequence for $(f,L)$
  \begin{equation}\label{e.ss1}
  E_2^{pq}= H^p(\cY,R^q f_* L) \Rightarrow H^{p+q}(\cX,L)
  \end{equation}
  degenerates at $E_2$.

  \item For every $q$, $R^q f_* L$ is a constant constructible
      $\F_\ell$-module on $\cY$.

  \item The base change morphism $\BC\colon g^* Rf_* L\to Rf'_* h^* L$ is
      an isomorphism.

  \item The morphism $\PF_{f'}\colon Rg_*(K\otimes Rf'_*h^*L)\to
      Rg_*Rf'_*(f'^*K\otimes h^* L)$ deduced from the projection formula
      morphism $K\otimes Rf'_*h^*L\to Rf'_*(f'^*K\otimes h^* L)$ is an
      isomorphism.
\end{enumerate}
  Then the spectral sequence (of type \eqref{e.secss})
  \begin{equation}\label{e.ss2}
  E_2^{pq}=H^p(\cY',K\otimes R^qf'_*{h}^* L )\Rightarrow H^{p+q}(\cY',K\otimes Rf'_*{h}^*L)
  \end{equation}
  degenerates at $E_2$ and the homomorphism \eqref{e.tensor}
  \begin{equation}\label{e.sstensor}
  H^*(\cY',K)\otimes_{H^*(\cY,\F_\ell)}H^*(\cX,L)\to H^*(\cX',{f'}^*K\otimes {h}^*L)
  \end{equation}
  is an isomorphism.
\end{prop}

\begin{proof}
  Take any geometric point $t\to \cY'$.
  By (b), the $E_2$-term of \eqref{e.ss1} is
  \[E_2^{pq}=H^p(\cY,R^q f_* L)\simeq H^p(\cY,\F_\ell)\otimes (R^q f_* L)_t.\]
  By (c), \eqref{e.ss2} is isomorphic to
  \begin{equation}\label{e.ss3}
    {E'}_2^{pq}=H^p(\cY',K\otimes g^*R^q f_* L)\Rightarrow H^{p+q}(\cY',K\otimes g^* Rf_* L).
  \end{equation}
  By (b), ${E'}_2^{pq}\simeq H^p(\cY',K)\otimes (R^q f_* L)_t$. Thus the morphism \[H^*(\cY',K)\otimes_{H^*(\cY,\F_\ell)}E_2^{*q}\to{E'}_2^{*q}\]
  is an isomorphism. Il then follows from (a) and Lemma \ref{l.ss} that \eqref{e.ss3} degenerates at $E_2$ and the homomorphism
  \begin{equation}\label{e.ss4}
  H^*(\cY',K)\otimes_{H^*(\cY,\F_\ell)}H^*(\cX,L)\to H^*(\cY',K\otimes g^* Rf_* L)
  \end{equation}
  is an isomorphism. Thus \eqref{e.ss2} degenerates at $E_2$ and \eqref{e.sstensor} is an isomorphism since it is the composition of \eqref{e.ss4} with the morphism induced by the composition
  \[Rg_* (K\otimes g^*Rf_* L)\xrightarrow[\sim]{Rg_*(\id_K\otimes \BC)} Rg_*(K\otimes Rf'_* h^* L) \xrightarrow[\sim]{\PF_{f'}} Rg_* Rf'_*({f'}^* K \otimes h^* L),\]
  of the isomorphisms in (c) and (d).
\end{proof}

In the rest of this section, let $k$ be a separably closed field of
characteristic $\neq \ell$.

\begin{prop}\label{p.Kunneth}
Let $G$ be a connected algebraic group over $k$, and let $X$ be an algebraic
space of finite presentation over $k$ endowed with an action of $G$. Let
  \[\xymatrix{\cX'\ar[r]^{h} \ar[d]_{f'} &[X/G]\ar[d]^f\\
  \cY'\ar[r]^g &BG}\]
be a 2-cartesian square of quasi-compact, quasi-separated Artin stacks,
where $f$ is the canonical projection. Let $K\in D^+_\cart(\cY',\F_\ell)$.
Suppose that
  the map $e\colon H^*([X/G])\to H^*(X)$ induced by the projection $X\to [X/G]$ is surjective.
  Then $H^*([X/G])$ is a finitely generated free $H^*(BG)$-module, the spectral sequence
\[E_2^{pq}=H^p(\cY',K\otimes R^qf'_* \F_\ell)\Rightarrow H^{p+q}(\cY',K\otimes Rf'_*\F_\ell)\]
degenerates at $E_2$, and the homomorphism
  \[H^*(\cY',K)\otimes_{H^*(BG)}H^*([X/G]) \to H^*(\cX',{f'}^* K)\]
is an isomorphism.
\end{prop}

\begin{proof}
For the second and the third assertions, we apply Proposition \ref{p.ss}. By
Corollary \ref{l.BG} and generic base change (Remark \ref{r.gbc}),
conditions (b) and (c) of Proposition \ref{p.ss} are satisfied. For $L\in
D^+_\cart([X/G],\F_\ell)$, the diagram
\[\xymatrix{Rg_*K\otimes Rf_* L \ar[r]^{\PF_{g}}\ar[rd]_{\PF_f} & Rg_*(K\otimes g^*Rf_*L)\ar[r]^{\BC} & Rg_*(K\otimes Rf'_*h^*L)
\ar[r]^{\PF_{f'}} & R(gf')_*(f'^*K\otimes h^* L)\ar[d]^\simeq\\
& Rf_*(f^*Rg_*K\otimes L)\ar[r]^{\BC'} & Rf_*(Rh_*f'^* K\otimes L)\ar[r]^{\PF_{h}} & R(fh)_* (f'^*K\otimes h^* L)}
\]
commutes. Take $L=\F_\ell$. Then generic base change (Remark \ref{r.gbc})
and Proposition \ref{l.pf} (d) imply that $\BC$, $\BC'$, $\PF_f$, $\PF_g$,
$\PF_h$ are isomorphisms, hence $\PF_{f'}$ is an isomorphism as well, which
proves condition (d) of Proposition \ref{p.ss}. Next we check condition (a)
of Proposition \ref{p.ss}. The Leray spectral sequence for $f$ is
  \begin{equation}\label{e.Kunnethss}
  E_2^{pq}=H^p(BG,R^qf_*\F_\ell)\Rightarrow H^{p+q}([X/G]).
  \end{equation}
Since $R^qf_*\F_\ell$ is constant of value $H^q(X)$, we have $E_2^{pq}\simeq
H^p(BG)\otimes H^q(X)$. As $e$ is an edge homomorphism for
\eqref{e.Kunnethss}, its surjectivity implies $d_r^{0q}=0$ for all $r\ge 2$.
It then follows from the $H^*(BG)$-module structure of \eqref{e.Kunnethss}
that it degenerates at $E_2$. The first assertion of Proposition
\ref{p.Kunneth} then follows from the fact that $H^*(X)$ is a
finite-dimensional vector space.
\end{proof}

\begin{prop}\label{p.surj}
  Let $G=\GL_{n,k}$, $T$ be a maximal torus of $G$, $A=\Ker (-^\ell\colon T \to T)$.
  Then the map $H^*(BA,\F_\ell)\to H^*(G/A,\F_\ell)$ induced by the projection $G/A\to BA$ is surjective.
\end{prop}

\begin{proof}
Let us recall the proof on \cite[page~566]{Quillen1}. Consider the following
diagram with 2-cartesian squares (Proposition \ref{p.cart}):
  \[\xymatrix{G/A\ar[r]\ar[d] & G/T\ar[d]\ar[r] & \Spec k\ar[d]\\
  BA\ar[r] & BT\ar[r] & BG.}\]
Note that the arrow $BA\to BT$ can be identified with the composition
$BA\simto[X/T] \to BT$, where $X=A\backslash T$, and the first morphism is
an isomorphism by Corollary \ref{p.quot}. The map $H^*(BA)\to H^*(X)$
induced by the projection $\pi\colon X=A\backslash T\to BA$ is surjective.
Indeed, using K\"unneth formula this reduces to the case where $T$ has
dimension $1$, which follows from Lemma \ref{l.H1} below. Note that $\pi$
can be identified with the composition $X\to [X/T]\simeq BA$. Thus, by
Proposition \ref{p.Kunneth} applied to $f\colon [X/T]\to BT$, the map
  \[H^*(G/T)\otimes_{H^*(BT)}H^*(BA)\to H^*(G/A)\]
is an isomorphism. We conclude by applying the fact that $H^*(BT)\to
H^*(G/T)$ is surjective (Theorem \ref{l.finite}).
\end{proof}

\begin{lemma}\label{l.H1}
Let $A$ be an elementary abelian $\ell$-group, and let $X$ be a connected
algebraic space endowed with an $A$-action such that $X$ is the maximal
connected Galois \'etale cover of $[X/A]$ whose group is an elementary
abelian $\ell$-group. Then the homomorphism
  \begin{equation}\label{e.H1}
  H^1(BA,\F_\ell)\to H^1([X/A],\F_\ell)
  \end{equation}
induced by the projection $[X/A]\to BA$ is an isomorphism.
\end{lemma}

\begin{proof}
For any connected Deligne-Mumford stack $\cX$, $H^1(\cX,\F_\ell)$ is
canonically identified with $\Hom(\pi_1(\cX),\F_\ell)$, and \eqref{e.H1} is
induced by the morphism
  \[\pi_1([X/A])\to \pi_1(BA)\simeq A.\]
The assumption means that $A$ is the maximal elementary abelian
$\ell$-quotient of $\pi_1([X/A])$.
\end{proof}

\begin{prop}\label{p.absurj}
Let $X$ be an abelian variety over $k$, $A=X[\ell]=\Ker(\ell\colon X\to X)$.
Then the map $H^*(BA,\F_\ell)\to H^*(X/A,\F_\ell)$ induced by the projection
$X/A\to BA$ is surjective.
\end{prop}

\begin{proof}
We apply Lemma \ref{l.H1} to the morphism $\ell\colon X\to X$, which
identifies the target with $X/A$. By Serre-Lang's theorem \cite[XI
Th\'eor\`eme 2.1]{SGA1}, this morphism is the maximal \'etale Galois cover
of $X$ by an elementary abelian $\ell$-group. Thus $H^1(BA)\to H^1(X/A)$ is
an isomorphism. It then suffices to apply the fact that $H^*(X/A)$ is the
exterior algebra of $H^1(X/A)$.
\end{proof}

\section{Proof of the structure theorem}\label{s.7}
We proceed in several steps:
\begin{enumerate}[(1)]
\item  We first prove Theorem \ref{p.str} (b) when $\mathcal{X}$ is a
    Deligne-Mumford stack with finite inertia, and whose inertia groups
    are elementary abelian $\ell$-groups.

\item We prove Theorem \ref{p.str} (b) for $\mathcal{X}$ a quotient stack
    $[X/G]$.

\item For certain quotient stacks $[X/G]$ we establish estimates for the
    powers of $F$ annihilating the kernel and the cokernel of $a_G(X,K)$
    \eqref{e.aEstar}.

\item Using (3), we prove Theorem \ref{p.str} (b) for Deligne-Mumford
    stacks with finite inertia.

\item We prove Theorem \ref{p.str} (a) and the first assertion of (b) for
    Artin stacks having a stratification by global quotients.
\end{enumerate}

\begin{construction}\label{s.c2}
Let $f\colon X\to Y$ be a morphism of commutatively ringed topoi such that
$\ell \cO_Y=0$, $K\in D(X)$. The Leray spectral sequence of~$f$,
  \[E_2^{ij}= H^i(Y,R^jf_*K) \Rightarrow H^{i+j}(X,K),\]
gives rise to an edge homomorphism
\begin{equation}
  e_{f,K}\colon H^*(X,K) \to H^0(Y,R^*f_*K),
\end{equation}
which is a homomorphism of $\F_\ell$-(pseudo-)algebras if $K\in D(X)$ is a
(pseudo-)ring. The following crucial lemma is similar to Quillen's result
\cite[Proposition 3.2]{Quillen1}.
\end{construction}

\begin{lemma}\label{l.10.2}
Let $K$ be a pseudo-ring in $D(X)$. Assume that $c=\cd(Y)< \infty$. Then
$(\Ker e_{f,K})^{c+1}=0$. Moreover, if $K$ is commutative, then $e_{f,K}$ is
a uniform $F$-isomorphism; more precisely, for $b\in E_2^{0,*}$, we have
$b^{\ell^{n}} \in \Img e_{f,K}$, where $n=\max\{c-1,0\}$.
\end{lemma}

\begin{proof}
We imitate the proof of \cite[Proposition 3.2]{Quillen1} (for the case of
finite cohomological dimension). We have $E^{ij}_2 = 0$ for $i > c$.
Consider the multiplicative structure on the spectral sequence (Example
\ref{e.ssmult}). As $\Ker e_{f,K} = F^1H^*(X,K)$, where $F^{\bullet}$
denotes the filtration on the abutment, $(\Ker e_{f,K})^{c+1}\subset
F^{c+1}H^*(X,K) = 0$. If $K$ is commutative and $b \in E^{0,*}_r$, then the
formula $d_r(b^\ell) = \ell b^{\ell-1}d_r(b)=0$ implies that $b^\ell\in
E_{r+1}^{0,*}$. Thus for $b\in E_2^{0,*}$, $b^{\ell^n} \in
E^{0,*}_{2+n}=E^{0,*}_{\infty}=\Img e_{f,K}$.
\end{proof}

\begin{construction}
Let $\cX$ be a Deligne-Mumford stack of finite presentation and finite
inertia over $k$. By Keel-Mori's theorem \cite{KeelMori} (see \cite[Theorem
6.12]{Rydh} for a generalization), there exists a coarse moduli space
morphism
\[f\colon \cX\to Y,\]
which is proper and quasi-finite. Let $K\in D^+_\cart(\cX,\F_\ell)$. Then
Construction \ref{s.c2} and Lemma \ref{l.10.2} apply to $f$ and $K$ with
$\cd_\ell(Y) \le 2\dim(Y)$.

For any geometric point $t$ of $Y$, consider the following diagram of Artin
stacks with 2-cartesian squares:
\[\xymatrix{\cX_t\ar[r]\ar[d] &\cX_{(t)}\ar[r]\ar[d]& \cX\ar[d]^{f}\\
t\ar[r] & Y_{(t)} \ar[r] & Y.}
\]
We have canonical isomorphisms
\begin{equation}\label{e.10.3.1}
(R^qf_*K)_t \simto H^q(\mathcal{X}_{(t)},K)  \simto H^q(\mathcal{X}_t,K),
\end{equation}
the second one by the proper base change theorem (cf.\ \cite[Theorem
9.14]{OlssonSh}). Therefore, if we let $P_Y$ denote the category of
geometric points of $Y$ (Definition \ref{s.point}), the map
\begin{equation}\label{e.10.3.2}
H^0(Y,R^qf_*K) \to \varprojlim_{t \in P_Y}
H^q(\mathcal{X}_{(t)},K) \simto \varprojlim_{t \in P_Y} H^q(\mathcal{X}_t,K),
\end{equation}
is an isomorphism if $K\in D^+_c(\cX,\F_\ell)$, by Proposition \ref{p.sp}.
On the other hand, recall \eqref{e.R} that
\[
R^q(\mathcal{X},K) = \varprojlim_{(x \colon \mathcal{S} \to \mathcal{X}) \in \mathcal{C}_{\mathcal{X}}} H^q(\mathcal{S},K_x) = \Gamma(\widehat{\mathcal{C}_{\mathcal{X}}},H^q(K_\bullet)),
\]
where $K_x=x^* K$ and $H^q(K_\bullet)$ denotes the presheaf on
$\mathcal{C}_{\mathcal{X}}$ whose value at $x$ is $H^q(\mathcal{S},K_x)$. We
define a category $\cC_f$ and functors
\[\xymatrix{&\cC_f\ar[rd]^{\psi}\ar[ld]_{\varphi}\\
\cC_\cX & & P_Y}
\]
as follows. The category $\cC_f$ is cofibered over $P_Y$ by $\psi$. The
fiber category of $\psi$ at a geometric point $t\to Y$ is
$\mathcal{C}_{\mathcal{X}_{(t)}}$. The pushout functor
$\mathcal{C}_{\mathcal{X}_{(t)}} \to \mathcal{C}_{\mathcal{X}_{(z)}}$ for a
morphism of geometric points $t \to z$ is induced by the morphism
$\cX_{(t)}\to \cX_{(z)}$ (Remark \ref{r.Cfunc}). The functors
$\varphi_t\colon \cC_{\cX_{(t)}} \to \cC_{\cX}$ induced by the morphisms
$\mathcal{X}_{(t)} \to \mathcal{X}$ define $\varphi$. Thus we have an
inverse image map
\begin{equation}
\varphi^* \colon R^q(\mathcal{X},K) \to \Gamma(\widehat{\mathcal{C}_f},\varphi^*H^q(K_{\bullet})).
\end{equation}
By Lemma \ref{l.cofcof} we have
\[\psi_*\varphi^*H^q(K_{\bullet})_t\simeq \Gamma(\widehat{\cC_{\cX_{(t)}}},\varphi_t^*H^q(K_\bullet)).\]
Thus we have
\begin{equation}\label{e.10.3.7}
\Gamma(\widehat{\mathcal{C}_f},\varphi^*H^q(K_{\bullet})) \simeq \Gamma(\widehat{P_Y},\psi_*\varphi^*H^q(K_{\bullet}))\simto \varprojlim_{t \in P_Y} \varprojlim_{(x \colon \mathcal{S} \to \mathcal{X}_{(t)}) \in \mathcal{C}_{\mathcal{X}_{(t)}}}H^q(\mathcal{S},K_x)
\end{equation}
\end{construction}

\begin{prop}\label{p.10.4}\leavevmode
\begin{enumerate}
\item The following diagram commutes
\[
\xymatrix{H^q(\mathcal{X},K) \ar[r]^-{e^q_{f,K}} \ar[d]_{a^q_{\mathcal{X},K}} & H^0(Y,R^qf_*K) \ar[r]^-{\eqref{e.10.3.2}}  & \varprojlim_{t \in P_Y} H^q(\mathcal{X}_{(t)},K) \ar[d]^{{\varprojlim_{t \in P_Y}} a^q_{\mathcal{X}_{(t)},K}} \\
 R^q(\mathcal{X},K) \ar[r]^-{\varphi^*} &\Gamma(\widehat{\mathcal{C}_f}, \varphi^*H^q(K_{\bullet})) \ar[r]^-{\eqref{e.10.3.7}}_-\sim & \varprojlim_{t \in P_Y} \varprojlim_{(x \colon \mathcal{S} \to \mathcal{X}_{(t)}) \in \mathcal{C}_{\mathcal{X}_{(t)}}}H^q(\mathcal{S},K_x).}
\]

\item $\varphi^*$ is an isomorphism.

\item Consider the commutative square
\[\xymatrix{H^q(\cX_{(t)},K)\ar[d]_{a^q_{\cX_{(t)},K}}\ar[r]^\sim &H^q(\cX_t,K)\ar[d]^{a^q_{\cX_t,K}}\\
\varprojlim_{\cS\in\cC_{\cX_{(t)}}}
H^q(\cS,K)\ar[r]^{\iota^*}& \varprojlim_{\cS\in \cC_{\cX_t}} H^q(\cS,K)}
\]
defined by the functor $\iota\colon \cC_{\cX_t}\to \cC_{\cX_{(t)}}$
induced by the inclusion $\cX_t\to \cX_\cT$, in which the upper horizontal
map is the second isomorphism of \eqref{e.10.3.1}. The map $\iota^*$ is an
isomorphism.
\end{enumerate}
\end{prop}

\begin{proof}
Assertion (a) follows from the definitions. For (b) it suffices to show that
$\varphi$ is cofinal. Let $\tau \colon \mathcal{C}_{\mathcal{X}} \to
\mathcal{C}_f$ be the functor carrying an $\ell$-elementary point $x \colon
[S/A] \to \mathcal{X}$, with $s$ the closed point of $S$, to the induced
$\ell$-elementary point $\tau(x) \colon [S/A] \to \mathcal{X}_{(f(s))}$.
Then we have $\varphi\tau \simeq \mathrm{id}_{\mathcal{C}_{\mathcal{X}}}$,
and a canonical natural transformation $\tau\varphi \to
\id_{\mathcal{C}_f}$, carrying an object $\xi\colon [S/A]\to \cX_{(t)}$ of
$\cC_f$ to the cocartesian morphism $\tau\varphi(\xi)\to \xi$ in $\cC_f$
above the morphism $f(s)\to t$ in $P_Y$. These exhibit $\tau$ as a left
adjoint to $\varphi$. Therefore, by Lemma \ref{l.adjcof} below, $\varphi$ is
cofinal. For (c), it suffices again to show that $\iota$ is cofinal. Let $X
\to \mathcal{X}_{(t)}$ be an \'etale atlas. As $f$ is quasi-finite, up to
replacing $X$ by a connected component, we may assume that $X$ is a strictly
local scheme, finite over $Y_{(t)}$. Then $\cX_{(t)}\simeq [X/G]$, where
$G=\Aut_{\cX_{(t)}}(x)$, $x$ is the closed point of $X$. Let $\xi \colon
[S/A] \to [X/G]$ be an $\ell$-elementary point of $[X/G]$. The
$\ell$-elementary point $[x/A]\to \cX_t$, endowed with the morphism in
$\cC_{[X/G]}$ given by the diagram
\[[S/A] \to [X/A] \leftarrow [x/A]\]
in $\mathcal{C}'_{[X/G]}$, defines an initial object of $(\xi \downarrow
\iota)$. Therefore, $\iota$ is cofinal.
\end{proof}

\begin{lemma}\label{l.adjcof}
Let $G \colon \mathcal{A} \to \mathcal{B}$ be a functor. If $G$ has a left
adjoint, then $G$ is cofinal.
\end{lemma}

\begin{proof}
Let $F\colon \cB\to \cA$ be a left adjoint to $G$. Then, for every object
$b$ of $\mathcal{B}$, $(Fb, b \to GFb)$ is an initial object of $(b
\downarrow G)$. Thus $(b\downarrow G)$ is connected.
\end{proof}

\begin{cor}\label{c.DMl}
The assertion of Theorem \ref{p.str} (b) holds if $\mathcal{X}$ is a
Deligne-Mumford stack with finite inertia, whose inertia groups are
elementary abelian $\ell$-groups. More precisely, if $c=\cd_\ell(Y)$, where
$Y$ is the coarse moduli space of $\cX$, then $(\Ker a_{\cX,K})^{c+1}=0$ and
for $K$ commutative and $b\in E_2^{0,*}$, we have $b^{\ell^{n}} \in \Img
a_{\cX,K}$, where $n=\max\{c-1,0\}$.
\end{cor}

\begin{proof}
It suffices to show that, for all $t \in P_Y$,
\[
a^q_{\mathcal{X}_t,K} \colon H^q(\mathcal{X}_t,K) \to \varprojlim_{(x \colon \mathcal{S} \to \mathcal{X}_t)\in \mathcal{C}_{\mathcal{X}_t}} H^q(\mathcal{S},K_x)
\]
is an isomorphism. Indeed, by Proposition \ref{p.10.4} (c) this will imply
that the right vertical arrow in the diagram of Proposition \ref{p.10.4} (a)
is an isomorphism. As \eqref{e.10.3.2} is an isomorphism, $\varphi^*$ is an
isomorphism (Proposition \ref{p.10.4} (b)), and $e_{f,K}= \bigoplus
e^q_{f,K}$ has nilpotent kernel and, if $K$ is commutative, is an
$F$-isomorphism (Lemma \ref{l.10.2}), it will follow that $a_{\mathcal{X},K}
= \bigoplus a^q_{\mathcal{X},K}$ has the same properties with the same
bounds for the exponents. As $f \colon \mathcal{X} \to Y$ is a coarse moduli
space morphism, there exists a finite radicial extension $t' \to t$ and a
geometric point $y'$ of $\mathcal{X}$ above $t'$ such that
$(\mathcal{X}_{t'})_\mathrm{red} \simeq B\mathrm{Aut}_{\mathcal{X}}(y')$.
Therefore we are reduced to showing that $a_{\mathcal{X},K}$ is an
isomorphism for $\mathcal{X} = BA_k$, where $A$ is an elementary abelian
$\ell$-group. In this case, $\id_{BA_k} \colon BA_k \to BA_k$ is a final
object of $\cC_{BA_k}$, so we can identify $R^q(BA_k,K)$ with $H^q(BA_k,K)$,
and $a^q_{BA_k,K}$ with the identity.
\end{proof}

\begin{cor}\label{c.quotl}
Suppose $\mathcal{X}=[X/G]$ is a global quotient stack (Definition
\ref{d.quot}), where the action of $G$ on $X$ satisfies the following two
properties:
\begin{enumerate}
\item The morphism $\gamma \colon G \times X \to X \times X$, $(g,x)
    \mapsto (x,xg)$ is finite and unramified.

\item All the inertia groups of $G$ are elementary abelian $\ell$-groups.
\end{enumerate}
Then the assertions of Corollary \ref{c.DMl} hold.
\end{cor}

\begin{proof}
As $\gamma$ in (a) can be identified with the morphism $X \times_{[X/G]} X
\to X \times X$, which is the pull-back of the diagonal morphism
$\Delta_{[X/G]} \colon [X/G] \to [X/G] \times [X/G]$ by $X \times X \to
[X/G] \times [X/G]$, (a) implies that $\Delta_{[X/G]}$ is finite and
unramified. In particular, $[X/G]$ is a Deligne-Mumford stack. Moreover, as
the inertia stack is the pull-back of $\Delta_{[X/G]}$ by $\Delta_{[X/G]}$,
$[X/G]$ has finite inertia. Taking (b) into account, we see that $[X/G]$
satisfies the assumptions of \ref{c.DMl}, and therefore \ref{p.str} (b)
holds for $[X/G]$.
\end{proof}

\begin{prop}\label{p.quotgen}
Theorem \ref{p.str} (b) for global quotient stacks $[X/G]$ (Definition
\ref{d.quot}) follows from Theorem \ref{p.str} (b) for $G$ linear.
\end{prop}

\begin{proof}
Consider the system of subgroups $G_i=L\cdot A[m\ell^i]\cdot F$ of $G=L\cdot
A\cdot F$ as in the proof of Theorem \ref{t.finite} (with $\Lambda=\F_\ell$
and $n=\ell$), where $m$ is the order of $F$. Note that every elementary
abelian $\ell$-subgroup of $A\cdot F$ is contained in $A[m\ell]\cdot F$. As
a consequence, every elementary abelian $\ell$-subgroup of $G$ is contained
in $G_1$, so that the restriction map $R^*_G(X,K)\to R^*_{G_i}(X,K)$ is an
isomorphism for $i\ge 1$. Consider the commutative diagram
\[\xymatrix{H^*([X/G],K)\ar[r]\ar[d]_{a_G(X,K)} & H^*([X/G_{4d}],K)\ar[d]^{a_{G_{4d}}(X,K)}\ar[r]^\alpha & H^*([X/G_{2d}],K)\ar[d]^{a_{G_{2d}}(X,K)}\\
R^*_G(X,K)\ar[r]^\sim & R^*_{G_{4d}}(X,K)\ar[r]^\sim & R^*_{G_{2d}}(X,K),}
\]
where $d=\dim A$. By Remark \ref{r.finite}, $H^*([X/G],K)$ is the image of
$\alpha$. Thus it suffices to show that $a_{G_{2d}}(X,K)$ has nilpotent
kernel and, if $K$ is commutative, $a_{G_{4d}}(X,K)$ is a uniform
$F$-surjection.
\end{proof}

\begin{prop}\label{p.quotstr}
Theorem \ref{p.str} (b) holds for global quotient stacks of the form
$[X/G]$, where $G$ is either a linear algebraic group, or an abelian
variety.
\end{prop}

\begin{proof}
Although by Proposition \ref{p.quotgen} it would suffice to treat the case
where $G$ is linear, we prefer to treat both cases simultaneously, in order
to later get better bounds for the power of $F$ annihilating the kernel and
the cokernel of the map $a_{\mathcal{X},K}$ (Corollary \ref{c.bound}). We
follow closely the arguments of Quillen for the proof of \cite[Theorem
6.2]{Quillen1}. If $G$ is linear, choose an embedding of $G$ into a linear
group $L = \GL_n$ over $k$ \cite[Corollaire II.2.3.4]{DG}, and a maximal
torus $T$ of $L$. If $G$ is an abelian variety, let $L=T=G$. In both cases,
denote by $S$ the kernel of $\ell \colon T \rightarrow T$, which is an
elementary abelian $\ell$-group of order $n$. We let $L$ act on $F =
S\backslash L$ by right multiplication. If $g \in L(k)$, and if $\{S\}$
denotes the rational point of $F$ defined by the coset $S$, the inertia
group of $L$ at $\{S\}g$ is $g^{-1}Sg$. Let us show that the diagonal action
of $G$ on $X \times F$ (resp.\ $X \times F \times F$) satisfies assumptions
(a) and (b) of Corollary \ref{c.quotl}. It suffices to show this for
$X\times F$. Consider the commutative square
\[\xymatrix{L\times L \ar[d]\ar[r]^\sim &L\times L\ar[d]\\
F\times L \ar[r] &F\times F}
\]
where the horizontal morphisms are the morphisms $\gamma \colon (x,g)\mapsto
(x,xg)$. As the vertical morphisms are finite and surjective, so is the
lower horizontal morphism. Moreover, the latter is unramified. Hence the
morphism $\gamma\colon F\times G\to F\times F$ is finite and unramified. The
same holds for the morphism $\gamma\colon (X\times F)\times G \to (X\times
F)\times (X\times F)$, $(x,y,g)\mapsto (x,y, xg,yg)$, because it is the
composite $X\times F\times G\to X\times X\times F\times G \to X\times
F\times X\times F$, where the first morphism $(x,y,g)\mapsto (x,xg,y,g)$ is
a closed immersion by the assumption that $X$ is separated and the second
morphism $(x,x',y,g)\mapsto (x, y,x',yg)$ is a base change of $F\times G\to
F\times F$. So (a) is satisfied for $X\times F$. Moreover, the inertia
groups of $G$ on $X \times F$ are conjugate in $L$ to subgroups of $S$, so
(b) is satisfied for $X \times F$.

As in \cite[6.2]{Quillen1}, consider the following commutative diagram
\begin{equation}\label{e.mtdiag}
\xymatrix{H^*([X/G],K) \ar[r] \ar[d]_{a_G(X,K)} & H^*([X \times F/G],[\pr_1/\id_G]^* K) \ar@<.5ex>[r]
\ar@<-.5ex>[r]
   \ar[d]^{a_G(X \times F, [\pr_1/\id_G]^* K)} & H^*([X \times F \times F/G],[\pr_1/\id_G]^* K)
\ar[d]^{a_G(X \times F \times F, [\pr_1/\id_G]^* K)} \\
R^*_G(X,K) \ar[r] & R^*_G(X \times F, [\pr_1/\id_G]^* K) \ar@<.5ex>[r] \ar@<-.5ex>[r] & R^*_G(X \times F \times F, [\pr_1/\id_G]^*K),}
\end{equation}
in which the double horizontal arrows are defined by $\pr_{12}$ and
$\pr_{13}$. By Corollary \ref{c.quotl}, $a_G(X \times F, [\pr_1/\id_G]^*K)$
and $a_G(X \times F \times F, [\pr_1/\id_G]^*K)$ have nilpotent kernels and,
if $K$ is commutative, are uniform $F$-surjections. To show that $a_G(X,K)$
has the same properties it thus suffices to show that the rows of
\eqref{e.mtdiag} are exact.

First consider the lower row. The component of degree $q$ is isomorphic by
definition \eqref{e.RE} to the projective limit over $(A,A',g)\in
\cA_G(k)^\natural$ of
\begin{multline}\label{e.mtseq}
  \Gamma(X^{A'},R^q\pi_* r^* K) \to \Gamma(X^{A'}\times F^{A'},R^q\pi_* r^* [\pr_1/\id_G]^* K) \\
  \rightrightarrows \Gamma(X^{A'}\times F^{A'}\times F^{A'},R^q\pi_* r^* [\pr_1/\id_G]^* K),
\end{multline}
where we have put $r \coloneqq  [1/c_g]$. In order to identify the second
and third terms of \eqref{e.mtseq}, consider the following commutative
diagram, where the middle and right squares are cartesian:
\[
\xymatrix{[X\times F/G] \ar[d]_{[\mathrm{pr}_1/\id_G]} & BA \times X^{A'} \times F^{A'} \ar[d]^{\id \times \mathrm{pr_1}} \ar[l]_-r \ar[r]^-{\pi} & X^{A'} \times F^{A'} \ar[d]^{\mathrm{pr}_1} \ar[r]^-{\mathrm{pr}_2} & F^{A'} \ar[d] \\
[X/G] & BA \times X^{A'} \ar[l]_r \ar[r]^{\pi} & X^{A'} \ar[r] & \Spec k}.
\]
We have (by base change for the middle square)
\[
\mathrm{pr}_1^*R^q\pi_*(r^*K)\simto R^q\pi_*(\id \times \mathrm{pr}_1)^*r^*K\simeq R^q\pi_*r^*[\pr_1/\id_G]^*K .
\]
By the K\"unneth formula for the right square, we have
\[
\Gamma(X^{A'} \times F^{A'}, \mathrm{pr}_1^*R^q\pi_*r^*K) \simto \Gamma(X^{A'},R^q\pi_*r^*K) \otimes \Gamma(F^{A'},\F_{\ell}).
\]
Therefore we get a canonical isomorphism
\[
\Gamma(X^{A'} \times F^{A'},R^q\pi_*r^*[\pr_1/\id_G]^*K) \simto \Gamma(X^{A'},R^q\pi_*r^*K) \otimes \Gamma(F^{A'},\F_{\ell}).
\]
We have a similar identification for $X^{A'} \times F^{A'} \times F^{A'}$,
and these identifications produce an isomorphism between \eqref{e.mtseq} and
the tensor product of $\Gamma(X^{A'},R^q\pi_*r^*K)$ with
\begin{equation}\label{e.mtseq2}
  \Gamma(\Spec k,\F_\ell)\to \Gamma(F^{A'},\F_\ell) \rightrightarrows \Gamma(F^{A'}\times F^{A'},\F_\ell).
\end{equation}
As $A'$ is an elementary abelian $\ell$-subgroup of $G$, $A'$ is conjugate
in $L$ to a subgroup of~$S$, hence $F^{A'} \ne \emptyset$. It follows that
\eqref{e.mtseq2}, \eqref{e.mtseq} and hence the lower row of
\eqref{e.mtdiag} are exact.

In order to prove the exactness of the upper row of \eqref{e.mtdiag},
consider the square of Artin stacks with representable morphisms,
\begin{equation}\label{e.mtsq}
\xymatrix{[(Y \times F)/G] \ar[r] \ar[d] & [Y/G] \ar[d] \\
[F/L] \ar[r] & BL,}
\end{equation}
where $Y$ is an algebraic space of finite presentation over $k$ endowed with
an action of $G$, the horizontal morphisms are induced by projection from
$F$ and the vertical morphisms are induced by the embedding $G\to L$. The
square is 2-cartesian by Proposition \ref{p.cart} and $BS\simeq[(S\backslash
L)/L]=[F/L]$. By Propositions \ref{p.Kunneth}, \ref{p.surj} and
\ref{p.absurj}, $H^*([F/L])$ is a finitely generated free $H^*(BL)$-module
and the homomorphism
\[
H^*([Y/G],K) \otimes_{H^*(BL)} H^*([F/L]) \rightarrow H^*([Y \times F/G],[\pr_1/\id_G]^* K)
\]
defined by \eqref{e.mtsq} is an isomorphism. Applying the above to $Y=X$ and
$Y=X\times F$, we obtain an identification of the upper row of
\eqref{e.mtdiag} with the sequence
\begin{multline*}
H^*([Y/G],K)\to H^*([Y/G],K) \otimes_{H^*(BL)} H^*([F/L])\\
\rightrightarrows H^*([Y/G],K) \otimes_{H^*(BL)} H^*([F/L])\otimes_{H^*(BL)} H^*([F/L]),
\end{multline*}
which is exact by the usual argument of faithfully flat descent.
\end{proof}

\begin{cor}\label{c.bound}
Let $\cX=[X/G]$ be a global quotient stack, and assume that either (a) $G$
is embedded in $L=\GL_n$, $n \ge 1$, or (b) $G$ is an abelian variety. Let
$K\in D^+_c([X/G],\F_\ell)$ be a pseudo-ring. Let $d = \dim X$. In case (a),
let $e=\dim L/G$, $f=2\dim L-\dim G$. In case (b), let $e=0$, $f=\dim G$.
Then
\begin{enumerate}[(i)]
\item $(\Ker a_G(X,K))^m=0$, where $m = 2d + 2e+1$,

\item for $K$ commutative and $y \in R^*_G(X,K)$, we have $y^{\ell^N} \in
    \Img a_G(X,K)$ for $N \ge \max\{2d +2e-1,0\} + \log_{\ell}(2d+2f+1)$.
\end{enumerate}
\end{cor}

\begin{proof}
As in the proof of Proposition \ref{p.quotstr}, let $F = S\backslash L$. We
have $\cd_{\ell}((X \times F)/G) \le 2\dim ((X\times F)/G)=2d + 2e$. As all
inertia groups of $G$ acting on $X \times F$ are elementary abelian
$\ell$-groups, by Corollary \ref{c.quotl} we have $(\Ker a_G(X\times F,
\pr_1^* K))^m=0$, hence (i) by \eqref{e.mtdiag}. For (ii), set $a_G(X,K) =
a_0$, $a_G(X \times F,\pr_1^*K) = a_1$, $a_G(X \times F \times F,\pr_1^*K) =
a_2$. Denote by $u_0 \colon H^*([X/G],K) \rightarrow H^*([X \times
F/G],[\pr_1/\id_G]^* K)$ (resp.\ $v_0 \colon R^*_G(X,K) \rightarrow R^*_G(X
\times F, [\pr_1/\id_G]^* K)$) the left horizontal map in \eqref{e.mtdiag},
and $u_1 = d_0 - d_1 \colon H^*([X \times F/G],[\pr_1/\id_G]^* K)
\rightarrow H^*([X \times F \times F/G], [\pr_1/\id_G]^* K)$ (resp.\ $v_1 =
d_0 - d_1 \colon R^*_G(X \times F, [\pr_1/\id_G]^* K) \rightarrow R^*_G(X
\times F \times F, \pr_1^* K)$), the map deduced from the double map
$(d_0,d_1)$ in \eqref{e.mtdiag}. As $d_0$ and $d_1$ are compatible with
raising to the $\ell$-th power, so is $u_1$ (resp.\ $v_1$). Let $N_1 =
\max\{2d +2e-1,0\}$. By Corollary \ref{c.quotl} we have $v_0(y)^{\ell^{N_1}}
= a_1(x_1)$ for some $x_1 \in H^*([X \times F/G], [\pr_1/\id_G]^* K)$. By
\eqref{e.mtdiag} we have $a_2u_1(x_1) = v_1a_1(x_1)=0$. Let $h$ be the least
integer $\ge \log_{\ell}(2d  +2f+1)$. As above we have $\cd_{\ell}((X \times
F \times F)/G) \le 2d +2f$, so by Corollary  \ref{c.quotl} we get
$u_1(x_1)^{\ell^{h}} = 0$, hence by \eqref{e.mtdiag} $x_1^{\ell^{h}} =
u_0(x_0)$ for some $x_0 \in H^*([X/G],K)$, and finally $y^{\ell^{N_1 +h}} =
a_0(x_0)$.
\end{proof}

\begin{remark}\label{r.bound}\leavevmode
\begin{enumerate}
\item If in case (a) of Corollary \ref{c.bound}, we assume moreover that
    $X$ is affine, then $\cd_\ell((X\times F)/G)\le d+e$ and
    $\cd_\ell((X\times F\times F)/G)\le d+f$ by the affine Lefschetz
    theorem \cite[XIV Corollaire 3.2]{SGA4}. Thus in this case (i) holds
    for $m=d+e+1$ and (ii) holds for $N \ge \max\{d +e-1,0\} +
    \log_{\ell}(d+f+1)$.
\item Let $f\colon \cY\to \cX$ be a finite \'etale morphism of Artin
    stacks of constant degree $d$. As the composite
    $H^*(\cX,K)\xrightarrow{f^*}H^*(\cY,f^*K)\xrightarrow{\tr_{f,K}}
    H^*(\cX,K)$ is multiplication by $d$, $f^*$ is injective if $d$ is
    prime to $\ell$. Thus, in this case, if $\Ker a_{\cY,f^*K}$ is a
    nilpotent ideal, then $\Ker a_{\cX,K}$ is a nilpotent ideal with the
    same bound for the exponent. This applies in particular to the
    morphism $[X/H]\to [X/G]$, where $H<G$ is an open subgroup of index
    prime to $\ell$.
\end{enumerate}
\end{remark}

\begin{prop}\label{s.c2cont}
Theorem \ref{p.str} (b) holds if $\cX$ is a Deligne-Mumford stack of finite
inertia. More precisely, if $c=\cd_\ell(Y)$, where $Y$ is the coarse moduli
space of $\cX$, and if $r$ (resp.\ $s$) is the maximal number of elements of
the inertia groups (resp.\ $\ell$-Sylow subgroups of the inertia groups) of
$\cX$, then $(\Ker a(\cX,K))^{(c+1)((s-1)^2+1)}=0$, and for $K$ commutative
and $b\in R^*(\cX,K)$, we have $b^{\ell^N}\in \Img a(\cX,K)$ for $N\ge
\max\{c-1,0\}+\max\{r^2-2r,0\}+\lceil \log_\ell(2(r-1)^2+1)\rceil +\lceil
\log_\ell((s-1)^2+1)\rceil$. Here $\lceil x \rceil$ for a real number $x$
denotes the least integer $\ge x$.
\end{prop}

\begin{proof}
Consider the coarse moduli space morphism $f\colon \cX\to Y$. For every
geometric point $t$ of $Y$, there exists a finite radicial extension $t' \to
t$ and a geometric point $y'$ of $\mathcal{X}$ above $t'$ such that
$(\mathcal{X}_{t'})_\mathrm{red} \simeq B\mathrm{Aut}_{\mathcal{X}}(y')$.
Note that for any field $E$, a finite group $G$ of order $m$ can be embedded
into $\GL_{m}(E)$, given for example by the regular representation $E[G]$ of
$G$. Moreover, if $m\neq 2$ or the characteristic of $E$ is not $2$, then
$G$ can be embedded into $\GL_{m-1}(E)$, because the subrepresentation of
$E[G]$ generated by $g-h$, where $g,h\in G$, is faithful. Thus, by Remark
\ref{r.bound}, the map $a_{\cX_t,K}$ in Proposition \ref{p.10.4} (c)
satisfies $(\Ker a_{\cX_t,K})^{(s-1)^2+1}=0$, and, for $K$ commutative,
$a_{\cX_t,K}$ is a uniform $F$-surjection for all geometric points $t\to Y$
with bound for the exponent given by $\max\{r^2-2r,0\}+\lceil
\log_\ell(2(r-1)^2+1) \rceil$, independent of $t$. Thus $(\Ker
\varprojlim_{t\in P_Y} a_{\cX_t,K})^{(s-1)^2+1}=0$, and Lemma \ref{p.limit}
below implies that $\varprojlim_{t\in P_Y} a_{\cX_t,K}$ is a uniform
$F$-surjection, with bound for the exponent given by
$\max\{r^2-2r,0\}+\lceil \log_\ell(2(r-1)^2+1) \rceil +\lceil
\log_\ell((s-1)^2+1) \rceil$. Hence, by Lemma \ref{l.10.2} and Proposition
\ref{p.10.4}, $a_{\cX,K}$ has the stated properties.
\end{proof}

\begin{lemma}\label{p.limit}
Let $\cC$ be a category, and let $u\colon R\to S$ be a homomorphism of
pseudo-rings in $\grvec^\cC$. If $u$ is a uniform $F$-injection (resp.\
uniform $F$-isomorphism) (Definition \ref{s.grvec}), then $\varprojlim_\cC
u$ is also a uniform $F$-injection (resp.\ uniform $F$-isomorphism). More
precisely, if $m\ge 0$ is an integer such that for every object $i$ of $\cC$
and every $a\in \Ker u_i$, $a^m=0$ (resp.\ and if $n\ge 0$ is an integer
such that for every object $i$ of $\cC$ and every $b\in S_i$, $b^{\ell^n}\in
\Img u_i$), then for every $x\in \Ker \varprojlim_\cC u$, $x^m=0$ (resp.\
for every $y\in \varprojlim_\cC S$ and every integer $N\ge n+\log_\ell(m)$,
$y^N\in \Img \varprojlim_\cC u$).
\end{lemma}

\begin{proof}
Let $x =(x_i)$ be an element in the kernel of $\varprojlim_\cC u$. Since
$x_i$ is in $\Ker u_i$, $x^{m}=(x_i^{m})=0$.

Assume now that $u$ is a uniform $F$-isomorphism with bounds for the
exponents given by $m$ and $n$, and let $y=(y_i)$ be an element of
$\varprojlim_\cC S$. For every object $i$ of $\cC$, take $a_i$ in $R_i$ such
that $u_i(a_i)=y_i^{\ell^n}$. For every morphism $\alpha\colon i\to j$
in~$\cC$, the following diagram commutes
\[\xymatrix{R_j\ar[r]^{u_j}\ar[d]_{R_\alpha} &S_j\ar[d]^{S_\alpha}\\
R_i\ar[r]^{u_i} & S_i.}\]
It follows that
\[u_i(R_\alpha(a_j)-a_i)=S_\alpha(u_j(a_j))-u_i(a_i)=S_\alpha(y_j^{\ell^n})-y_i^{\ell^n}=0.\]
Let $h$ be the least integer $\ge \log_\ell(m)$. Then
$0=(R_\alpha(a_j)-a_i)^{\ell^h} = R_\alpha(a_j)^{\ell^h} - a_i^{\ell^h}$, so
that $w=(a_i^{\ell^h})$ is an element of $\varprojlim_\cC R$. By definition,
$u(w)=y^{\ell^{n+h}}$.
\end{proof}

In order to deal with the general case, we need the following lemma.

\begin{lemma}\label{p.shr}
Let $u\colon R\to S$ be a homomorphism of pseudo-rings  in $\grvec^\cC$
endowed with a splitting (Definition \ref{r.ringhom}). Then $(\Ker u)R=0$.
In particular, $(\Ker u)^2=0$.
\end{lemma}

\begin{proof}
Let $a\in \Ker u$, $b\in R$.  Since $u(a)=0$, $ab=u(a)b=0$.
\end{proof}

\begin{prop}\label{p.strb1}
The first assertion of Theorem \ref{p.str} (b) holds.
\end{prop}

\begin{proof}
If $i\colon\cY\to \cX$ is a closed immersion, $j\colon \cU\to\cX$ is the
complement, then the following diagram of graded rings commutes:
  \[\xymatrix{H^*(\cY,Ri^! K)\ar[r]\ar[d]_{a_{\cY,Ri^! K}} & H^*(\cX,K)\ar[d]^{a_{\cX,K}}\ar[r] & H^*(\cU,j^*K)\ar[d]^{a_{\cU,j^*K}}\\
  R^*(\cY,Ri^! K)\ar[r]^u & R^*(\cX,K) \ar[r] & R^*(\cU,K).}\]
The first row is exact and $u$ is the composition of the inverse of the
isomorphism $R^*(\cX,i_* Ri^! K)\simto R^*(\cY, Ri^! K)$ and the map
$R^*(\cX,i_* Ri^! K) \to R^*(\cX,K)$ induced by adjunction $i_*Ri^!K\to K$.
The composition
  \[R^*(\cY,Ri^! K)\xto{u}  R^*(\cX,K)\to R^*(\cY,i^* K)\]
is induced by $Ri^! K \to i^* K$, hence has square-zero kernel by Lemma
\ref{p.shr}. Thus $(\Ker u)^2=0$. It follows that if both $a_{\cY,Ri^!K}$
and $a_{\cU,j^*K}$ have nilpotent kernels, then $a_{\cX,K}$ has nilpotent
kernel. Using this, we reduce by induction to the global quotient case. In
this case, the assertion follows from Propositions \ref{p.quotgen} and
\ref{p.quotstr}.
\end{proof}

This finishes the proof of the structure theorem (Theorem \ref{p.str} (b)).

\begin{lemma}\label{l.fgalg}
Let $\cC$ be a category having finitely many isomorphism classes of objects.
Let $\cA$ be the category whose objects are the elementary abelian
$\ell$-groups and whose morphisms are the monomorphisms. Let $F\colon \cC\to
\cA$ be a functor. Let $\cF$ be the presheaf of $\F_\ell$-algebras on $\cA$
given by $\cF(A)=\Sym(A\spcheck)$. Let $\cG$ be a presheaf of
$F^*\cF$-modules on $\cC$. Assume that, for every object $x$ of $\cC$,
$\cG(x)$ is a finitely generated $\cF(F(x))$-module. Then
$R=\varprojlim_{x\in \cC}\cF(F(x))$ is a finitely generated
$\F_\ell$-algebra and $S=\varprojlim_{x\in \cC}\cG(x)$ is a finitely
generated $R$-module.
\end{lemma}

\begin{proof}
We may assume that $\cC$ has finitely many objects. For any monomorphism
$u\colon A\to B$ of elementary abelian $\ell$-groups, $\cF(u)\colon
\cF(B)\to \cF(A)$ carries $\Sym(B\spcheck)^{\GL(B)}$ into
$\Sym(A\spcheck)^{\GL(A)}$. Thus
$A\mapsto\cE(A)=\Sym(A\spcheck)^{\GL(A)}\subset \cF(A)$ defines a
subpresheaf $\cE$ of $\F_\ell$-algebras of $\cF$. As $\GL(A)$ is a finite
group, by \cite[V Corollaire 1.5]{SGA1} $\cF(A)$ is finite over $\cE(A)$ and
$\cE(A)$ is a finitely generated $\F_\ell$-algebra. For given $A$ and $B$,
since $\GL(B)$ acts transitively on the set of monomorphisms $u\colon A\to
B$, the map $\Sym(B\spcheck)^{\GL(B)}\to \Sym(A\spcheck)$, restriction of
$\cF(u)$, does not depend on $u$. Thus $\cE(u)$ only depends on $A$ and $B$.
Therefore, via the functor $\mathrm{rk}\colon \cA\to \N$ carrying $A$ to its
rank, $\cE$ factorizes through a presheaf $\cR$ on the totally ordered set
$\N$: we have a 2-commutative diagram
\[\xymatrix{\cC\ar[r]^F \ar[d]_f&\cA\ar[d]^{\cE}\ar[dl]_{\mathrm{rk}}\\
\N\ar[r]^{\cR} & \cB^{\op},}
\]
where $\cB$ denotes the category of $\F_\ell$-algebras of finite type, and
$\cR(n)=\Sym((\F_\ell^n)\spcheck)^{\GL_n(\F_\ell)}$, with, for $m\le n$,
$\F_\ell^m$ included in $\F_\ell^n$ by any monomorphism. For a morphism $u
\colon A \to B$ of $\cA$, $\cF(u)\colon \cF(B)\to \cF(A)$ is surjective,
hence, as $\cF(B)$ is finite over $\cE(B)$, $\cF(A)$ is finite over
$\cE(B)$, and $\cE(A)\subset \cF(A)$ is finite over $\cE(B)$. By Lemma
\ref{l.tree} below, for each $x$ in $\cC$, $\cE(F(x))$ is finite over
\[Q=\varprojlim_{y\in \cC} \cE(F(y))\simeq \varprojlim_{y\in \cC} \cR(f(y)).\]
The rest of the proof is similar to the proof of the last assertion of
\emph{loc.\ cit.} As $\cC$ has finitely many objects, there exists a
finitely generated $\F_\ell$-subalgebra $Q_0$ of $Q$ such that, for each $x$
in $\cC$, $\cE(F(x))$ is integral, hence finite over $Q_0$. Note that $R$ is
a $Q$-submodule, \emph{a fortiori} a $Q_0$-submodule, of $\prod_{x\in
\cC}\cF(F(x))$. For each $x$ in $\cC$, $\cF(F(x))$ is finite over
$\cE(F(x))$, hence finite over $Q_0$. It follows that $\prod_{x\in
\cC}\cF(F(x))$ is finite over $Q_0$. As $Q_0$ is a noetherian ring, $R$ is
finite over $Q_0$, hence a finitely generated $\F_\ell$-algebra. Similarly,
$S$ is a finitely generated $Q_0$-module, hence a finitely generated
$R$-module. Note that $Q$ is also finite over $Q_0$, hence a finitely
generated $\F_\ell$-algebra, though we do not need this fact.
\end{proof}

The first step of the proof of Lemma \ref{l.tree} consists of simplifying
the limit $Q$ using cofinality. Among the functors
\[f_1\colon
\begin{xy}
(11,0)*{\bullet}
\ar(-4,0)*{(\bullet};(2,0)*{\bullet)}
\ar(5,0);(9,0)
\end{xy}\qquad\qquad f_2\colon
\begin{xy}
\ar@<0.5mm> (0,-2)*{\bullet}="a";(0,2)*{\bullet}="b"
\ar@<-0.5mm> "a";"b"
\ar (3,0);(7,0)
\ar (9,-2)*{\bullet};(9,2)*{\bullet}
\end{xy}\qquad\qquad f_3\colon
\begin{xy}
\ar (0,-2)*{\bullet};(2,2)*{\bullet}="top"
\ar (4,-2)*{\bullet};"top"
\ar (6,0);(10,0)
\ar (12,-2)*{\bullet};(12,2)*{\bullet}
\end{xy}\qquad\qquad f_4\colon
\begin{xy}
\ar (2,-2)*{\bullet}="bottom";(0,2)*{\bullet}
\ar "bottom";(4,2)*{\bullet}
\ar (6,0);(10,0)
\ar (12,-2)*{\bullet};(12,2)*{\bullet}
\end{xy}
\]
$f_1$, $f_2$, and $f_3$ are cofinal, while $f_4$ is not cofinal. It turns
out that after making contractions of types $f_1$, $f_2$, and $f_3$, we
obtain a rooted forest, of which the source of $f_4$ is a prototype.

For convenience we adopt the following order-theoretic definitions. We
define a \emph{rooted forest} to be a partially ordered set $\cP$ such that
$\cP_{\le x}=\{y\in \cP\mid y\le x\}$ is a finite chain for all $x\in \cP$.
We define a \emph{rooted tree} to be a nonempty connected rooted forest. Let
$\cP$ be a rooted tree. For $x,y\in \cP$, we say that $y$ is a \emph{child}
of $x$ if $x<y$ and there exists no $z\in \cP$ such that $x<z<y$. By the
connectedness of $\cP$, $m(x)=\min\cP_{\le x}$ is independent of $x\in \cP$,
hence $\cP$ has a least element $r$, equal to $m(x)$ for all $x$. We call
$r$ the \emph{root} of $\cP$.

\begin{remark}
Although we do not need it, let us recall the comparison with
graph-theoretic definitions. A \emph{graph-theoretic rooted tree} $\cT$ is a
connected acyclic (undirected) graph with one vertex designated as the root
\cite[page~30]{SerreArbres}. For a graph-theoretic rooted tree $\cT$, we let
$V(\cT)$ denote the set of vertices of $\cT$ equipped with the tree-order,
with $x\le y$ if and only if the unique path from the root $r$ to $y$ passes
through $x$. For any $x\in V(\cT)$, $V(\cT)_{\le x}$ consists of vertices on
the path from $r$ to $x$, so that $V(\cT)_{\le x}$ is a finite chain. Thus
$V(\cT)$ is a rooted tree. Conversely, for any rooted tree $\cP$, we
construct a graph-theoretic rooted tree $\Gamma(\cP)$ as follows. Let $G$ be
the graph whose set of vertices is $\cP$ and such that two vertices $x$ and
$y$ are adjacent if and only if $y$ is a child of $x$ or $x$ is a child of
$y$. Note that each $x\le x'$ in $\cP$ can be decomposed into a sequence
$x=x_0<x_1<\dots<x_n=x'$, $n\ge 0$, each $x_{i+1}$ being a child of $x_{i}$,
which defines a path from $x$ to $x'$ in $G$. Thus the connectedness of
$\cP$ implies the connectedness of $G$. If $G$ admits a cycle, then there
exists $y\in \cP$ that is a child of distinct elements $x$ and $x'$ of
$\cP$, which contradicts the assumption that $\cP_{\le y}$ is a chain. Let
$r$ be the root of $\cP$. Then $\Gamma(\cP)=(G,r)$ is a graph-theoretic
rooted tree. We have $\cP=V(\Gamma(\cP))$ and $\cT=\Gamma(V(\cT))$.
\end{remark}

The next lemma is probably standard but we could not find an adequate
reference.

\begin{lemma}\label{l.pretree}
Let $\cC$ be a category and let $f \colon \cC \to \N$ be a functor. Let
$\cP$ be the set of full subcategories of $\cC$ that are connected
components of $f^{-1}(\N_{\ge n})$ for some $n \in \N$. Order $\cP$ by
inverse inclusion: for elements $S$ and $T$ of $\cP$, we write $S \le T$ if
$S \supset T$. Let $\psi\colon \cC \to \cP$ be the functor carrying an
object $x$ to the connected component $\psi(x)$ of $f^{-1}(\N_{\ge f(x)})$
containing $x$, and let $\phi\colon \cP \to \N$ be the functor carrying $S$
to $\min f(S)$. Then:
\begin{enumerate}
\item $f=\phi \psi$.
\item $\psi \colon \cC\to \cP$ is cofinal (Definition \ref{s.cofinal}) and
    $\phi\colon \cP\to \N$ is strictly increasing.
\item $\cP$ is a rooted forest. Moreover, if $\cC$ has finitely many
    isomorphism classes of objects, then $\cP$ is a finite set.
\end{enumerate}
\end{lemma}

\begin{proof}
(a) Let $x$ be an object of $\cC$. As $x\in\psi(x)$, $\phi(\psi(x))=\min
f(\psi(x))\le f(x)$. Conversely, as $\psi(x)\subset f^{-1}(\N_{\ge f(x)})$,
$f(\psi(x))\subset \N_{\ge f(x)}$, so that $\phi(\psi(x))\ge f(x)$. Thus
$\phi(\psi(x))=f(x)$.

(b) Let $S\in\cP$. Note that $S$ is a connected component of $f^{-1}(\N_{\ge
\phi(S)})$. By definition, $(S\downarrow \psi)$ is the category of pairs
$(x,S\le \psi(x))$. Note that $S\supset \psi (x)$ implies that $x$ is in
$S$. Conversely, for $x$ in $S$, $S$ is a connected component of
$f^{-1}(\N_{\ge n})$ for $n\le f(x)$, hence $S\supset \psi(x)$. Thus
$(S\downarrow \psi)$ can be identified with $S$, hence is connected. This
shows that $\psi$ is cofinal. Now let $S<T$ be elements of $\cP$. We have
$\phi(S)\le \phi(T)$. If $\phi(S)=\phi(T)=n$, then $S$ and $T$ are both
connected components of $f^{-1}(\N_{\ge n})$, which contradicts with the
assumption $S\supsetneq T$. Thus $\phi(S)<\phi(T)$.

(c) Let $S\in \cP$. Let $T,T'\in \cP_{\le S}$. Then $T$ (resp.\ $T'$) is a
connected components of $f^{-1}(\N_{\ge n})$ (resp.\ $f^{-1}(\N_{\ge n'})$),
and $T$ and $T'$ both contain $S$. Thus $T\supset T'$ if $n\le n'$ and
$T\subset T'$ if $n\ge n'$. Therefore, $\cP_{\le S}$ is a chain. As $\phi$
is strictly increasing, $\phi$ induces an injection $\cP_{\le S}\to \N_{\le
\phi(S)}$, hence $\cP_{\le S}$ is a finite set. Therefore, $\cP$ is a rooted
forest. Note that for $S\in \cP$ and $x$ in $S$, every object $y$ of $\cC$
isomorphic to $x$ is also in $S$. Thus, if $\cC$ has finitely many
isomorphism classes of objects, then $\cP$ is a finite set.
\end{proof}

\begin{lemma}\label{l.tree}
Let $\cC$ be a category having finitely many isomorphism classes of objects
and let $f\colon \cC\to \N$ be a functor. Let $\cR$ be a presheaf of
commutative rings on $\N$ such that, for each $m\le n$, $\cR(m)$ is finite
over $\cR(n)$. Let $Q=\varprojlim_{x\in \cC} \cR(f(x))$. Then:
\begin{enumerate}
\item For each object $x$ of $\cC$, $\cR(f(x))$ is finite over $Q$.
\item For each connected component $S$ of $\cC$ and each $r$ in $S$
    satisfying $f(r)=\min f(S)$, we have
    \[\Img(Q\to\cR(f(r)))=\Img(\cR(\max f(S))\to \cR(f(r))).\]
\end{enumerate}
\end{lemma}

\begin{proof}
By Lemma \ref{l.pretree}, we may assume that $\cC$ is a finite rooted tree
with root $r$. We prove this case by induction on $\#\cC$. Let $B\subset
\cC$ be the set of children of $r$.  For each $c\in B$, $\cC_{\ge c}$ is a
rooted tree with root $c$ and $Q$ is the fiber product over $\cR(f(r))$ of
the rings $Q_c=\varprojlim_{x\in \cC_{\ge c}} \cR(f(x))$ for $c\in B$. If
$B$ is empty, then $\cC=\{r\}$ and the assertions are trivial. If $B=\{c\}$,
then $Q\simeq Q_c$ and it suffices to apply the induction hypothesis to
$Q_c$. Assume $\#B>1$. Let $n=\max f(\cC)$, $n_c=\max f(\cC_{\ge c})$, and
let $c_0\in B$ be such that $n_{c_0}=\min_{c\in B} n_c$. The complement
$\cC'$ of $\cC_{\ge c_0}$ in $\cC$ is a rooted tree with root $r$, and $Q$
is the fiber product over $\cR(f(r))$ of the rings $Q_{c_0}$ and
$Q'=\varprojlim_{x\in \cC'}\cR(f(x))$. By the induction hypothesis,
$A=\Img(Q_{c_0}\to \cR(f(r)))=\Img(\cR(n_{c_0})\to \cR(f(r)))$ and
$\Img(Q'\to \cR(f(r)))=\Img(\cR(n)\to \cR(f(r)))$, so that we have a
cartesian square of commutative rings
\[\xymatrix{Q\ar[r]^{\alpha'}\ar[d]_{\beta'} &Q'\ar[d]^{\beta}\\
Q_{c_0}\ar[r]^{\alpha} & A.}
\]
As $\alpha$ is surjective, we have $\Ker(\alpha')\simeq \Ker(\alpha)$ and
$\alpha'$ is surjective (cf.\ \cite[Lemme 1.3]{Ferrand}), which implies (b).
Moreover, as $\beta$ is finite, $\beta'$ is finite. Indeed, if $A=\sum_i
a_i\beta(Q')$, then for liftings $a'_i$ of $a_i$, $Q_{c_0}=\sum_i
a'_i\beta'(Q)$. The assertion (a) then follows from the induction hypothesis
applied to $Q_{c_0}$ and $Q'$.
\end{proof}

\begin{proof}[Proof of Theorem \ref{p.str} (a)]
Let $(j_i\colon \cX_i\to \cX)_i$ be a finite stratification of $\cX$ by
locally closed substacks. The system of functors $(\cC_{\cX_i}\to
\cC_{\cX})_i$ is essentially surjective. Thus the map
  \[R^*(\cX,K)\to \prod_{i} R^*(\cX_i,j_i^* K)\]
is an injection. Thus, for the first assertion of Theorem \ref{p.str} (a),
we may assume that $\cX$ is a global quotient stack, in which case the
assertion follows from Theorem \ref{t.main2} (a) and Proposition \ref{p.RR}.

Let $H^q(K_\bullet)$ denote the presheaf on $\mathcal{C}_{\mathcal{X}}$
whose value at $x\colon \cS\to \cX$ is $H^q(\mathcal{S},K_x)$, where
$K_x=x^* K$, so that $R^q(\cX,K)=\varprojlim_{\cC_\cX} H^q(K_\bullet)$. Let
$N$ be the set of morphisms $f$ in $\cC_\cX$ such that
$(H^*(\F_{\ell\bullet}))(f)$ and $(H^*(K_\bullet))(f)$ are isomorphisms. By
Lemma \ref{l.loclim}, $\varprojlim_{\cC_\cX} H^q(K_\bullet)\simeq
\varprojlim_{N^{-1}\cC_\cX}H^q(K_{\bullet})$ and similarly for
$H^*(\F_{\ell\bullet})$. We claim that $N^{-1}\cC_\cX$ has finitely many
isomorphism classes of objects. Then $\varprojlim_{N^{-1}\cC_\cX}$ commutes
with direct sums, and, by Lemma \ref{l.fgalg}, $R^*(\cX,\F_\ell)$ and
$R^*(\cX,K)$ are finitely generated $R$-modules for a finitely-generated
$\F_\ell$-algebra $R$, hence the second assertion of Theorem \ref{p.str}
(a). Using again the fact that the system of functors $(\cC_{\cX_i}\to
\cC_{\cX})_i$ is essentially surjective, we may assume in the above claim
that $\cX=[X/G]$ is a global quotient. Consider the diagram \eqref{e.8.4.2}.
Note that the functor $\cA_G(k)^\natural\to \cE_G(\pi_0)$ induces a
bijection between the sets of isomorphism classes of objects, and
$\cE_G(\pi_0)$ is essentially finite by Lemma \ref{l.6.18} (a), thus
$\cA_G(k)^\natural$ has finitely many isomorphism classes of objects.
Moreover, as $E$ is essentially surjective, it suffices to show that, for
every object $(A,A',g)$ of $\cA_G(k)^\natural$, the category
$M^{-1}P_{X^{A'}}$ has finitely many isomorphism classes of objects. Here
$M$ is the set of morphisms $f$ in $P_{X^{A'}}$ such that
$(E_{(A,A',g)}^*H^*(K_\bullet))(f)$ is an isomorphism. Let $(X_i)$ be a
finite stratification of $X^{A'}$ into locally closed subschemes such that
$K\res X_i$ has locally constant cohomology sheaves. For a given $i$, all
objects in the image of $P_{X_i}\to M^{-1}P_{X^{A'}}$ are isomorphic.
Moreover, the system of functors $(P_{X_i}\to P_{X^{A'}})_i$ is essentially
surjective. Therefore, $M^{-1}P_{X^{A'}}$ has finitely many isomorphism
classes of objects.
\end{proof}

\section{Stratification of the spectrum}\label{s.strat}
In this section we fix an algebraically closed field $k$  and a prime number
$\ell$ invertible in $k$.

\begin{construction}\label{c.amalg}
Let $X$ be a separated algebraic space of finite type over $k$, and let $G$
be an algebraic group over $k$ acting on $X$. Define
\begin{equation}
\underline{(G,X)} \coloneqq \Spec H^{\varepsilon *}([X/G])_{\mathrm{red}},
\end{equation}
where $\varepsilon = 1$ if $\ell = 2$, and $\varepsilon = 2$ otherwise. In
particular, for an elementary abelian $\ell$-group $A$,
\[
\underline{A} \coloneqq \underline{(A,\Spec k)} = \Spec (H^{\varepsilon*}_A)_{\mathrm{red}}
\]
is a standard affine space of dimension equal to the rank of $A$. The map
$(A,C)^*$ \eqref{e.Ac*} induces a morphism of schemes
\begin{equation}\label{e.Acl}
(A,C)_* \colon \underline{A} \to \underline{(G,X)},
\end{equation}
hence $a(G,X)$ \eqref{e.aGX} induces a morphism of schemes
\begin{equation}\label{e.amalg}
Y \coloneqq \varinjlim_{(A,C) \in \mathcal{A}^{\flat}_{(G,X)}} \underline{A} \to \underline{(G,X)}.
\end{equation}
It follows from Theorem \ref{c.main} that \eqref{e.amalg} is a universal
homeomorphism.

By Corollary \ref{c.finite}, $(A,C)_*$ is finite. Moreover,
$\cA_{(G,X)}^\flat$ is essentially finite by Lemma \ref{l.Afincat}. It
follows that $Y\simeq \Spec(\varprojlim_{(A,C)\in \cA^\flat_{(G,X)}}
(H^{\varepsilon*}_A)_{\red})$ is finite over $\underline{(G,X)}$ and is a
colimit of $\underline{A}$ in the category of locally ringed spaces and in
particular a colimit of $\underline{A}$ in the category of schemes. This
remark gives another proof of the second assertion of Corollary \ref{c.PS},
as promised. Moreover, the $\F_\ell$-algebras
$H^{\varepsilon*}([X/G])_{\red}$ and $\varprojlim_{(A,C)\in
\cA^\flat_{(G,X)}} (H^{\varepsilon*}_A)_{\red}$ are equipped with Steenrod
operations (see Construction \ref{r.Steenrod} below), compatible with the
ring homomorphism $H^{\varepsilon*}([X/G])_{\red}\to \varprojlim_{(A,C)\in
\cA^\flat_{(G,X)}} (H^{\varepsilon*}_A)_{\red}$.

The structure of $Y$ is described more precisely by the following
\textit{stratification theorem}, similar to \cite[Theorems 10.2,
12.1]{Quillen2}.
\end{construction}

\begin{theorem}\label{t.strat}
Denote by $V_{(A,C)}$ the reduced subscheme of $Y$ that is the image of the
(finite) morphism $\underline{A}\to Y$ induced by $(A,C)$. Let
\begin{gather*}
\underline{A}^+ \coloneqq \underline{A} - \bigcup_{A' < A} \underline{A}',\\
V_{(A,C)}^+ \coloneqq V_{(A,C)} - \bigcup_{A' < A}V_{(A',C\res A')},
\end{gather*}
where $A' < A$ means $A' \subset A$ and $A' \ne A$, and $C\res A'$ denotes
the component of $X^{A'}$ containing~$C$. Then
\begin{enumerate}
\item The Weyl group $W_G(A,C)$ \eqref{e.Weyl} acts freely on
    $\underline{A}^+$ and the morphism $\underline{A}^+ \to Y$ given by
    $(A,C)$ induces a homeomorphism
\begin{equation}\label{e.strat}
\underline{A}^+/W_G(A,C) \to V_{(A,C)}^+.
\end{equation}

\item The subschemes $V_{(A,C)}$ of $Y$ are the integral closed subcones
    of $Y$ that are stable under the Steenrod operations on
    $\varprojlim_{(A,C)\in \cA^\flat_{(G,X)}}
    (H^{\varepsilon*}_A)_{\red}$.

\item Let $(A_i,C_i)_{i \in I}$ be a finite set of representatives of
    isomorphism classes of objects of $\mathcal{A}^{\flat}_{(G,X)}$. Then
    the $V_{(A_i,C_i)}$ form a finite stratification of $Y$, namely $Y$ is
    the disjoint union of the $V_{(A_i,C_i)}^+$, and $V_{(A_i,C_i)}$ is
    the closure of $V_{(A_i,C_i)}^+$.
\end{enumerate}
\end{theorem}

The proof is entirely analogous to that of \cite[Theorems 10.2,
12.1]{Quillen2}. One key step in the proof is the following analogue of
\cite[Proposition 9.6]{Quillen2}.

\begin{prop}
Let $(A,C)$, $(A',C')$ be objects of $\cA_{(G,X)}^\flat$. The square of
topological spaces
\begin{equation}\label{e.stratsq}
\xymatrix{\underline{A}^+\times \Hom_{\cA_{G,X}^\flat}((A,C),(A',C'))\ar[d]_{\pr_1}\ar[r] & \underline{A'}\ar[d]^{(A',C')}\\
\underline{A}^+\ar[r]^{(A,C)} & Y}
\end{equation}
is cartesian. Here the upper horizontal arrow is induced by
\[\Hom_{\cA_{G,X}^\flat}((A,C),(A',C'))\to
\Hom(\underline{A},\underline{A'}),\quad
u\mapsto \Spec (Bu)^*_{\red},
\]
where $(Bu)^*\colon H^{\varepsilon*}_{A'}\to H^{\varepsilon*}_{A}$.
\end{prop}

As in \cite[Proposition 9.6]{Quillen2}, this follows from the fact that
$\cA_{(G,X)}^\flat$ admits fiber products, whose proof is very similar to
that of \cite[Lemma 9.1]{Quillen2}.

\begin{remark}
The morphism \eqref{e.strat} is not an isomorphism of schemes in general. In
particular, the square \eqref{e.stratsq} is not cartesian in the category of
schemes. This is already shown by the example $G=\GL_\ell$, $X=\Spec(k)$,
$A=\mu_\ell$ embedded diagonally in $G$. Let $T$ be the standard maximal
torus, and let $\{e_1,\dots, e_\ell\}$ be the standard basis of $T[\ell]$.
Then $W_G(T)\simeq W_G(T[\ell])\simeq \SG_\ell$ acts on $T[\ell]$ by
permuting this basis, and $W_G(A)=\{1\}$. Note that
$\underline{T[\ell]}\simeq \Spec(\Sym(T[\ell]\spcheck))\simeq
\Spec(\F_\ell[t_1,\dots, t_\ell])$, and $V_{T[\ell]}=Y\simeq
\underline{T[\ell]}/W_G(T[\ell])$ can be identified with the spectrum of the
symmetric polynomials in $t_1,\dots, t_\ell$. As the image of the $d$-th
fundamental symmetric polynomial in $t_1,\dots,t_\ell$ under homomorphism
$\phi\colon \F_\ell[t_1,\dots,t_\ell]\to \F_\ell[t]$ carrying $t_i$ to $t$
is $0$ for $1\le d\le \ell-1$ and $t^\ell$ for $d=\ell$, the diagram of
schemes
\[\xymatrix{\underline{A}\ar[r]\ar[d] &\underline{T[\ell]}\ar[d]\\
V_A\ar[r] &V_{T[\ell]}}
\]
is given by the diagram of rings
\[\xymatrix{\F_\ell[t] &\F_\ell[t_1,\dots,t_\ell]\ar[l]_-{\phi}\\
\F_\ell[t^\ell]\ar@{^{(}->}[u] & \F_\ell[t_1,\dots,t_\ell]^{\SG_\ell}.\ar@{^{(}->}[u]\ar[l]}
\]
Therefore, \eqref{e.strat} is given by the Frobenius map on $\F_\ell[t,
t^{-1}]$.
\end{remark}

Let $(f,u)\colon (X,G)\to (Y,H)$ be an equivariant morphism with $(X,G)$ and
$(Y,H)$ as before. The induced morphism of quotient stacks $[f/u]\colon
[X/G]\to [Y/H]$ induces maps
\[[f/u]^*\colon H^{\varepsilon*}([Y/H])\to H^{\varepsilon*}([X/G]), \quad \underline{(f,u)}\colon \underline{(Y,H)}\to \underline{(X,G)}.\]
Moreover, $(f,u)$ induces a functor $\cA_{(u,f)}^\flat\colon
\cA_{(G,X)}^\flat\to \cA_{(H,Y)}^\flat$ sending $(A,C)$ to $(uA,C')$, where
$uA$ is the image of $A$ under $u$ and $C'$ is the component of $Y^{uA}$
containing $fC$, the image of $C$ under~$f$. We have the following analogue
of \cite[Proposition 10.9]{Quillen2}, with essentially the same proof.

\begin{prop}
The following conditions are equivalent.
\begin{enumerate}
\item $\cA_{(u,f)}^\flat$ is an equivalence of categories.
\item $[f/u]^*$ is a uniform $F$-isomorphism.
\item $\underline{(f,u)}$ is a universal homeomorphism.
\end{enumerate}
\end{prop}

\begin{construction}\label{r.Steenrod}
As Mich\`ele Raynaud observed in \cite[Section~4]{Raynaud}, the formalism of
\emph{Steenrod operations} \cite{Epstein} applies to the
$\F_\ell$-cohomology of any topos. Let us review the construction of the
operations in this case.

Let $X$ be a topos, let $K$ be a commutative ring in $D^+(X,\F_\ell)$, and
let $i$ be an integer. The Steenrod operations are $\F_\ell$-linear maps
\begin{equation*}
P^i \colon H^*(X,K) \to \begin{cases}
H^{*+i}(X,K) & \text{if $\ell = 2$,} \\
H^{*+2(\ell -1)i}(X,K) & \text{if $\ell > 2$.}
\end{cases}
\end{equation*}
For $\ell = 2$, $P^i$ is sometimes denoted $\mathrm{Sq}^i$.

First note that for every complex $M\in C(X,\F_\ell)$, $\SG_\ell$ acts on
$L^{\otimes \ell}$ by permutation of factors (with the usual sign rule).
This induces a (non triangulated) functor
\begin{equation}\label{e.SG}
(-)^{\otimes \ell}\colon D(X,\F_\ell)\to D([X/\SG_\ell],\F_\ell).
\end{equation}
Here we used the notation $[X/G]$ for the topos of sheaves in $X$ endowed
with an action of a finite group $G$. For a morphism $c\colon \F_\ell \to
M[q]$ in $D(X,\F_\ell)$, applying \eqref{e.SG}, we obtain a morphism
\[c^{\otimes \ell}\colon \F_\ell \simeq \F_\ell^{\otimes
\ell}\to (M[q])^{\otimes \ell}\simeq M^{\otimes \ell}\otimes S^{\otimes q}[q\ell]
\]
in $D([X/\SG_\ell],\F_\ell)$, where $S\in \Mod([X/\SG_\ell],\F_\ell)$ is the
pullback of the sheaf on $B\SG_\ell$ given by the sign representation
$\sgn\colon \SG_\ell\to \F_\ell^\times$. This defines a map
\begin{equation}\label{e.Steenrod}
H^q(X,M)\to H^{q\ell}([X/\SG_\ell],M^{\otimes \ell}\otimes S^{\otimes q}), \quad c\mapsto c^{\otimes \ell}.
\end{equation}

Now choose a cyclic subgroup $C$ of $\SG_\ell$ of order $\ell$ and a basis
$x$ of $H^1(BC,\F_\ell)\simeq \Hom(C,\F_\ell)$. Note that $\sgn\res C$ is
constant of value $1$. Consider the composite map
\begin{multline*}
D_{C,x}\colon H^q(X,K)\xrightarrow{\eqref{e.Steenrod}} H^{q\ell}([X/\SG_\ell],K^{\otimes \ell}\otimes S^{\otimes q})
\xrightarrow{\pi} H^{q\ell}([X/\SG_\ell],K\otimes S^{\otimes q})\\
\to H^{q\ell}([X/C],K)\simeq \bigoplus_{k}H^{k}(BC,\F_\ell)\otimes H^{q\ell-k}(X,K),
\end{multline*}
which turns out to be $\F_\ell$-linear. Here $\pi$ is given by
multiplication $K^{\otimes \ell}\to K$. Recall (Remark \ref{r.cohBA}) that
for $k\ge 0$, $H^k(BC,\F_\ell)=\F_\ell w_k$, where $w_k=x(\beta
x)^{(k-1)/2}$ for $k$ odd and $w_k=(\beta x)^{k/2}$ for $k$ even. We define
\[D^k_{C,x}\colon H^q (X, K) \to H^{q\ell-k}(X,K) \]
by the formula $D_{C,x}u=\sum_k w_k\otimes D^k_{C,x} u$. Let
$m=\frac{\ell-1}{2}$. We define
\[P^i_{C,x}=\begin{cases}
  (-1)^{i+m(q^2-q)/2}(m!)^{-q}D^{(q-2i)(\ell-1)}_{C,x}& \text{if $\ell> 2$, $q\ge 2i$,}\\
  D^{q-i}_{C,x}& \text{if $\ell=2$, $q\ge i$,}\\
  0& \text{otherwise.}
\end{cases}
\]

For $\ell=2$, we have $C=\SG_2$ and $x$ is unique. For $\ell>2$, $\sigma\in
\SG_\ell$ and $a\in \F_\ell^\times$, $D^{2k}_{\sigma C\sigma^{-1},a(x\circ
c_\sigma)}=\sgn(\sigma)^q a^{-k}D^{2k}_{C,x}$, where $c_\sigma\colon \sigma
C\sigma^{-1} \to C$ is the homomorphism $g\mapsto \sigma^{-1} g \sigma$.
Thus
\[P^i_{\sigma C\sigma^{-1},a(x\circ
c_\sigma)} = \sgn(\sigma)^{q} (a^m)^q P^i_{C,x},
\]
where $a^m=\pm 1$. In particular, up to a sign, $P^i_{C,x}$ is independent
of the choices of $C$ and $x$. Let $T\in \SG_\ell$ be the permutation
defined by $T(n)=n+1$ for $n\in \Z/\ell\Z$. In the following we will take
$C$ to be the subgroup generated by $T$ and take $x$ to be the dual basis of
$T$, and omit them from the indices.

For a homomorphism of commutative rings $K\to K'$, the induced homomorphism
$H^*(X,K)\to H^*(X,K')$ is compatible with Steenrod operations on $H^*(X,K)$
and $H^*(X,K')$. Moreover, for a morphism of topoi $f\colon X\to Y$,
Steenrod operations are compatible with the isomorphism $H^*(X,K)\simeq
H^*(Y,Rf_* K)$.

It is easy to check the following properties of Steenrod operations, where
we write $H^*$ for $H^*(X,K)$:
\begin{itemize}
\item For $x \in H^i$ (resp.\ $x \in H^{2i}$), $P^ix = x^{\ell}$ if $\ell
    = 2$ (resp.\ $\ell> 2$);

\item If one defines
\[
P_t \colon H^* \to H^*[t, t^{-1}]
\]
by $P_t(x) = \sum_{i \in \Z} P^i(x)t^i$, then $P_t$ is a ring homomorphism
(Cartan's formula).
\end{itemize}

In the case where $X$ has enough points and $K\in \Mod(X,\F_\ell)$, Epstein
showed the following additional properties \cite[Theorem 8.3.4]{Epstein}:
\begin{itemize}
\item $P_t\colon H^*\to H^*[t]$. In other words, $P^i=0$ for $i<0$.

\item $P^0=\id$ for $K=\F_\ell$ (this depends on the choices of $C$ and
    $x$ above).
\end{itemize}
In particular, $P_t(x) = x +x^{\ell}t$ for $x \in H^1(X,\F_\ell)$, $\ell =
2$ (resp.\ $x \in H^2(X,\F_\ell)$, $\ell > 2)$.

The above can be easily adapted to $D^+_\cart$ of Artin stacks. For a
morphism of Artin stacks $f\colon \cX\to \cY$ and  a commutative ring $K\in
D^+_{\cart}(\cX,\F_\ell) $, Steenrod operations are compatible with the
isomorphism $H^*(X,K)\simeq H^*(Y,Rf_* K)$. Therefore, for $K'\in
D^+_{\cart}(Y,\F_\ell)$, Steenrod operations are compatible with the
restriction homomorphism $H^*(Y,K')\to H^*(X,f^*K')$.
\end{construction}

For related results on the Chow rings of classifying spaces and much more,
we refer the reader to Totaro's book \cite{Totaro} and the bibliography
thereof.

\bibliographystyle{abbrv}
{\small\bibliography{quillen}}

Luc Illusie, Laboratoire de Math\'ematiques d'Orsay, Univ.\ Paris-Sud, CNRS,
Universit\'e Paris-Saclay, 91405 Orsay, France; email:
\texttt{Luc.Illusie@math.u-psud.fr}

Weizhe Zheng, Morningside Center of Mathematics, Academy of Mathematics and
Systems Science, Chinese Academy of Sciences, Beijing 100190, China; email:
\texttt{wzheng@math.ac.cn}

\end{document}